\documentclass[a4paper, 12pt, oneside, notitlepage]{amsart}
\usepackage[dvipsnames]{xcolor}
\usepackage[colorlinks=true,linkcolor=Red,citecolor=Green]{hyperref}
\usepackage{tikz-cd}
\usepackage[margin=3cm]{geometry}
\usepackage[normalem]{ulem}
\usepackage{amsmath,amssymb,amsthm,graphicx,mathrsfs,bbm,url}
\usepackage{amsthm}
\usepackage{wrapfig}
\usepackage{enumitem}
\usepackage{mathtools}
\usepackage[utf8]{inputenc}
 \usepackage[T1]{fontenc}
\usepackage[super]{nth}
\usepackage[open, openlevel=2, depth=3, atend]{bookmark}
\hypersetup{pdfstartview=XYZ}
\usepackage[font=footnotesize]{caption}
\usepackage{a4wide}
\usepackage[all]{xy}

\usepackage{epstopdf}
 
\usepackage{hyperref}

\theoremstyle{plain}
\newtheorem{theorem}{Theorem}[section]
\newtheorem*{theorem*}{Theorem}

\newtheorem{lemma}[theorem]{Lemma}
\newtheorem{proposition}[theorem]{Proposition}
\newtheorem{corollary}[theorem]{Corollary}

\theoremstyle{definition}
\newtheorem{definition}[theorem]{Definition}

\theoremstyle{remark}
\newtheorem{remark}[theorem]{Remark}

\numberwithin{equation}{section}

\newcommand{\fg}{\mathfrak{g}}

\newcommand{\C}{\mathbb{C}}
\newcommand{\R}{\mathbb{R}}

\newcommand{\Z}{\mathbb{Z}}

\newcommand{\V}{\mathbb{V}}

\newcommand{\HH}{\mathbb{H}}

\newcommand{\eps}{\varepsilon}

\newcommand{\mc}{\mathcal}

\newcommand{\dd}{\mathrm{d}}
\newcommand{\n}{\mathbf{n}}
\newcommand{\kk}{\mathbf{k}}
\newcommand{\btheta}{\boldsymbol{\theta}}
\newcommand{\comp}{\mathrm{comp}}
\newcommand{\bk}{\mathbf{k}}
\newcommand{\Lk}{\mathbf{L}^{\otimes \mathbf{k}}}

\DeclareMathOperator{\vol}{vol}

\DeclareMathOperator{\Tr}{Tr}
\DeclareMathOperator{\Op}{Op}

\DeclareMathOperator{\hol}{hol}

\newcommand{\be}{\begin{equation}}
\newcommand{\ee}{\end{equation}}

\author{Sebastián Muñoz-Thon}
\address{Universit\'e Paris-Saclay, Laboratoire de math\'ematiques d’Orsay, 91405, Orsay, France.}
\email{sebastian.munoz-thon@universite-paris-saclay.fr}

\title[Heat propagation on covers]{Heat propagation on abelian covers of principal bundles}

\begin{document}

\begin{abstract}
We study the long-time asymptotics of heat kernels on Abelian covers of compact manifolds and on Abelian covers of principal bundles. For Abelian covers, we establish three complementary asymptotic expansions: a distributional expansion describing correlations of the heat semigroup, a local pointwise expansion of the heat kernel on compact subsets, and a global expansion valid on the natural diffusive scale, revealing the Gaussian profile governing heat propagation. We then extend these results to horizontal heat kernels on principal bundles under natural holonomy and curvature assumptions. In this setting, we show that the leading asymptotics are entirely determined by the geometry of the underlying Abelian cover, while the nontrivial fiber modes decay exponentially fast due to a uniform spectral gap. The proof combines Floquet theory on Abelian covers with the Borel--Weil decomposition on principal bundles. 
\end{abstract}

\maketitle

\section{Introduction}

\subsection{Asymptotics of the heat kernel on Abelian covers} \label{subsection:results-abelian}

Let $M_{0}$ be a smooth, closed, connected manifold equipped with a Riemannian metric $g_{0}$. Let $\rho \colon \pi_{1}(M_{0}) \to \Z^{d}$ be a surjective representation. This defines a $\Z^{d}$-cover $\pi \colon M \to M_{0}$ by considering the associated bundle $M := \widetilde{M_{0}} \times _{\rho} \Z^{d}$, where $\widetilde{M_{0}}$ denotes the universal cover of $M_{0}$. Note that we have an action on $\widetilde{M}_0 \times \Z^{d}$ given by
$\tau_\kk(x,\n) = (x,\n + \kk)$, which descends to the quotient $M$. We will refer to this construction as an Abelian cover, and we refer to \S\ref{ssection:abelian-covers} for further details. We endow $M$ with the metric $g=\pi^{*}g_{0}$, and we denote by $\Delta_{M}$ the associated Laplace--Beltrami operator. Let $H_M(t,x,y)$ denote the corresponding heat kernel. To state our results, we will need the spaces $(B^{s,r}(M))_{s,r \geq 0}$. Functions in $B^{s,r}$ are locally $H^{s}$ and exhibit increasing decay at infinity in $M$ as $r$ grows. We refer to \S\ref{ssection:functional-spaces} for the proper definition.

Our first result is a dual expansion, i.e., we obtain asymptotics of correlations of the diffusion.

\begin{theorem}
\label{theorem:main1}
There exist a constant $\kappa > 0$, and distributions $\mathcal{C}_{j} \in \mathcal{D}'(M \times M)$, such that for all integers $N \geq 1$, we have in $\mathcal{D}'(M \times M)$:
\begin{equation}
    \label{equation:decay-correlation-0}
\begin{split}
t^{d/2} H_{M}(t) =  \kappa \mathbf{1}_{M} \otimes \mathbf{1}_{M} + \sum_{j =1}^{N-1} t^{-j} \mathcal{C}_{j} + \mathcal{R}_N(t)
\end{split}
\end{equation}
where the following estimates hold: for all $s > 0$, there exists a constant $C > 0$ such that
\begin{equation} \label{eq:decay-estimates}
\begin{split}
|\langle \mathcal{C}_j,f \otimes g\rangle| & \leq C\|f\|_{B^{s,2j}}\|g\|_{B^{s,2j}},\quad 1 \leq j \leq N - 1,\\
|\langle \mathcal{R}_N,f \otimes g\rangle| & \leq C \left(e^{-Ct}\|f\|_{L^{2}}\|g\|_{L^{2}} + \langle t \rangle^{-N}\|f\|_{B^{s,2N+d+1}}\|g\|_{B^{s,2N+d+1}}\right), \quad \forall t \geq 0.
\end{split}
\end{equation}
\end{theorem}

That \eqref{equation:decay-correlation-0} holds for functions in $L^{2}(M) \cap B^{s,2N+d+1}(M)$ follows from the density of $C_{\mathrm{comp}}^{\infty}(M)$ in this space, together with the bounds in \eqref{eq:decay-estimates}. The
constant $\kappa>0$ is explicit (see \eqref{equation:expression-cj}) and is related to the determinant of a covariance matrix of the first eigenvalue of an operator that appears after decomposing the Laplacian using Floquet theory. The distributions are shown to be nonzero in Lemma \ref{lemma:cj-non-zero}, and we explicitly compute how $\mathcal{C}_{1}$ acts, see \eqref{equation:c1}.

The dual expansion above captures the large-time behavior of the heat kernel in the weak sense. Our second result upgrades this to a local, pointwise expansion: it describes the asymptotics of the transition density when both endpoints remain in fixed compact subsets of the cover.

\begin{theorem}
\label{theorem:main1-local}
For every pair of compact sets $K_{x},K_{y} \subset M$, every $\ell_{x},\ell_{y} \geq 0$, and integers $N \geq 1$, there exist smooth functions $k_{j} \in C^{\infty}(K_{x}\times K_{y})$, $j=1,\ldots,N-1$ such that in $C^{\ell_{x},\ell_{y}}(K_{x}\times K_{y})$:
\[
t^{d/2}H_{M}(t,x,y)=\kappa+\sum_{j=1}^{N-1}t^{-j}k_{j}(x,y)+R_{N}(t,x,y),
\]
where the following holds: 
\[ 
\|R_{N}(t)\|_{C^{\ell_{x},\ell_{y}}(K_{x} \times K_{y})} \leq C_{K_{x},K_{y},N,\ell_{x},\ell_{y}}(e^{-ct}+\langle t\rangle^{-N}), \quad t \geq 1.
\]
\end{theorem}

While the local expansion corresponds to observing the heat kernel on bounded spatial scales, the asymptotic behavior changes when the endpoints are allowed to drift apart with time. Our final result captures the full diffusive behavior of the heat kernel. Specifically, it establishes a uniform asymptotic expansion when the displacement between the endpoints is of the natural diffusive order $|\n-\mathbf{m}|=\mathcal{O}(\sqrt{t})$, revealing the Gaussian profile that governs long-time heat propagation on the Abelian cover.

\begin{theorem}
\label{theorem:main1-global}
Let $D \subset M$ be a relatively compact open subset with smooth boundary, such that $\pi \colon D \to M_{0}$ is surjective. Then, there exists a positive definite matrix $H$, such that for every $R>0$, every integer $N \geq 1$, and every $\ell_{x},\ell_{y} \geq 0$, there are $P_{j} \in C^{\infty} (\R^{d} \times \overline{D} \times \overline{D})$, $j=1, \ldots, N-1$, polynomial in the first variable, such that for $t \geq 1$, $\n,\mathbf{m} \in \Z^{d}$, and $x_{0},y_{0} \in \overline{D}$ satisfying $|\n-\mathbf{m}| \leq R \sqrt{t}$:
\[
\begin{split}
    t^{d/2} H_{M}(t,\tau_{\n}x_{0},\tau_{\mathbf{m}}y_{0})=&\exp \left(-\frac{1}{2 t}\langle H^{-1} (\n-\mathbf{m}), \n-\mathbf{m}\rangle \right) \\ &\times 
    \left[\kappa +\sum_{j=1}^{N-1} t^{-j/2} P_{j}\left(\frac{\n-\mathbf{m}}{\sqrt{t}}, x_{0}, y_{0}\right)\right]+R_{N} (t, \n-\mathbf{m}, x_{0},y_{0}),
\end{split}
\]
where the following holds:
\[
\sup _{|\n| \leq R \sqrt{t}} \|R_{N}(t, \n, \bullet,\bullet)\|_{C^{\ell_{x},\ell_{y}}(\overline{D}\times \overline{D})} \leq C_{R, N,\ell_{x},\ell_{y}}(e^{-c t}+t^{-N / 2}).
\]
\end{theorem}

The theorem is stated for pairs $(\n,x_{0}) \in \Z^{d} \times \overline{D}$, rather than for points $x \in M$. This avoids any ambiguity caused by the possibility that a point $x \in M$ may be written in more than one way (i.e., it has more than one preimage under $\tau_{\n}$ over $D$). If $D$ is chosen to be a genuine fundamental domain, then this representation is unique up to boundary ambiguity. For a general relatively compact $D$ with $\pi(D)=M_{0}$, uniqueness is not required: the estimate holds uniformly for every pair $(\n,x_{0})$. 

The matrix $H$ is the Hessian of the lowest eigenvalue of the operator that is obtained after decomposing $\Delta_{M}$ using Floquet theory, see Lemma \ref{lemma:resonance-non-degenerate} and the proof of the theorem for further details.

The preceding result also admits a probabilistic interpretation. Up to the conventional normalization of the Laplacian, $H_{M}~d\vol_{M}$ is the probability that Brownian motion starting at $x$ lies near $y$ at time $t$, equivalently, $H_{M}$ is its transition density with respect to the Riemannian volume measure. Hence, Theorem \ref{theorem:main1-global} is a refined local central limit theorem for its displacement in the Abelian covering direction: on the diffusive scale $|\n-\mathbf{m}|=\mathcal{O}(\sqrt{t})$, this displacement is governed by a centered Gaussian law whose covariance is determined by $H$, with the functions $P_{j}$ providing higher-order corrections. The local expansion of Theorem \ref{theorem:main1-local} describes the corresponding transition probabilities when both endpoints remain in fixed compact subsets of the cover.

\subsection{Decay on Abelian covers of principal bundles} \label{subsection:results-isometric}

We now consider a more general setting. Let $p_{0} \colon P_{0} \to M_{0}$ be a principal $G$-bundle where $G$ is a compact connected Lie group, and assume that $P_{0}$ is equipped with a $G$-equivariant connection $\nabla^{P_{0}}$. The bundle $P_{0}$ carries a canonical volume measure $\nu_{0}$ given locally by $p_{0}^{*}(\dd \vol_{M_{0}}) \wedge dg$, where $dg$ is the Haar probability measure of $G$. 

Let $\mathrm{Ad}(P_{0}):= P_{0}\times_{\mathrm{Ad}}\mathfrak{g} \to M_{0}$ be the adjoint bundle associated to the adjoint representation $\mathrm{Ad} \colon G \to \mathrm{End}(\mathfrak{g})$. We will say that the curvature $F_{\mathrm{Ad}(P_{0})} \in C^{\infty} (M_{0}, \Lambda^{2} T^{*} M_{0} \otimes \mathrm{Ad}(P_{0}))$ of $\nabla^{P_{0}}$ is \emph{nondegenerate at $x$} if the following holds:
\[
\mathrm{Span} \{ F_{\operatorname{Ad}(P_{0})}(x)(X, Y) : X, Y \in T_x M_{0} \}=\mathrm{Ad}(P_{0})_{x}.
\]
The curvature is globally nondegenerate if it is nondegenerate at every $x \in M_{0}$. 

Now let $P := \pi^*P_0 \to M$ be the pullback bundle over $M$ (where as in \S\ref{subsection:results-abelian}, $\pi \colon M \to M_0$ is a $\Z^d$-cover). Note that $P$ is also a $(G \times \Z^d)$-bundle over $M_0$, and $\mathbb{Z}^d$-bundle over $P_0$ (see \cite[\S2.3]{Cekic-Lefeuvre-Munoz-Thon-26}), that is the following commutative diagram holds:
\[
\begin{tikzcd}
P = \pi^*P_0 \arrow{r}{} \arrow[swap]{dr}{} \arrow[swap]{d}{p} & P_0 \arrow{d}{p_{0}} \\
M \arrow{r}{\pi} & M_0
\end{tikzcd}
\]
Abusing of the notation, the map $P \to P_{0}$ will be also called $\pi$.

We can lift $\nabla^{P_{0}}$ to obtain a connection $\nabla=\pi^{*}\nabla^{P_{0}}$ on $P$. Viewing $P \to M_{0}$ as a $G\times \Z^{d}$ bundle, the holonomy group based at $x_{0} \in M_{0}$
\[
\mathrm{Hol}(P, \nabla)=\{\tau_{\gamma}: \gamma \text{ is a loop based at }x_{0}\}.
\]
where $\tau_{\gamma} \colon P_{x_{0}} \to P_{x_{0}}$ denotes parallel transport along $\gamma$. After choosing a point in $P_{x_0}$, and hence identifying the fiber $P_{x_0}$ with $G \times \Z^d$, we can identify it with a subgroup of $G \times \Z^d$. We refer to \S\ref{subsection:curvatures} for more details.

The measure $\nu_{0}$ can be lifted to a measure $\nu$ (of infinite volume when $d > 0$). In the following result, given $f \in C^\infty_{\comp}(P)$, we let $f_{\bk=\mathbf{0}} \in C^\infty_{\comp}(M)$ be the function obtained by averaging in the $G$-fibers of $P$, namely
\[
f_{\bk=\mathbf{0}}(x) := \int_{P_x} f(w \cdot g) \dd g, \qquad x \in M
\]
where $w \in P_x$ is arbitrary. The meaning of the index $\bk=\mathbf{0}$ will become clear later, when introducing the Borel--Weil calculus on principal bundles, see \S\ref{ssection:bw-calculus}. 

Using $\nabla$, we obtain a horizontal distribution $\mathbb {H}_{P} \subset TP$. The restriction of the exterior derivative to $\HH_{P}$ defines the horizontal exterior derivative
\[
d_{\mathbb{H}_{P}}\colon C^{\infty}(P) \to C^{\infty}(P,\mathbb{H}_{P}^{*}), \qquad d_{\mathbb{H}_{P}}f:=df|_{\mathbb{H}_{P}}.
\]
The \emph{horizontal Laplacian} on $P$ is
\[
\Delta_P^{\mathbb H} :=(d_{\mathbb H_{P}})^*d_{\mathbb H_{P}},
\]
where $(d_{\mathbb H_{P}})^*$ denotes the formal $L^2$-adjoint in $P$, we refer to \S\ref{subsection:horizontal-laplacian} for more details. Let $H_{P}^{\mathbb{H}}(t,u,v)$ denote the corresponding heat kernel. Probabilistically, the diffusion generated by $-\Delta_{P}^{\HH}$ is the horizontal Brownian motion on $P$.

We shall establish the following analog of Theorem \ref{theorem:main1}:

\begin{theorem}
\label{theorem:main2}
Assume that $\mathrm{Hol}(P, \nabla)$ is dense in $G \times \Z^{d}$, and the curvature of $\nabla^{P_{0}}$ is globally nondegenerate. Then, there exist a constant $\kappa > 0$, and $\mathcal{A}_{j} \in \mathcal{D}'(P \times P)$ such that for all integers $N \geq 1$, we have in $\mathcal{D}'(P \times P)$:
\[
t^{d/2} H_{P}^{\mathbb{H}}(t)= \kappa  \mathbf{1}_{P}\otimes \mathbf{1}_{P} + \sum_{j =1}^{N-1} t^{-j} \mathcal{A}_j + \mathcal{R}_N(t),
\]
where the following estimates hold: for all $s > 0$, there exists a constant $C > 0$ such that
\begin{equation}
\begin{split}
|\langle \mathcal{A}_j,f \otimes g\rangle| & \leq C\|f_{\bk=\mathbf{0}}\|_{B^{s,2j}}\|g_{\bk=\mathbf{0}}\|_{B^{s,2j}},\quad 1 \leq j \leq N - 1,\\
|\langle \mathcal{R}_N,f \otimes g\rangle| & \leq C \left(e^{-Ct}\|f\|_{L^{2}}\|g\|_{L^{2}} + \langle t \rangle^{-N}\|f_{\bk=\mathbf{0}}\|_{B^{s,2N+d+1}}\|g_{\bk=\mathbf{0}}\|_{B^{s,2N+d+1}}\right), \quad \forall t \geq 0.
\end{split}
\end{equation}
\end{theorem}

As in the case of Abelian covers, the previous theorem describes the long-time behavior of the horizontal heat kernel only in the weak sense. The hypothesis on the holonomy and the curvature are needed in the proof to obtain a spectral gap for the leading eigenvalue of a differential operator into which $\Delta_{P}^{\mathbb{H}}$ decomposes, see \S\ref{thm:spectral-gap}. We also emphasize that we only need the hypothesis on the curvature of $\nabla^{P_{0}}$ rather than the one of $\nabla$. This comes from the fact that the line bundle is flat, see the proof of Lemma \ref{lemma:identity}.     Regarding the constraint on the holonomy, we actually require a little less than the density of the full group, see Remark \ref{remark:holonomy}.

The main example to keep in mind is (as in \cite{Cekic-Lefeuvre-Munoz-Thon-26}), to consider Abelian covers the frame bundle of a closed manifold $FN_{0}$, which is a $\mathrm{SO}(n-1)$-principal bundle over the unit sphere bundle $SN_{0}$, where $\dim N_{0}=n$.

Our next result upgrades this to a local asymptotic expansion of the heat kernel itself. A notable feature of the principal bundle setting is that, under the holonomy and curvature assumptions above, the leading asymptotics are insensitive to the $G$-fiber variables: the coefficients depend only on the projections of the points to the base manifold. Thus, locally, the horizontal heat kernel exhibits the same asymptotic behavior as the heat kernel on the underlying Abelian cover.

\begin{theorem}
\label{theorem:main2-local}
Assume that $\mathrm{Hol}(P, \nabla)$ is dense in $G \times \Z^{d}$, and the curvature of $\nabla^{P_{0}}$ is globally nondegenerate. Then for every pair of compact sets $K_{u},K_{v}\subset P$, every $\ell_{u},\ell_{v}\geq 0$, and every integer $N \geq 1$, there exist smooth functions $k_{j}^{P}\in C^{\infty}(K_{u}\times K_{v})$, $j=1,\ldots,N-1$, such that in
$C^{\ell_{u},\ell_{v}}(K_{u}\times K_{v})$
\[
t^{d/2}H_{P}^{\mathbb H}(t,u,v)
=
\kappa + \sum_{j=1}^{N-1}t^{-j}k_{j}^{P}(u,v) + R_{N}^{P}(t,u,v),
\]
where the following holds:
\[
\|R_{N}^{P}(t)\|_{C^{\ell_{u},\ell_{v}}(K_{u} \times K_{v})} \leq C_{K_{u},K_{v},N,\ell_{u},\ell_{v}}(e^{-ct}+\langle t\rangle^{-N}), \quad t \geq 1.
\]
Furthermore, if $k_{j}^{M}\in C^{\infty}(p(K_{u})\times p(K_{v}))$ are the coefficients in Theorem \ref{theorem:main1-local}, then $k_{j}^{P}(u,v) = k_{j}^{M}(p(u),p(v))$.
\end{theorem}

As in the purely Abelian case (Theorem \ref{theorem:main1-global}), we have a global expansion which shows the influence of a Gaussian profile. Similarly as in the local one, the coefficients are entirely determined by the geometry of the base manifold, that is, they are constant along the $G$-fibers and coincide with those of the corresponding expansion on the Abelian cover $M$.

\begin{theorem}[Global Expansion]
\label{theorem:main2-global}
Assume that $\mathrm{Hol}(P, \nabla)$ is dense in $G \times \Z^{d}$, and the curvature of $\nabla^{P_{0}}$ is globally nondegenerate. Let $D\Subset M$ be an open subset with smooth boundary such that $\pi(D)=M_{0}$, and set $D_{P}:=p^{-1}(D) \subset P$. Then there exists a positive-definite matrix $H$ such that, for every $R>0$, every integer $N \geq 1$, and every $\ell_{u},\ell_{v} \geq 0$, there exist functions $P_{j}^{P} \in C^{\infty}(\R^{d} \times \overline{D_{P}}\times \overline{D_{P}})$,  $j=1,\ldots,N-1$, polynomial in their first variable, for which the following holds. For every $t \geq 1$, every $\n, \mathbf{m} \in \Z^{d}$, and every $u_{0},v_{0} \in \overline{D_{P}}$ satisfying $|\n-\mathbf{m}|\leq R\sqrt t$,
one has
\[
\begin{split}
t^{d/2} H_{P}^{\mathbb {H}} (t,\tau_{\n}u_{0},\tau_{\mathbf{m}}v_{0}) 
=&\exp\left( -\frac{1}{2t}\langle H^{-1}(\n-\mathbf{m}),\n-\mathbf{m} \rangle \right)
\\
&\times \left[ \kappa + \sum_{j=1}^{N-1} t^{-j/2} P_{j}^{P} \left(
\frac{\n-\mathbf{m}}{\sqrt t}, u_{0},v_{0} \right) \right] + R_{N}^{P}( t,\n-\mathbf{m},u_{0},v_{0}).
\end{split}
\]
Moreover,
\[
\sup_{|\mathbf{r}|\leq R\sqrt t} \| R_{N}^{P}(t,\mathbf{r},\bullet,\bullet)\|_{C^{\ell_{u},\ell_{v}} (\overline{D_{P}}\times\overline{D_{P}})} \leq C_{R,N,\ell_{u},\ell_{v}} ( e^{-ct}+t^{-N/2}).
\]
More precisely, the coefficients may be chosen to be constant in both $G$-fiber variables. If $P_{j}^M \in C^\infty (\R^d\times\overline{D}\times\overline{D})$
are the coefficients of Theorem \ref{theorem:main1-global}, then $P_{j}^{P}(\mathbf{z},u_{0},v_{0})= P_{j}^{M} (\mathbf{z},p(u_{0}),p(v_{0}))$.
\end{theorem}

Thus, the horizontal Brownian motion has, to every polynomial order in the large-time expansion, the same diffusive behavior in the Abelian covering direction as its projection to $M$. Dependence on the $G$-fiber variables is carried by nontrivial representation modes and is exponentially small.

Finally, let us comment that although at first sight, it may seem more natural to study the Laplace--Beltrami operator on the total space $P$, we decided to work instead with the horizontal Laplacian. We have two main reasons. First, $\Delta_{P}^{\mathbb{H}}$ is naturally associated with the geometry of the connection, and is the relevant diffusion in the sub-Riemannian setting. Second, this operator already captures the long-time behavior of the full heat semigroup. Indeed, after endowing $P$ with a connection metric (introduced in \S\ref{subsection:horizontal-laplacian}), the operator $\Delta_{P}$ decomposes as the sum of the horizontal Laplacian and the Laplacian along the compact $G$-fibers. Since these operators commute, the heat semigroup factorizes accordingly. The fiberwise Laplacian has a discrete spectrum, whose kernel consists precisely of the functions constant along the fibers, while every higher fiber mode decays exponentially fast because of the spectral gap. Consequently, the large-time asymptotics are entirely determined by the horizontal heat semigroup, and our results immediately imply the corresponding expansions for the Laplace--Beltrami heat kernel.

\begin{corollary}
\label{corollary:laplace-expansion}
Let $\Delta_{P}$ denote the Laplace--Beltrami operator associated with a connection metric on $P$, and let $H_{P}(t,u,v)$ be its heat kernel. Under the assumptions of Theorems \ref{theorem:main2}, \ref{theorem:main2-local}, and \ref{theorem:main2-global}, the corresponding conclusions remain valid after replacing $H_{P}^{\mathbb H}(t,u,v)$ by $H_{P}(t,u,v)$.
\end{corollary}

\subsection{Related results}

Let $M$ be a complete Riemannian manifold endowed with a free cocompact action of $\mathbb{Z}^{d}$. In \cite{Lott-99}, Lott studied the long-time behavior of the heat kernel in the diffusive regime, in which the distance between the two endpoints is comparable with $\sqrt{t}$, similarly as in Theorem \ref{theorem:main1-global}. He showed that, after rescaling, the heat propagation is governed by a Euclidean metric on $\R^{d}$, defined by the Hodge inner product on the space of cohomology classes associated with the covering. Subsequently, Kotani and Sunada established a local central limit theorem and an asymptotic power-series expansion for the heat kernel on a manifold with a cocompact Abelian group action \cite[Theorems 3 \& 4]{Kotani-Sunada-00}. Thus, on the underlying Abelian cover, Theorems \ref{theorem:main1-local} and \ref{theorem:main1-global} should be viewed as belonging to the line of work initiated by \cite{Lott-99,Kotani-Sunada-00}. Furthermore, our approach is based on similar ideas, relying on a Floquet decomposition and an analysis of a twisted Laplacian.

Relative to Lott’s work, our results provide a substantially finer asymptotic description. While \cite{Lott-99} identifies the leading diffusive profile and the effective Euclidean geometry governing heat propagation, our theorems yield asymptotic expansions to arbitrary order together with quantitative estimates for the differentiated remainders.

The comparison with Kotani–Sunada is more subtle. Their work concerns the same diffusive regime and already establishes a local central limit theorem together with a full asymptotic power-series expansion of the heat kernel. Our contribution therefore lies not in the existence of such an expansion, but rather in its functional-analytic formulation. More precisely, our expansions are accompanied by remainder estimates that are uniform after differentiation in both kernel variables, to arbitrarily prescribed finite order, throughout the diffusive region. Moreover, these pointwise expansions are complemented by the distributional expansion of Theorem \ref{theorem:main1}, whose coefficients define continuous operators on $C^{\infty}_{\comp}(M)$ and whose remainder is controlled in the weighted spaces $B^{s,r}(M)$. 

For manifolds with boundary, Geng and Iyer \cite{Geng-Iyer-21} obtained exact long-time asymptotics for heat kernels on Abelian covers of compact manifolds with smooth boundary, allowing mixed Dirichlet and Neumann boundary conditions.

Related asymptotic questions for periodic elliptic operators have also been considered at the level of Green functions. In particular, Kha \cite{Kha-18} obtained precise spatial asymptotics for Green functions of generic periodic second order elliptic operators on Abelian covers, both inside spectral gaps and at their edges. The common analytic mechanism is the reduction, by Floquet theory, to the perturbation of a distinguished band function near a nondegenerate spectral minimum. In the present work, the corresponding band is the lowest eigenvalue $\lambda_{0}(\btheta)$ of the twisted Laplacian, and its Hessian at $\btheta=\mathbf{0}$ determines the Gaussian profile in
Theorem \ref{theorem:main1-global}.

To the best of our knowledge, analogous long-time expansions for horizontal heat kernels on Abelian covers of compact principal bundles, uniform simultaneously in the Abelian and (compact) group variables, do not appear in the previous literature. In particular, the reduction of the leading asymptotics to the fiberwise constant mode, and the resulting identification of all the polynomial coefficients with those of the underlying Abelian cover, seem to be specific to the present setting. Our analysis relies on the recently developed Borel--Weil calculus introduced by Ceki\'c and Lefeuvre \cite{Cekic-Lefeuvre-24}. This calculus provides a semiclassical description of $G$-equivariant operators on a compact principal $G$-bundle, where the semiclassical parameters are the highest weights $\bk$ parametrizing the irreducible representations of $G$. Applied to a horizontal Laplacian, it gives a family of twisted operators $\Delta_{\bk}$. Under global nondegeneracy of the curvature, one has a lower bound of the form $\lambda_{0} (\Delta_{\bk}) \geq c|\bk|$ for  $|\bk|\gg 1$, whereas dense holonomy gives control on whether the eigenvalue is zero. We generalize these results to our twisted case to obtain a spectral gap for the decomposed operator $\Delta_{\btheta,\bk}$. This allows us to show that every nontrivial $G$-type is negligible in the large-time limit, and conclude in virtue of the results in the purely Abelian case. Hence, the arguments of the present paper combine this Borel--Weil decomposition with the Floquet decomposition in the Abelian cover. In this sense, the proof involves two dual parameters, namely, $\btheta\in \mathrm{U}(1)^{d}$, associated with the Abelian covering, and $\bk$, associated with the irreducible representations of $G$.

In the circle-bundle case $G=\mathrm{U}(1)$, the operators $\Delta_{\bk}$ are magnetic Laplacians. Polynomial growth of their lowest eigenvalue can persist even when the curvature is degenerate. For example, Helffer and Kordyukov \cite{Helffer-Kordyukov-09} considered magnetic fields which vanish regularly to order $r\geq 0$ along a hypersurface and proved that $\lambda_{1}(\Delta_{\bk}) \gtrsim |\bk|^{\frac{2}{r+2}}$. The globally nondegenerate case corresponds to $r=0$ and gives the linear lower bound. This suggests that some of our conclusions may continue to hold under suitable finite-order degeneracy assumptions. We discuss this particular case in \S\ref{subsection:G=U(1)}.

Finally, let us mention that the present work is also motivated by the recent dynamical results by the author in collaboration with Ceki\'c and Lefeuvre \cite{Cekic-Lefeuvre-Munoz-Thon-26}, who established asymptotic expansions for correlations of isometric extensions of volume-preserving Anosov flows on Abelian covers. There, we also combine Floquet theory in the Abelian variable with the Borel--Weil calculus in the compact-group variable. In contrast with this dynamical work, the heat semigroup considered here is smoothing and self-adjoint. This makes it possible not only to obtain weak asymptotics of correlations, as in Theorems~\ref{theorem:main1} and~\ref{theorem:main2}, but also to derive the local and diffusive-scale pointwise expansions of the heat kernels themselves.

\subsection{Organization of the article} We begin with \S\ref{section:abelian-covers}, where we briefly review some constructions on Abelian covers from \cite{Cekic-Lefeuvre-Munoz-Thon-26}, and certain facts about Borel--Weil calculus developed in \cite{Cekic-Lefeuvre-24}. In Section \ref{section:abelian-extension-laplacian} we study the Laplacian on Abelian extensions and prove the results from \S\ref{subsection:results-abelian}. Then, in Section \ref{section:decay-isometric} we prove the asymptotic expansions for the horizontal Laplacian, i.e., the results from \S\ref{subsection:results-isometric}. To this end, we prove a spectral gap for the induced operator in the Borel--Weil calculus in \S\ref{subsection:spectral-gap}. We also discuss the case $G=\mathrm{U}(1)$ in \S\ref{subsection:G=U(1)}.

\subsection{Acknowledgments.} I would like to thank Thibault Lefeuvre and Rento Velozo Ruiz for discussions related to the project and for comments on an earlier draft of the paper.

The author was supported by the European Research Council (ERC) under the European Union’s Horizon 2020 research and innovation programme (Grant agreement no. 101162990 -- ADG).

\section{Abelian covers and isometric extensions}

In this section we review several preliminaries needed for our analysis. We begin with the definitions about Abelian covers in \S\ref{ssection:abelian-covers}. We then discuss well-known facts about Floquet theory in \S\ref{ssection:floquet-theory}. In \S\ref{ssection:functional-spaces} we introduce functional spaces that will be useful to our asymptotic analysis, while in \S\ref{ssection:bw-calculus} we briefly review the Borel--Weil theory. Finally, in \S\ref{subsection:curvatures} we review the definitions of curvature, holonomy, and certain differential operators on the Borel--Weil calculus of particular importance for this work.

\label{section:abelian-covers}

\subsection{Abelian covers}

\label{ssection:abelian-covers}

In this subsection we briefly give definitions and facts about Abelian covers and how we will trivialize them.

Let $M_0$ be a smooth closed manifold, $x_0 \in M_0$ be an arbitrary fixed point, and let
\begin{equation}
    \label{equation:representation-rho}
    \rho \colon \pi_1(M_0,x_0) \to \Z^d
\end{equation}
be a surjective representation onto $\Z^d$ for some $d \geq 1$. Let
\[
M := \widetilde{M}_0 \times_{\rho} \Z^d
\]
be the associated Abelian cover, where $\widetilde{M}_0$ denotes the universal cover of $M_0$. That is $M = (\widetilde{M}_0 \times \Z^d)/\sim$ where

\[
    (x,\mathbf{n}) \sim (\gamma.x, - \rho(\gamma)+\mathbf{n}),\quad \forall x \in \widetilde{M}_0, \n \in \Z^d, \gamma \in \pi_1(M_0, x_0).
\]
Note that $M$ is connected as $\rho$ is surjective, that it is a principal $\Z^d$-bundle over $M_0$, namely there is a surjective projection $\pi \colon M \to M_0$ such that the preimage of each point can be identified with $\Z^d$, and it is equipped with a $\Z^d$-action $\tau \colon M \times \Z^d \to M$. On the universal cover, this action is given by 
\[
    \tau_\kk(x,\n) = (x,\n + \kk),\quad \kk, \n \in \mathbb{Z}^d, x \in \widetilde{M}_0,
\]
and it is immediate to verify that it descends to the quotient $M$. Finally, we can lift the Riemannian metric $g_{0}$ to a Riemannian metric $g$ in $M$ in a unique way so that $\pi \colon M \to M_{0}$ is a local isometry.

Let $\mathrm{U}(1) := \R/2\pi\Z$. Given $\btheta \in \mathrm{U}(1)^d$, we introduce the representation
\[
    \alpha_{\btheta} \colon \Z^d \to \mathrm{U}(1), \qquad \alpha_{\btheta}(\n) :=  \btheta \cdot \n ~ \text{ mod } 2\pi.
\]
We will write $H^1_{\mathrm{dR}}(M_0, \mathbb{R})$ for the first de Rham cohomology group of $M_0$, and denote by $H^1_{\mathrm{dR}}(M_0, 2\pi \mathbb{Z})$ the lattice inside of it, whose elements $[\eta]$ satisfy for any closed loops $\gamma \subset M_0$
\[
    \int_\gamma [\eta] \in 2\pi \mathbb{Z}.
\]

The following is taken from  \cite[\S2.1]{Cekic-Lefeuvre-Munoz-Thon-26}

\begin{proposition}\label{prop:eta-theta}
\begin{enumerate} [label=(\roman*), itemsep=5pt]
    \item There exists a smooth map
    \[
    \mathrm{U}(1)^d \ni \btheta \mapsto [\eta_{\btheta}] \in H^1_{\mathrm{dR}}(M_0, \mathbb{R})/H^1_{\mathrm{dR}}(M_0, 2\pi \mathbb{Z})
    \]
    such that for all smooth closed curves $\gamma \subset M_0$:
    \begin{equation}
        \label{equation:eta-theta}
            \int_\gamma [\eta_{\btheta}] = \rho(\gamma) \cdot \btheta ~ \mod 2\pi.
    \end{equation}
    Moreover, the dependence on $\btheta$ is linear in the sense that the map $\btheta \mapsto [\eta_{\btheta}]$ descends from a linear map $\mathbb{R}^d \ni \btheta \mapsto [\widetilde{\eta}_{\btheta}] \in H^1_{\mathrm{dR}}(M_0, \mathbb{R})$, and so is in particular smooth. As a consequence, for any $v \in T_{\btheta} \mathrm{U}(1)^d \simeq \mathbb{R}^d$, we have $D[\eta_{\btheta}](v) = [\widetilde{\eta}_{v}]$. \item Introduce
    \begin{equation}
    \label{equation:stheta}
        s_{\btheta}(x) := \exp\left(i \int_\gamma \eta_{\btheta} \right), \qquad x \in M,
    \end{equation}
    where $\gamma := \pi(\gamma')$ and $\gamma'$ is an arbitrary path joining $\widehat{x}_0$ to $x$ in $M$. Then $s_{\btheta}$ is well-defined, and  
    \[
        s_{\btheta} \circ \tau_\n = e^{i \n \cdot \btheta} s_{\btheta},\quad s_{\btheta}^{-1} d s_{\btheta} = i \pi^* \eta_{\btheta},\quad \forall \mathbf{n} \in \mathbb{Z}^d.
    \]
    \item Let $U \subset \mathrm{U}(1)^d$ be an open contractible set. Then there exists a smooth map
    \[
        U \ni \btheta \mapsto \eta_{\btheta} \in C^\infty(M_0, T^*M_0) \cap \ker d, 
    \]
    whose de Rham cohomology class agrees with the one of $[\widetilde{\eta}_{\btheta}]$ constructed in (i).  Moreover, if $\mathbf{0} \in U$, we may assume that $\eta_{\mathbf{0}} \equiv 0$, and we can assume that $(\eta_{\btheta})_{\btheta \in U}$ are harmonic with respect to a background Riemannian metric $g_0$ on $M_0$. Consequently, the conclusions of (2) are valid with smooth dependence on $\btheta \in U$.
\end{enumerate}
    
\end{proposition}

\subsection{Floquet theory} \label{ssection:floquet-theory} Here we present basic concepts of Floquet theory. In particular, we review the decomposition of functions into its Fourier modes.

For $\btheta \in \mathrm{U}(1)^d$ we introduce the associated line bundles
\[
    L_{\btheta} := \widetilde{M}_0 \times_{e^{-i \alpha_{\btheta} \circ \rho}} \mathbb{C} = M \times_{e^{-i\alpha_{\btheta}}} \mathbb{C},
\]
where as usual we have 
\[
    M \times_{e^{-i\alpha_{\btheta}}} \C = M \times \C/\sim,\quad (x, z) \sim (\tau_{\n} x, e^{i \btheta \cdot \n}z), \forall \n \in \Z^d.
\]
Observe that $L_{\btheta}$ is a flat Hermitian (i.e. equipped with an inner product in its fibers and a compatible flat connection) line bundle with holonomy along a curve $\gamma \in \pi_1(M_0, x_0)$ given by
\[
\exp\left(-i \alpha_{\btheta} \circ \rho(\gamma)\right) = \exp\left(-i \btheta \cdot \rho(\gamma)\right).
\]
From the definition of $L_{\btheta}$, we see that the sections in $C^\infty(M_0, L_{\btheta})$ correspond bijectively to equivariant functions in
\begin{equation}
    \label{equation:equivariant-space}
    C^\infty_{\btheta}(M) := \{F \in C^\infty(M) \mid \forall \mathbf{n} \in \mathbb{Z}^d,\, F \circ \tau_\n = e^{i \btheta \cdot \n} F\},
\end{equation}
the space of Fourier modes of frequency $\btheta$. Given a section $s \in C^\infty(M_0, L_{\btheta})$ we will denote its equivariant lift by $\widetilde{s} \in C^\infty_{\btheta}(M)$.

We now observe that the line bundles $L_{\btheta}$ are topologically trivial.

\begin{proposition}
\label{proposition:ltheta-trivial}
    Let $U \subset \mathrm{U}(1)^d$ be an open contractible set. Then, the family $(L_{\btheta})_{\btheta \in U}$ of line bundles is smoothly trivial, i.e. there exists a smooth family $(s_{\btheta})_{\btheta \in U}$ such that $s_{\btheta} \in C^\infty(M_0, L_{\btheta})$ is nowhere vanishing. 
\end{proposition}
    By a smooth family of sections $(s_{\btheta})_{\btheta \in U}$ here we mean that the family of equivariant lifts $\widetilde{s}_{\btheta} \in C^\infty_{\btheta}(M) \subset C^\infty(M)$ depends smoothly on $\btheta$. Recall however, that the family of line bundles $\{L_{\btheta} \mid \btheta \in \mathrm{U}(1)^d\}$ is not trivial as a \emph{family}, i.e. we cannot choose a smooth family $\{s_{\btheta} \mid \btheta \in \mathrm{U}(1)^d\}$ of trivialising sections, see \cite[Remark~2.7]{Cekic-Lefeuvre-Munoz-Thon-26}.

After choosing a representative $\eta_{\btheta}$ of $[\eta_{\btheta}]$ as in Proposition \ref{prop:eta-theta}, there exists an equivariant function $s_{\btheta} \in C_{\btheta}^\infty(M)$ of pointwise unit norm, and an isomorphism
\[
C^\infty(M_0) \to C^\infty_{\btheta}(M), \qquad f \mapsto F_{\btheta} := (\pi^*f) s_{\btheta}.
\]
Conversely, the inverse map is denoted by
\[
C^\infty_{\btheta}(M) \to C^\infty(M_0), \qquad F_{\btheta} \mapsto f_{\btheta},
\]
where $f_{\btheta} \in C^\infty(M_0)$ is the unique function such that $F_{\btheta}/s_{\btheta} = \pi^*f_{\btheta}$. We note that since $s_{\btheta}$ has pointwise unit norm, it is uniquely determined up to multiplication by a function of the form $\pi^*h$, where $h \in C^\infty(M_0)$ has pointwise unit norm; therefore the identification above is also unique up to multiplication by $\pi^*h$. Note also that $C^\infty_{\mathbf{0}}(M)$ corresponds to $\Z^d$-periodic functions, so they are pullbacks of functions on $M_0$, that is $C^\infty_{\mathbf{0}}(M) = \pi^* C^\infty(M_0)$. Furthermore, we can do this procedure \emph{smoothly} for $\btheta$ varying in a contractible open set $U \subset \mathrm{U}(1)^d$.

We introduce the following $L^2$-norm on $C^\infty_{\btheta}(M)$ which coincides with the $L^2$-norm on sections of $L_{\btheta}$ under the identification above:
\[
    \|F_{\btheta}\|^2_{L^2_{\btheta}(M)} := \|F_{\btheta}\|_{L^2(M_{0}, L_{\btheta})}^2 = \|f_{\btheta}\|^2_{L^2(M_0)},
\]
where $L^{2}(M_{0})$ is with respect to the volume measure $\vol_{M_0}$. Here we abuse the notation slightly by identifying sections of $L_{\btheta}$ with equivariant functions in $C_{\btheta}^\infty(M)$. The space $L^2_{\btheta}(M)$ is then defined as the completion of $C^\infty_{\btheta}(M)$ with respect to this norm. 

We now express an arbitrary smooth function in terms of its Fourier modes.

\begin{proposition}\label{prop:fourier-theory}
    Let $f \in C^\infty_{\comp}(M)$. Then 
    \begin{equation}
        \label{equation:floquet-decomposition}
            f = \dfrac{1}{(2\pi)^d} \int_{\mathrm{U}(1)^d} F_{\btheta} ~\dd \btheta,
    \end{equation}
    where $F_{\btheta} \in C_{\btheta}^\infty(M)$ is defined as
\begin{equation}\label{eq:theta-frequency}
    F_{\btheta}(x) := \sum_{\n \in \Z^d} f(\tau_{\n}(x)) e^{-i \n \cdot \btheta}, \quad x \in M.
\end{equation}
Finally, we have the following Parseval identity: for all $f,g \in C^\infty_{\comp}(M)$,
\begin{equation}
\label{equation:parseval}
\langle f,g\rangle_{L^2(M)} = \dfrac{1}{(2\pi)^d} \int_{\mathrm{U}(1)^d} \langle F_{\btheta}, G_{\btheta}\rangle_{L^2(M_0, L_{\btheta})}\, \dd\btheta.
\end{equation}
\end{proposition}
Here, $L^{2}(M)$ is defined with respect to the measure $\vol_{M}$.

We end this section with an auxiliary claim about the result of deriving $f_{\btheta}$ in $\btheta$ and the corresponding operation on $f$. We observe that
\begin{equation}\label{eq:auxiliary-Hj}
    s_{\btheta}^{-1} \partial_{\btheta_j} s_{\btheta} = i H_j
\end{equation} showing directly that $H_j$ is well-defined. 

We fix an open contractible set $U \subset \mathrm{U}(1)^d$ as above and work over $U$. We define
\[
    H_j(x) := \int_{\widehat{x}_0}^x \pi^* \partial_{\btheta_j} \eta_{\btheta},\quad x \in M,\quad 1 \leq j \leq d,
\]
where the integration is along an arbitrary path between the basepoint $\widehat{x}_0 \in M$ and $x$. For a multi-index $\alpha$ we may set $\mathbf{H}^\alpha := \prod_{j = 1}^d H_j^{\alpha_j}$. We then have

\begin{proposition}\label{prop:differentiation-btheta}
    For an arbitrary multi-index $\alpha$ we have
    \[
        \partial^\alpha_{\btheta} f_{\btheta} = (-i)^{|\alpha|}(\mathbf{H}^\alpha f)_{\btheta},\quad f \in C^\infty_{\comp}(M),\quad \btheta \in U.
    \]
\end{proposition}

\subsection{Functional spaces on Abelian covers}  \label{ssection:functional-spaces}

In this subsection we introduce the functional spaces that appear in the bounds of our main theorems.

Let $g_0$ be the Riemannian metric on $M_0$ and let $g$ denote its lift to $M$ so that $\pi \colon M \to M_{0}$ is a local isometry. Let $\nabla$ be the Levi-Civita connection induced by $g$ on $M$. Recall that, given $u \in C^\infty(M)$ and an integer $s \geq 0$, $\nabla^s u$ is a section of the symmetric $s$-tensor bundle $\mathrm{Sym}^s(T^*M) \to M$ such that
\[
\nabla^s u (x ; v, \dotsc, v) := \partial^s_t u(\gamma(t))|_{t = 0}, \qquad \forall (x,v) \in TM,
\]
where $t \mapsto \gamma(t)$ is the geodesic generated by $(x,v)$. The vector bundle $\mathrm{Sym}^s(T^*M) \to M$ carries a natural inner product in its fibers induced by the Riemannian metric $g$. We note that by definition, for every $\mathbf{n} \in \mathbb{Z}^d$, we have that $\tau_{\mathbf{n}} \colon (M, g) \to (M, g)$ is an isometry and that the differential of $\tau_{\mathbf{n}}$ is an isometry of the fibers of $\mathrm{Sym}^s(T^*M)$.

Given an integer $s \geq 0$, and an open subset $U \subset M$, the $H^s$-norm on $U$ is defined as:
\[
\|u\|^2_{H^s(U)} := \|u\|^2_{L^2(U)} + \|\nabla^s u\|^2_{L^2(U, \mathrm{Sym}^s T^*U)},
\]
and the $C^s$-norm is defined as
\[
\|u\|_{C^s(U)} := \|u\|_{C^0(U)} + \|\nabla^s u\|_{C^0(U, \mathrm{Sym}^s T^*U)}.
\]
Since $\tau_{\mathbf{n}}$ acts by isometries of $\mathrm{Sym}^s(T^*M)$, we have 
\begin{equation}\label{eq:symmetry-CHs}
    \|\tau_{\mathbf{n}}^*u\|_{F^s(U)} = \|u\|_{F^s(\tau_{\mathbf{n}}U)},\quad F \in \{C, H\}. 
\end{equation}
Let $\widehat{x}_0 \in M$ be an arbitrary point in $M$. Let $D \subset M$ be a relatively compact open subset with smooth boundary such that $\widehat{x}_0 \in D$ and $\pi \colon D \to M_0$ is surjective. For nonnegative integers $r, s \geq 0$, we introduce the norm
\[
\|u\|_{B^{s,r}(M)} := \sum_{\n \in \Z^d} \langle \n \rangle^{r} \|u\|_{H^s(\tau_{\mathbf{n}}D)},
\]
where $\langle \n \rangle := (1+|\n|^2)^{1/2}$ is the Japanese bracket, and $|\bullet|$ is the Euclidean norm in $\Z^d$. The space $B^{s,r}(M)$ is defined as the completion of $C^\infty_{\comp}(M)$ with respect to the previous norm. Similarly, we let
\[
\begin{split}
\|u\|_{C^{s,r}(M)} & := \sup_{\n \in \Z^d} \langle \n \rangle^{r} \|u\|_{C^s(\tau_{\mathbf{n}}D)},
\end{split}
\]
and define $C^{s,r}(M)$
as the completion of $C^\infty_{\comp}(M)$ with respect to the norm $\|\bullet\|_{C^{s,r}(M)}$.
It is straightforward to verify that the functional space $F^{s,r}(M)$ is independent of the choice of background metric $g_0$ on $M_0$ and domain $D \subset M$, for $F \in \{B,C\}$ (a different choice of $D$ would give equivalent norms). For non integer values of $r,s \geq 0$, and $F \in \{B,C\}$, the spaces $F^{s,r}(M)$ are defined by real interpolation. 

An open subset $U \subset \mathrm{U}(1)^d$ is called \emph{good} if there exists an open neighbourhood $V \subset \mathrm{U}(1)^d$ of $\overline{U}$ such that $V$ is contractible. In the next lemma, given $f \in C^\infty_{\comp}(M)$, and $U \subset \mathrm{U}(1)^d$, a good open subset, we write $f_\bullet \in C^\infty(U,C^\infty(M_0))$ to denote the function $U \to C^\infty(M_0), \btheta \mapsto f_{\btheta}$, defined in \S \ref{ssection:floquet-theory}. For $k \geq 0$ an integer, we will write 
\[
    \|f_{\bullet}\|_{C^k(U, H^s(M_0))} := \max_{|\alpha| \leq k} \sup_{\btheta \in U} \|\partial^{\alpha}_{\btheta} f_{\btheta}\|_{H^s(M_0)}.
\]

The following lemma is taken from \cite[\S2.4]{Cekic-Lefeuvre-Munoz-Thon-26}

\begin{lemma} Let $n=\dim M_{0}$.
\begin{enumerate} [label=(\roman*), itemsep=5pt]
    \item \label{lemma:relation-spaces}
    Let $r, s \geq 0$. For all $\eps > 0$, there exists a constant $C > 0$ such that:
    \[
    \begin{array}{lll}
    \|u\|_{B^{s,r}(M)} & \leq C \|u\|_{C^{s,r+d+\eps}(M)}, & \qquad \forall u \in C^{s,r+d+\eps}(M), \\
    \|u\|_{C^{s,r}(M)} & \leq C \|u\|_{B^{s+n/2+\eps,r}(M)}, & \qquad \forall u \in B^{s+n/2+\eps,r}(M).
    \end{array}
    \]
    \item \label{lemma:spaces_theta_sobolev}
    Let $U \subset \mathrm{U}(1)^d$ be a good open subset, $k \geq 0$ an integer, and $s \geq 0$ a nonnegative real number. Then, there exists $C > 0$ such that for all $f \in B^{s,k}(M)$
    \[
    \|f_\bullet\|_{C^k(U,H^s(M_0))} \leq C \|f\|_{B^{s,k}(M)}.
    \]
\end{enumerate}    
\end{lemma}

\subsection{Borel--Weil calculus} \label{ssection:bw-calculus}
    
    The Borel--Weil calculus was introduced in \cite{Cekic-Lefeuvre-24} as a key tool to study $G$-equivariant (pseudo)differential operators on $G$-principal bundles $P_0 \to M_0$, when $G$ is a compact Lie group. Given a unitary representation $\rho \colon G \to \mathrm{U}(V_\rho)$, we denote by $E_\rho := M_0 \times_\rho V_\rho$ the corresponding associated vector bundle, defined as the set of equivalence classes $[z,\xi] \sim [z.g, \rho(g)^{-1}\xi]$, where $z \in P_0, \xi \in V_\rho$. 
    
    \subsubsection{General description} Before detailing the construction, let us first explain the general idea of the Borel--Weil semiclassical calculus. Let $\mathbf{Q} \colon C^\infty(P_0) \to C^\infty(P_0)$ be a $G$-equivariant differential operator acting on functions on $P_0$. That is $\mathbf{Q}$ commutes with the right-action of $G$, i.e. $R_g^*\mathbf{Q}=\mathbf{Q}R_g^*$ for all $g \in G$. By the Peter--Weil theorem, any function $f \in C^\infty(P_0)$ can be decomposed into Fourier modes in each fiber of the principal bundle. Namely to any such $f$, it is possible to associate its Fourier transform
    \begin{equation}
        \label{equation:fourier-decomposition}
    \mc{F}(f) := \bigoplus_{\rho \in \widehat{G}} f_\rho,
    \end{equation}
    where the sum runs over $\widehat{G}$, the set of all unitary irreducible representations of $G$, $f_\rho \in C^\infty(M_0, E_\rho^{\oplus \dim E_\rho})$ and $E_\rho \to M_{0}$ is the corresponding associated vector bundle. 
    
    As $\mathbf{Q}$ is equivariant, it acts diagonally on the above decomposition, that is, it induces operators $Q_\rho \colon C^\infty(M_0,E_\rho) \to C^\infty(M_0,E_\rho)$. In certain problems, one is interested in the behavior of the operator $Q_\rho$ as $\rho$ ``tends to infinity'' in the space of irreducible representations. However, this raises a number of issues, the first one being that the rank of $E_\rho$ diverges to $+\infty$ as $\rho$ ``increases'' unless $G$ is Abelian, thus making it very hard to keep track effectively of the behavior of the operator.
    
    A convenient solution is to realize geometrically $E_\rho$, using the Borel--Weil theorem, as a space of holomorphic sections of line bundles defined over a space associated with $P_0$ called the \emph{flag bundle}. More precisely, $V_\rho$ can be explicitly realized as $V_\rho \simeq H^0(G/T, J_\rho)$, where $T < G$ is a maximal torus and $J_\rho$ is a holomorphic line bundle induced by the representation. A section of $E_\rho \to M_0$ can then be identified with a section of the flag bundle $F_0 := P_0/T \to M_0$ with values in a certain complex line bundle $L_\rho$ which is \emph{fiberwise holomorphic}, that is holomorphic in restriction to every fiber $P_0(x_0)/T \simeq G/T$ where $x_0 \in M_{0}$. Here, $L_\rho$ is a complex line bundle whose restriction to $P_0(x_0)/T \simeq G/T$ is holomorphically equivalent to $J_\rho \to G/T$.
    
    \subsubsection{Lie-theoretic preliminaries} \label{sssection:lie} See \cite[Section 2.1.2]{Cekic-Lefeuvre-24} for further discussion. Let $G$ be a compact Lie group and $\widehat{G}$ denote the set of all irreducible representations. The Lie algebra $\fg$ decomposes as
    \[
    \fg = \mathfrak{z}(\fg) \oplus \fg' = \mathfrak{z}(\fg) \oplus [\fg,\fg],
    \]
    where $\mathfrak{z}(\fg)$ denotes the center of $\fg$. Let $T < G$ be a maximal torus. Its Lie algebra then splits as
    \[
    \mathfrak{t} = \mathfrak{z}(\fg) \oplus \mathfrak{t}',
    \]
    where $\mathfrak{t}' := \mathfrak{t} \cap \fg'$. Let $\mathfrak{a}^{\mathrm{ss}}_{+} \subset (i\mathfrak{t}')^{*}$ be the polyhedral convex positive cone spanned by the set of positive roots $\Delta^{+}(\mathfrak{g}_{\mathbb{C}})$ of the group
    (the superscript $\mathrm{ss}$ stands for semisimple, and $(i\bullet)^*$ denotes real-linear functionals from $\bullet$ to $i\mathbb{R}$), and set
    \[
    \mathfrak{a}_{+} := (i\mathfrak{z}(\mathfrak{g}))^{*} \oplus \mathfrak{a}^{\mathrm{ss}}_{+}
    \subset (i\mathfrak{t})^{*},
    \]
    the positive Weyl chamber.

    We introduce
    \[
    a := \dim \mathfrak{z}(\fg), \qquad b:= \dim \mathfrak{t}', \qquad r := a+b = \mathrm{rank}(G).
    \]
    The set $\mc{A}$ of \emph{analytically integral weights} corresponds to
    \[
    \mc{A} := \{\alpha \in (i\mathfrak{t})^* ~:~ \alpha(H) \in 2\pi i \Z, \forall H \in \mathfrak{t}, \exp(H)=1\}.
    \]
    Let $\{\lambda_1,\dots,\lambda_a\}$ be a system of generators of $\mathfrak{z}(\mathfrak{g})$, chosen such that any $\alpha \in (i\mathfrak{z}(\mathfrak{g}))^{*} \cap \mathcal{A}$ can be written as
    \[
    \alpha = \sum_{i=1}^{a} k_i \lambda_i , \qquad k_i \in \mathbb{Z}.
    \]
    Let $\{\lambda_{a+1},\dots,\lambda_r\}$ be a system of generators of
    \(\mathfrak{a}^{\mathrm{ss}}_{+} \cap \mathcal{A}\), chosen such that any
    \(\alpha \in \mathfrak{a}^{\mathrm{ss}}_{+} \cap \mathcal{A}\) can be written as
    \[
    \alpha = \sum_{i=a+1}^{d} k_i \lambda_i , \qquad k_i \in \mathbb{Z}_{\geq 0}.
    \]
    Finally, we can consider the following surjective map
    \[
    \phi \colon \mathbb{Z}^{a} \times \mathbb{Z}_{\ge 0}^{r-a} \longrightarrow \mathfrak{a}_{+} \cap \mathcal{A},
    \qquad
    \bk \longmapsto \sum_{i=1}^{r} k_i \lambda_i.
    \]
    The map $\phi$ may fail to be injective. To any element $\phi(\bk)$, one can associate the (unique) irreducible representation of $G$ with highest weight is $\phi(\bk)$. Hence, $\widehat{G}$ is parametrized by $\mathbb{Z}^{a} \times \mathbb{Z}_{\ge 0}^{b}/\sim$, where $\bk \sim \bk'$ if and only if $\phi(\bk)=\phi(\bk')$.
    
    The homogeneous space $G/T$ is called the \emph{flag manifold} associated with $G$; it is equipped with a natural complex structure. The Borel--Weil theorem asserts that the representation with highest weight $\phi(\bk)$ can be realized as
    \[
    H^0(G/T, \mathbf{J}^{\otimes \mathbf{k}}), \qquad \mathbf{J}^{\otimes \mathbf{k}} := J_1^{\otimes k_1} \otimes \dotsb \otimes J_r^{\otimes k_r}, 
    \]
    where $J_i \to G/T$ is a holomorphic Hermitian line bundle over $G/T$. For $i=1,\dotsc,a$, $J_i \to G/T$ is trivial. The left action of $G$ on $G/T$ lifts naturally to a fiberwise (isometric) action on $H^0(G/T,\mathbf{J}^{\otimes \mathbf{k}})$. This left action corresponds precisely to the irreducible representation of highest weight $\phi(\bk)$.
    
    We emphasize that if there exists $k_i \neq 0$ for some $i \in \{a+1,\dotsc,r\}$, then $\mathbf{J}^{\otimes \mathbf{k}} \to G/T$ is not topologically trivial (see \cite[Proposition 2.1.8]{Cekic-Lefeuvre-24}). If $k_{a+1}=\dotsb=k_r=0$, then $\mathbf{J}^{\otimes \mathbf{k}} \to G/T$ is trivial and the space of holomorphic sections $H^0(G/T, \mathbf{J}^{\otimes \mathbf{k}})$ reduces to $\C$.

    \subsubsection{The flag bundle and Borel--Weil calculus} The previous construction can be extended to the principal bundle $P_0 \to M_0$ as each fiber is isomorphic to $G$. We let $F_0 := P_0/T$ be the flag bundle over $M_0$ and $\mathbf{L}^{\otimes \bk} \to F_0$ be the line bundles obtained by the above construction. Each fiber $P_0(x)/T \simeq G/T$ for $x \in M$ is a complex manifold. In addition, the line bundles $\mathbf{L}^{\otimes \bk} \to F_0$ are holomorphic in restriction to each fiber. We can thus define $C^\infty_{\hol}(F_0,\Lk)$ as the space of \emph{fiberwise holomorphic} sections on $F_0$ with values in $\Lk$, that is such that the restriction to each fiber $P_0(x)/T$ is holomorphic. 
    There is also an $L^2$-orthogonal projector onto the space of fiberwise holomorphic sections $C^\infty_{\hol}(F_0,\Lk)$
    \[
    \Pi_{\bk} \colon C^\infty(F_0,\Lk) \to C^\infty_{\hol}(F_0,\Lk), \qquad L^2(F_0,\Lk) \to L^2_{\hol}(F_0,\Lk).
    \]
    
    The Fourier transform \eqref{equation:fourier-decomposition} should then be seen as a map
    \begin{equation}
        \label{equation:fourier-group}
    \mc{F} \colon C^\infty(P_0) \to \bigoplus_{\mathbb{Z}^{a} \times \mathbb{Z}_{\ge 0}^{b}/\sim} C^\infty_{\hol}(F_0,\Lk)^{\oplus d_{\bk}},
    \end{equation}
    where $d_{\bk}$ denotes the dimension of the corresponding irreducible representation with highest weight $\phi(\bk)$. It is an isometry for the $L^2$-scalar product.
    
    We now further assume that $P_0$ is equipped with a $G$-equivariant connection. It determines a horizontal subbundle $\HH_{P_{0}}\subset TP_{0}$, and an Ehresmann connection on $p_{F_{0}} \colon F_{0} \to M_{0}$, that is, a splitting
    \[
    TF_0 = \HH_{F_{0}} \oplus \V_{F_{0}},
    \]
    where $\V_{F_{0}} = \ker dp_{F_{0}}$ is the vertical fiber, and $\HH_{F_{0}}$ is the horizontal space provided by the connection. We also introduce $\HH_{F_{0}}^{*},\V_{F_{0}}^{*} \subset T^{*}F_{0}$ such that $\HH_{F_{0}}^{*}(\V_{F_{0}})=0=\V_{F_{0}}^{*}(\HH_{F_{0}})$. The map $dp_{F_0} \colon \HH_{F_{0}} \to TM_0$ is an isomorphism and therefore $dp_{F_0}^\top \colon T^{*}M_{0} \to \HH_{F_{0}}^{*}$ is an isomorphism as well.
    
    In addition, the $G$-equivariant connection on $P_{0}$ also induces a connection on each line bundle $\Lk \to F_{0}$. We are therefore in the framework of twisted quantization where one can consider the class $\Psi^m_{h,\bk}(F_{0},\mathbf{L})$ of twisted semiclassical pseudodifferential operators originally introduced by Charles \cite{Charles-00} for tensor powers of a single line bundle. See also \cite[Section 3.2]{Cekic-Lefeuvre-24} for a complete discussion. An element $\mathbf{A} \in \Psi^m_{h,\bk}(F_0,\mathbf{L})$ is a \emph{family} of operators 
    \[
    \mathbf{A}_{h,\bk} \colon C^\infty(F_0, \Lk) \to C^\infty(F_0,\Lk)
    \]
    defined for $h > 0$ and $\bk \in \Z^{a} \times \Z^b_{\geq 0}$ such that $h|\bk|\leq 1$, with the following property. For any contractible open subset $U \subset F_0$, if $s_i \in C^\infty(F_0,L_i)$ are local trivializing sections of pointwise norm $|s_{i}| = 1$, then there exists a family of standard $h$-semiclassical pseudodifferential operators $A_{h,\bk} \in \Psi^m_h(U)$ such that for all $f \in C^\infty_{\comp}(U)$:
    \[
    \mathbf{A}_{h,\bk}(f \mathbf{s}^{\bk}) = A_{h,\bk}(f) \mathbf{s}^{\bk},
    \]
    where equality holds on $U$.
    
    To introduce the Borel--Weil calculus, we need a preliminary notion:
    
    \begin{definition}[Admissible operators]
    Let $\mathbf{A} \in \Psi^m_{h,\bk}(F_0,\mathbf{L})$. The operator is \emph{admissible} if $[\mathbf{A},\Pi_{\bk}] \in h^\infty \Psi^{-\infty}_{h,\bk}(F_0,\mathbf{L})$.
    \end{definition}
    
    This means that the family of operators $\mathbf{A}_{h,\bk}$ preserve fiberwise holomorphic sections modulo negligible remainders. This leads to the following definition:
    
    \begin{definition}[Borel--Weil operators]
        For $m \in \R$, the class of Borel--Weil semiclassical operators of order $m$ on $P_0$ is defined as
        \[
        \Psi^m_{h,\mathrm{BW}}(P_0) := \{(\Pi_{\bk} \mathbf{A} \Pi_{\bk})_{h > 0, \bk \in \Z^a \times \Z^b_{\geq 0}, h|\bk|\leq 1} : \mathbf{A} \in \Psi^m_{h,\bk}(F_0,\mathbf{L}) \text{ is admissible}\}.
        \]
    \end{definition}

    It can be verified that $\Psi^\bullet_{h,\mathrm{BW}}(P_0)$ is indeed an algebra, see \cite[Lemma 3.3.11]{Cekic-Lefeuvre-24}. All $G$-equivariant differential operators naturally belong to $\Psi^m_{h,\mathrm{BW}}(P_0)$. The phase space corresponding to this quantization is $\HH_{F_{0}}^* \subset T^*F_0$. Observe that when the bundle $P_0 \to M_0$ is trivial, that is $P_0 \simeq M_0 \times G$, the phase space is isomorphic to $\HH_{F_{0}}^* \simeq T^*M_0 \times G/T$. One should think of operators in the Borel--Weil calculus as standard semiclassical pseudodifferential operators in the base variable (i.e. on $M_0$) \emph{with values in} Toeplitz operators on the complex manifold $G/T$. This statement can be made precise, see \cite[Theorem 3.3.4]{Cekic-Lefeuvre-24}.
    
    To any $\mathbf{A} \in \Psi^m_{h,\mathrm{BW}}(P_0)$, one can associate a principal symbol $\sigma_{\mathbf{A}} \in S^m(\HH_{F_{0}}^*)$. This is nothing but the restriction of the principal symbol of $\mathbf{A}$ to $\HH_{F_{0}}^*$, where $\mathbf{A}$ is seen as an operator in the twisted algebra $\Psi^m_{h,\bk}(F_0,\mathbf{L})$. Conversely, given $a \in S^m(\HH_{F_{0}}^*)$, there is a well-defined quantization procedure $\mathbf{A} := \Op_h^{\mathrm{BW}}(a) \in \Psi^m_{h,\mathrm{BW}}(P_0)$ constructing an operator in the Borel--Weil calculus such that $\sigma_{\mathbf{A}} = a$. All standard semiclassical results hold for Borel--Weil operators (invertibility of elliptic operators, propagation of singularities, Calder\'on---Vaillancourt, etc.), see \cite[Section 3.3]{Cekic-Lefeuvre-24} for further discussion. Finally, we note that the above construction can be carried out more generally for operators acting on sections of a vector bundle, see \cite[Remark 3.3.17]{Cekic-Lefeuvre-24}.

    \subsection{Connections, curvature and holonomy} \label{subsection:curvatures}
    
    Let $\nabla^{P_{0}}$ be the fixed $G$-equivariant connection on $P_{0} \to M_{0}$, and let $\mathrm{Ad}(P_{0}):= P_{0}\times_{\mathrm{Ad}}\mathfrak{g} \to M_{0}$ be the adjoint bundle associated to the adjoint representation $\mathrm{Ad} \colon G \to \mathrm{End}(\mathfrak{g})$. Denote by $F_{\mathrm{Ad}(P_{0})} \in C^{\infty} (M_{0}, \Lambda^{2} T^{*} M_{0} \otimes \mathrm{Ad}(P_{0}))$ the curvature of $\nabla^{P_{0}}$ on $P_{0}$. 

    We proceed to define certain differential operators that will appear in our proofs. Since each fiber of $F_{0}\to M_{0}$ is a complex manifold, the vertical tangent bundle admits the decomposition $\mathbb{V}_{F_{0}}\otimes \C =\mathbb{V}_{F_{0}}^{1,0} \oplus \mathbb{V}_{F_{0}}^{0,1}$. The complexified vertical cotangent bundle decomposes as $\mathbb{V}_{F_{0}}^{*}\otimes \C =(\mathbb{V}_{F_{0}}^{1,0})^{*} \oplus (\mathbb{V}_{F_{0}}^{0,1})^{*}$. Since $\mathbf{L}^{\otimes\bk}\to F_{0}$ is Hermitian and fiberwise holomorphic, there is a unique fiberwise partial Chern connection
    \begin{equation} \label{eq:chern}
        D_{\bk}^{\mathrm{Chern}} \colon C^{\infty}(F_{0},\mathbf{L}^{\otimes\bk}) \to C^{\infty}(F_{0}, \mathbf{L}^{\otimes\bk} \otimes \mathbb{V}_{F_{0}}^{*}).
    \end{equation}
    After complexification, it decomposes according to vertical type as $D_{\bk}^{\mathrm{Chern}} = \partial_{\bk} + \overline{\partial}_{\bk}$, where 
    \[
    \partial_{\mathbf{k}} \colon C^{\infty}(F_{0},\mathbf{L}^{\otimes\bk}) \to C^{\infty}( F_{0}, \mathbf{L}^{\otimes\bk} \otimes (\mathbb{V}_{F_{0}}^{1,0})^{*}), \qquad \overline{\partial}_{\bk} \colon C^{\infty}(F_{0},\mathbf{L}^{\otimes\bk}) \to C^{\infty}( F_{0}, \mathbf{L}^{\otimes\bk} \otimes (\mathbb{V}_{F_{0}}^{0,1})^{*} ).
    \]
    The operator $\overline{\partial}_{\mathbf{k}}$ is precisely the one defining the fiberwise holomorphic structure. Thus,
    \[
    C_{\mathrm{hol}}^{\infty}(F_{0},\mathbf{L}^{\otimes\bk}) = C^{\infty}(F_{0},\mathbf{L}^{\otimes \bk}) \cap \ker\overline{\partial}_{\mathbf{k}}.
    \]
    More explicitly, if $Z=Z^{1,0}+Z^{0,1}$ is a complexified vertical vector field, then $(D_{\bk}^{\mathrm{Chern}})_{Z} = (\partial_{\bk})_{Z^{1,0}} + (\overline{\partial}_{\bk})_{Z^{0,1}}$. In particular, if $u$ is fiberwise holomorphic, then $(D_{\bk}^{\mathrm{Chern}})_{Z}u = (\partial_{\bk})_{Z^{1,0}}u$.
    
    The parallel transport on $F_{0}$ can be lifted to a fiberwise (i.e. preserving the fibers of $\mathbf{L}^{\otimes \bk}$) parallel transport $\mathbf{L}^{\otimes \bk} \to F$ along horizontal directions, for all $\bk \in \widehat{G}$. Furthermore, it can be shown that it respects the fiberwise holomorphic structures \cite[Proposition~2.2.6]{Cekic-Lefeuvre-24}. This allows to define a partial horizontal connection
    \[
    \nabla_{\bk}^{\mathbb{H}_{F_{0}}} \colon C^{\infty} (F_{0}, \mathbf{L}^{\otimes \bk}) \to C^{\infty} (F_{0}, \mathbf{L}^{\otimes \bk} \otimes \mathbb{H}_{F_{0}}^{*})
    \]
    on $\mathbf{L}^{\otimes \bk} \to F_{0}$ as follows. Given a vector field $X \in C^{\infty}(M_{0}, T M_{0})$, let $\varphi$ be the flow generated on $M_{0}$ and $\psi^{F_{0}}$ be the flow generated by its horizontal lift $X^{\mathbb{H}_{F_{0}}}$ to $F_{0}$. Note that $\psi_t^{F_{0}}(x, w T)=(\varphi_t x, \tau_{\gamma(t)}^{F_{0}}(w T))$, where $\gamma(t)$ denotes the flow segment $(\varphi_{s} x)_{s \in[0, t]}$. We then set for $s \in C^{\infty}(F_{0}, \mathbf{L}^{\otimes \bk})$,
    \[
    (\nabla_{\bk}^{\HH_{F_{0}}})_{X^{\mathbb{H}_{F_{0}}}} s(w T):=\partial_{t}|_{t=0}(\tau_{\gamma(t)}^{\bk})^{-1}(s(\psi_t^{F_{0}}(w T))), \quad w T \in F_{0}.
    \]
    Combining this with the Chern connection, one obtains 
    a global connection
    \begin{equation} \label{eq:globa-connection}
        \overline{\nabla}_{\mathbf{k}}:=D_{\mathbf{k}}^{\text{Chern }}+\nabla_{\mathbf{k}}^{\HH_{F_{0}}} \colon C^{\infty}(F_{0}, \mathbf{L}^{\otimes \bk}) \to C^{\infty}(F_{0}, \mathbf{L}^{\otimes \bk} \otimes T^{*} F_{0}).
    \end{equation}
    Denote by $\mathbf{F}_{\overline{\nabla}}:=(F_{\overline{\nabla}_1}, \ldots, F_{\overline{\nabla}_d}) \in C^{\infty}(F_{0}, \Lambda^2 T^* F_{0})^{\oplus d}$ the vector of curvatures of the respective line bundles $L_{1}, \ldots, L_{d} \to F_{0}$. By construction, the curvature $F_{\overline{\nabla}_{\mathbf{k}}}$ of $\overline{\nabla}_{\mathbf{k}}$ satisfies 
    \begin{equation} \label{eq:bar-curvature}
        F_{\overline{\nabla}_{\bk}}=\bk \cdot \mathbf{F}_{\overline{\nabla}}=\sum_{i=1}^d k_{i} F_{\overline{\nabla}_{i}}.
    \end{equation}
    Given $x \in M_{0}$, the curvature of $\nabla^{P_{0}}$ is \emph{nondegenerate at $x$} if the following holds:
    \[
    \mathrm{Span} \{ F_{\operatorname{Ad}(P_{0})}(x)(X, Y) \mid X, Y \in T_x M_{0} \}=\mathrm{Ad}((P_{0})_{x}) .
    \]
    The curvature is globally nondegenerate if it is nondegenerate at every $x \in M_{0}$. 

    The \emph{boundary at infinity} of the Weyl chamber, $\partial_{\infty} \mathfrak{a}_{+}$, is the set of all possible limits $\bk /|\bk|$ as $|\bk| \to \infty$ and $\bk \in \widehat{G}$. It can be shown that $\partial_{\infty}\mathfrak{a}_{+} \simeq \mathbb{S}^{d-1} \cap (\R^{a} \times \R_{+}^{b})$, hence, the boundary at infinity is compact.

    For $\mathbf{l}=(\ell_{1},\ldots,\ell_{d}) \in \R^{a} \times \R_{+}^{b}$, define 
    \[
    \mathbf{l} \cdot \mathbf{F}_{\overline{\nabla}}=\sum_{i=1}^d \ell_{i} F_{\overline{\nabla}_{i}}.
    \]
    For $\mathbf{l} \in \partial_{\infty} \mathfrak{a}_{+} \simeq \mathbb{S}^{d-1} \cap (\R^{a} \times \R_{+}^{b})$, define the following number:
    \[
    F_{\min}(\mathbf{l}):=\min _{wT \in F_{0}} \max _{\substack{X, Y \in T_{p_{F_{0}}(wT)} M_{0} \\|X|=|Y|=1}}-i \times \mathbf{l} \cdot \mathbf{F}_{\overline{\nabla}}(wT)(X^{\mathbb{H}_{F_{0}}}, Y^{\mathbb{H}_{F_{0}}}) \geq 0,
    \]
    where $X^{\mathbb{H}_{F_{0}}}, Y^{\mathbb{H}_{F_{0}}}$ are the horizontal lifts of $X, Y$ respectively to $F_{0}$. Finally, define
    \[
    F_{\min }:=\min _{\mathbf{l} \in \partial_{\infty} \mathfrak{a}_{+}} F_{\min }(\mathbf{l}) \geq 0 .
    \]
    Since $\partial_{\infty} \mathfrak{a}_{+}$ is compact, this is indeed a minimum. Finally, by \cite[Lemma~2.2.14]{Cekic-Lefeuvre-24} the curvature $F_{\operatorname{Ad}(P)}$ is globally nondegenerate if and only if $F_{\min }>0$.

    As we mentioned in the introduction, we can pullback $\nabla^{P_{0}}$ to $P$ to obtain a connection $\nabla=\pi^{*}\nabla^{P_{0}}$. Fix $u \in P$ over $x_{0} \in M_{0}$. For every loop $\gamma$ at $x_{0}$, we obtain a unique element $h_{\gamma} \in G \times \Z^{d}$, such that the lift starting at $u$ ends at $h_{\gamma}.u$. We define the \emph{holonomy group} of $(P\to M_{0},\nabla)$ by 
    \[
    \mathrm{Hol}_{u}(P \to M_{0},\nabla)=\{h_{\gamma}:[\gamma]\in \pi_{1}(M_{0},x_{0})\} \leq G \times \Z^{d}.
    \]
    When there is no risk of confusion, we will simply write $\mathrm{Hol}(P,\nabla)$. Note that we can extend $\rho$ from $\pi_{1}(M_{0})$ to $\Omega_{x_{0}}M_{0}$, the set of loops based at $x_{0}$, by $\rho(\gamma)=\rho([\gamma])$. Writing $h_{\gamma}=(g_{\gamma},\rho(\gamma))$, we then have
    \begin{equation} \label{eq:holonomy}
        \mathrm{Hol}(P,\nabla)=\{(g_{\gamma},\rho(\gamma)):\gamma \in \Omega_{x_{0}}M_{0}\}.
    \end{equation}
    
\section{Abelian extensions of the Laplacian}

\label{section:abelian-extension-laplacian}

In this section, we prove the decay of correlations for Abelian extensions of the heat operator. 

\subsection{Preliminaries}

Recall that $\Delta_M$ denotes the Laplacian on $M$ and that we denote by $\Delta_{M_0}$ the Laplacian on $M_0$. We will assume that the volume of $M_{0}$ is normalized to be equal to $1$.

\subsubsection{Floquet reduction}\label{sssection:conjugacy} Using the isomorphism
\begin{equation}
    \label{equation:iso}
C^\infty(M_0) \to C^\infty_{\btheta}(M), \qquad f \mapsto \pi^*f \cdot s_{\btheta},
\end{equation}
introduced in \S\ref{ssection:floquet-theory}, we can compute, for $F_{\btheta} = \pi^*f \cdot s_{\btheta}$:
\[
\Delta_{M}F_{\btheta}=\Delta_{M}(\pi^{*}f)s_{\btheta}+2ds_{\btheta}\cdot d(\pi^{*}f)+(\pi^{*}f)\Delta_{M}s_{\btheta}
\]
and using Proposition \ref{prop:eta-theta}
\[
\Delta_{M}s_{\btheta}=d^{*}(is_{\btheta}\pi^{*}\eta_{\btheta})=i\langle ds_{\btheta},\pi^{*}\eta_{\btheta} \rangle+is_{\btheta}d^{*}(\pi^{*}\eta_{\btheta})=is_{\btheta}d^{*}(\pi^{*}\eta_{\btheta})+s_{\btheta}|\pi^{*}\eta_{\btheta}|^{2}.
\]
Hence, 
\begin{equation}\label{eq:conjugacy-operators}
    \Delta_{M}F_{\btheta}=\pi^{*}(\Delta_{M_{0}}f-2i\langle \eta_{\btheta},df\rangle +i(d^{*}\eta_{\btheta})f+|\eta_{\btheta}|^{2}f)s_{\btheta}.
\end{equation}
In other words, the induced operator $\Delta
_M \colon C^\infty_{\btheta}(M) \to C^\infty_{\btheta}(M)$ is conjugate \emph{via} \eqref{equation:iso} to the operator
\[ 
\Delta_{\btheta} \colon C^{\infty}(M_{0}) \to C^{\infty}(M_{0}), \, \,\Delta_{\btheta}f=\Delta_{M_{0}}f-2i\langle \eta_{\btheta},df \rangle+id^{*}\eta_{\btheta}f+|\eta_{\btheta}|^{2}f=(d+i\eta_{\btheta})^{*}(d+i\eta_{\btheta})f,
\]
which is a \emph{magnetic Laplacian} for each $\btheta$.
We also record a consequence of \eqref{eq:conjugacy-operators} on the corresponding propagators for future use, which follows from functional calculus:
\begin{equation}\label{eq:conjugacy-propagators}
    e^{-t\Delta_M} F_{\btheta} = \pi^*(e^{-t \Delta_{\btheta}} f) s_{\btheta},\quad t \geq 0.
\end{equation}

Making a different choice of $s_{\btheta}$ amounts to considering $s_{\btheta}' = \pi^*q \cdot s_{\btheta}$ for some $q \in C^\infty(M_0)$ of pointwise unit norm. The corresponding $1$-form then satisfies $i\eta_{\btheta}' = i \eta_{\btheta} + \frac{dq}{q}$, while the operator $\Delta_{\btheta}'$ becomes
\begin{equation}\label{eq:conjugated}
    \Delta_{\btheta}' = q^{-1}\Delta_{\btheta}q,
\end{equation}
i.e. $\Delta_{\btheta}'$ is conjugate to $\Delta_{\btheta}$ by the multiplication operator by the smooth function $q$. 

Since $\eta_{\btheta}$ is real, then $\Delta_{\btheta}$ is self-adjoint, nonnegative, and has a discrete spectrum $0\leq \lambda_{0}(\btheta) \leq \lambda_{1}(\btheta) \leq \dots$. The resolvent
\[
R(\btheta,z):=(\Delta_{\btheta}-z)^{-1}=\int_{0}^{\infty} e^{tz}e^{-t\Delta_{\btheta}}~ \dd t,
\]
is therefore well-defined, bounded and holomorphic on $L^{2}(M_{0})$ for $\Re(z) < 0$, since we have $\|e^{-t\Delta_{\btheta}}\|_{L^{2} \to L^{2}} \leq 1$. Equivalently, $R(\btheta,z)$ can be seen as an operator acting on $C^\infty(M_0,L_{\btheta})$ using the correspondence from Proposition \ref{prop:fourier-theory}. For simplicity, we will always consider these operators as acting on functions on $M_0$ rather than on sections of the line bundle $L_{\btheta} \to M_0$.

\subsubsection{Leading eigenvalue}\label{ssec:leading-resonance}

Since $\btheta \mapsto \eta_{\btheta}$ is smooth, it follows that $\btheta \mapsto \Delta_{\btheta}$ is smooth too. Then, $\btheta \mapsto R(\btheta,z)$ is smooth for $z$ outside the spectrum. It follows that the projector 
\begin{equation}\label{eq:projector-near-the-origin}
    \Pi_{0}(\btheta) = -\frac{1}{2\pi i}\oint_{\gamma_{0}} R(\btheta,z) \dd z,
\end{equation}
depends smoothly on $\btheta$, where $\gamma_{0}$ is a loop enclosing $\lambda_{0}(\btheta)$ only in the spectrum of $\Delta_{\btheta}$. Since $\lambda_{0}(\mathbf{0})=0$ is simple, $\Pi_{0}(\btheta)$ has rank one for $\btheta$ small. It follows that $\btheta \mapsto \lambda_{0}(\btheta)=\Tr(\Delta_{\btheta}\Pi_{0}(\btheta))$ is smooth.

Note also that we can write $\Pi_{0}(\btheta)= \langle \bullet, \varphi_{0}^{\btheta}\rangle \varphi_{0}^{\btheta}$, where $\varphi_{0}^{\btheta}$ is the eigenvector corresponding to $\lambda_{0}(\btheta)$, and $\|\varphi_{0}^{\btheta}\|_{L^{2}}= 1$. 

We now compute the first and second order derivatives of the leading eigenvalue at $\btheta = \mathbf{0}$:

\begin{lemma} \label{lemma:resonance-non-degenerate}
The following holds, for any $v \in T_{\mathbf{0}} \mathrm{U}(1)^d \simeq \mathbb{R}^d$
\[
    \lambda'_{0}(\mathbf{0})(v) = 0, \qquad \lambda''_{0}(\mathbf{0})(v, v) =2\|D_{v}\eta_{\mathbf{0}}\|_{L^{2}(M_{0})}^{2}.
\]
Moreover, we have $\lambda''_{0}(\mathbf{0})$ is positive-definite.    
\end{lemma}

Here $\lambda'_{0}(\mathbf{0})(v) = D_v \lambda_{0}(\mathbf{0})$ denotes the directional derivative at $\btheta = \mathbf{0}$ in the direction $v$, and $\lambda''_{0}(\mathbf{0})(v, v) = D^2_v \lambda_{0}(\mathbf{0})$ denotes the Hessian also at $\btheta = \mathbf{0}$ in the direction of $(v, v)$.

\begin{proof}
Let $u_{\btheta}:=\varphi_{0}^{\btheta}$. Since $\vol(M_{0})=1$, we may choose $u_{\mathbf{0}}=1$. We begin with the equality
\[
(\Delta_{\btheta}-\lambda_{0}(\btheta))u_{\btheta}=\Delta_{M_{0}}u_{\btheta}-2i\langle\eta_{\btheta} ,du_{\btheta}\rangle+id^{*}\eta_{\btheta}u_{\btheta}+|\eta_{\btheta}|^{2}u_{\btheta}-\lambda_{0}(\btheta)u_{\btheta}=0.
\]
We will take the directional derivatives $D_v$ in the direction of $v$ and denote the corresponding derivatives at $\btheta$ using the dot notation; note that all objects vary smoothly with respect to $\btheta$. Hence, differentiating we find:
\begin{equation} \label{eq:diff1}
    (-2i\langle\dot{\eta}_{\btheta},d\bullet \rangle+id^{*}\dot{\eta}_{\btheta}+2\langle \dot{\eta}_{\btheta},\eta_{\btheta}\rangle -\dot{\lambda}_{0}(\btheta) )u_{\btheta}  + (\Delta_{M_{0}}-2i\langle\eta_{\btheta},d\bullet\rangle+id^{*}\eta_{\btheta}+|\eta_{\btheta}|^{2}-\lambda_{0}(\btheta))\dot{u}_{\btheta}=0,
\end{equation}
Using that $u_{\mathbf{0}}=1$, $\lambda_{0}(\mathbf{0})=0$, and that $\eta_{\mathbf{0}}=0$, we obtain that 
\begin{equation} \label{eq:dotu0}
    -id^{*}\dot{\eta}_{\mathbf{0}}u_{\mathbf{0}}+\dot{\lambda}_{0}(\mathbf{0})=\Delta_{M_{0}}\dot{u}_{\mathbf{0}}.
\end{equation}
Integrating over $M_{0}$ and using integration by parts, we obtain $\dot{\lambda}_{0}(\mathbf{0})=0$.

Now we differentiate \eqref{eq:diff1} with respect to $\btheta$ to obtain:
\begin{equation} \label{eq:diff2}
\begin{aligned}
    (-2i \langle\ddot{\eta}_{\btheta},d\bullet \rangle+id^{*}\ddot{\eta}_{\btheta}+2\langle \ddot{\eta}_{\btheta},\eta_{\btheta}\rangle +2|\dot{\eta}_{\btheta}|^{2}-\ddot{\lambda}_{0}(\btheta) )&u_{\btheta}\\
    +2(-2i\langle\dot{\eta}_{\btheta},d\bullet\rangle+id^{*}\dot{\eta}_{\btheta}+2\langle \dot{\eta}_{\btheta},\eta_{\btheta}\rangle -\dot{\lambda}_{0}(\btheta) )&\dot{u}_{\btheta} 
    \\
    +(\Delta_{M_{0}}-2i\langle\eta_{\btheta},d\bullet\rangle+id^{*}\eta_{\btheta}+|\eta_{\btheta}|^{2}-\lambda_{0}(\btheta))&\ddot{u}_{\btheta}=0.
\end{aligned}    
\end{equation}
Recall from Proposition \ref{prop:eta-theta} (3), $\eta_{\btheta}$ is harmonic, implying that $\ddot{\eta}_{\btheta}$ is harmonic too. In addition, since $\eta_{\btheta}$ satisfies \eqref{equation:eta-theta}, differentiating twice we conclude $\int_{\gamma}\ddot{\eta}_{\btheta}=0$, implying that (its cohomology class is zero and hence) $\ddot{\eta}_{\btheta}=0$. Using this, we see from \eqref{eq:dotu0} that $\dot{u}_{\mathbf{0}}$ is constant. Hence, taking $\btheta=0$ in \eqref{eq:diff2} we get:
\[ 
\ddot{\lambda}_{0}(\mathbf{0})=2|\dot{\eta}_{\mathbf{0}}|^{2}+\Delta_{M_{0}}\ddot{u}_{\mathbf{0}}.
\]
Integration over $M_{0}$ gives $\ddot{\lambda}_{0}(\mathbf{0})=2\| \dot{\eta}_{\mathbf{0}}\|_{L^{2}(M_{0})}^{2} \geq 0$. 

Finally, we show that the second derivative is positive definite. If for $v \neq 0$ we obtain that the derivative is zero, then $\dot{\eta}_{\mathbf{0}}=0$ and $\rho(\gamma) \cdot v=0$ for any $\gamma \in \pi_{1}(M_{0})$. Surjectivity of the representation implies $v=0$, finishing the proof.    
\end{proof}

\subsection{Proof of Theorem \ref{theorem:main1}} 

We begin by obtaining uniform bounds in $\btheta$ for the decay of correlations of the operators $\Delta_{\btheta}$.

\begin{lemma} \label{lemma:Delta_theta}
Let $U$ be a neighborhood of $\mathbf{0} \in \mathrm{U}(1)^d$ such that the family of eigenvalues $(\lambda_{0}(\btheta))_{\btheta \in U}$ near $z = 0$ exists and depends smoothly on $\btheta$, see \S \ref{ssec:leading-resonance}. Then, there exists $C > 0$ such that for all $f,g \in C^{\infty}(M_{0})$ and for all $t \geq 0$ we have:
\begin{enumerate}[label=(\roman*), itemsep=5pt]
    \item If $\btheta \in U$, we have
\[
\begin{split}
\left|\int_{M_{0}} e^{-t\Delta_{\btheta}}f \cdot \overline{g}~ \dd \vol_{M_0} - e^{-t\lambda_{0}(\btheta)}\langle \Pi_{0}(\btheta)f, g \rangle_{L^{2}(M_{0})}\right|\leq e^{-Ct} \|f\|_{L^{2}(M_{0})}\|g\|_{L^{2}({M_{0}})}.
\end{split}
\]
\item If $\btheta \in \mathrm{U}(1)^d \setminus U$, we have
\[
\begin{split}
\left|\int_{M_{0}} e^{-t\Delta_{\btheta}}f \cdot \overline{g}~ \dd \vol_{M_{0}}\right|\leq e^{-Ct}\|f\|_{L^{2}({M_{0}})}\|g\|_{L^{2}({M_{0}})}.
\end{split}
\]
\end{enumerate}
\end{lemma}

Before giving the proof, we show a short technical lemma.

\begin{lemma}\label{lemma:lambda-theta}
$\lambda_{0}(\btheta)=0$ if and only if $\btheta=0$.    
\end{lemma}

\begin{proof}
It is clear that if $\btheta=0$ then $\lambda_{0}(\btheta)=0$ since $\Delta_{\btheta}$ becomes the usual Laplace--Beltrami operator in that case. So, let us assume that $\lambda_{0}(\btheta)=0$. This implies the existence of $u \neq 0$ so that $(d+i \eta_{\btheta})u=0$. Using this and the fact that $\eta_{\btheta}$ is real, we obtain $d|u|^{2}=0$. Hence, $|u|$ is a positive constant. Define $v=u/|u| \in C^{\infty}(M_{0},S^{1})$. Then, $-dv/(iv)=\eta_{\btheta}$. This implies $\rho(\gamma) \cdot \btheta \in 2 \pi \Z$ for any $\gamma \in \pi_{1}(M_{0})$. The surjectivity of $\rho$ shows that $\btheta=\mathbf{0} \in \mathrm{U}(1)^d$.    
\end{proof}

\begin{proof}[Proof of Lemma \ref{lemma:Delta_theta}]
We first prove $(i)$ and argue for nonzero $\btheta \in U$. The claim for $\btheta = \mathbf{0}$ then follows by continuity. Recall that $\Delta_{\btheta}$ is self-adjoint on $L^{2}(M_{0})$. Furthermore, for each $\btheta \in U$, let $\{\varphi_{j}^{\btheta}\}_{j=0}^{\infty}$ be an orthonormal basis of $L^{2}(M_{0})$ given by eigenvectors of $\Delta_{\btheta}$, associated to the eigenvalues $0 \leq \lambda_{0}(\btheta) \leq \lambda_{1}(\btheta) \leq \dots$ Then
\[ 
|\langle e^{-t\Delta_{\btheta}}f,g \rangle-e^{-t\lambda_{0}(\btheta)}\langle \Pi_{0}(\btheta)f,g\rangle|=\left|\sum_{j \geq 1}e^{-t\lambda_{j}(\btheta)} \langle f,\varphi_{j}^{\btheta}\rangle \langle \varphi_{j}^{\btheta},g \rangle \right| \leq e^{-t\lambda_{1}(\btheta)}\|f\|_{L^{2}}\|g\|_{L^{2}}.
\]
Recall that $\lambda_{0}(\mathbf{0})=0$ is simple and $\lambda_{1}(\mathbf{0})>0$. Then, shrinking $U$ if necessary, we obtain the existence of a constant such that $0<C<\lambda_{1}(\btheta)$ for any $\btheta \in U$. This shows the first result. 

For the proof of (ii), by continuity and Lemma \ref{lemma:lambda-theta}, there exists $C>0$ with $\lambda_{0}(\btheta) \geq C$ for any $\btheta \in \mathrm{U}(1)^{d} \setminus U$. Then, it follows that
\[  
|\langle e^{-t\Delta_{\btheta}}f,g \rangle| \leq \|e^{-t\Delta_{\btheta}}f\|_{L^{2}} \|g\|_{L^{2}} \leq e^{-Ct}\|f\|_{L^{2}}\|g\|_{L^{2}}.
\]
\end{proof}

The last step before the proof of Theorem \ref{theorem:main1} will be to establish a bound on the spectral projector.

\begin{lemma} \label{lemma:regularity_Pi_theta}
Let $U \subset \mathrm{U}(1)^d$ be a good open neighborhood of the origin $\mathbf{0} \in \mathrm{U}(1)^d$ such that $(\lambda_{0}(\btheta))_{\btheta \in U}$ and $\Pi_{0}(\btheta)$ are well-defined, see \S \ref{ssec:leading-resonance}. Let $k \geq 0$ be an integer, and $s \geq 0$. Then, for all multi-indices $\alpha$ such that $|\alpha| \leq k$, there exists $C=C(s,k,U) > 0$ such that for all $f,g \in B^{s,k}(M)$
\[ \sup_{\btheta \in U}|\partial_{\btheta}^{\alpha} \langle \Pi_{0}(\btheta)f_{\btheta},g_{\btheta}\rangle_{L^{2}(M_{0})} | \leq C\|f\|_{B^{s, k}(M)}\|g\|_{B^{s, k}(M)}. \]
\end{lemma}

\begin{proof}
In the first place, thanks to \eqref{eq:projector-near-the-origin} observe that for any $j = 1, \dotsc, d$, we have
\[
\begin{split}
    \partial_{\theta_j}\Pi_{0}(\btheta)&=\frac{1}{2\pi i}\oint_{\gamma_{0}} R(\btheta,z) (\partial_{\theta_{j}}\Delta_{\btheta})R(\btheta,z) \dd z \\
    &=\frac{1}{2\pi i}\oint_{\gamma_{0}} R(\btheta,z)( -2i\langle\partial_{\theta_{j}}\eta_{\btheta},d\bullet \rangle+id^{*}\partial_{\theta_{j}}\eta_{\btheta}+2\langle \partial_{\theta_{j}}\eta_{\btheta},\eta_{\btheta}\rangle )R(\btheta,z) \dd z,
\end{split}
\]
and inductively we see that derivatives of the spectral projector involve derivatives of $-2i\langle\eta_{\btheta},d\bullet\rangle+id^{*}\eta_{\btheta}+|\eta_{\btheta}|^{2}$  and powers of the resolvent $R(\btheta,z)$ for $z \in \gamma_{0}$. Since $\gamma_{0}$ is disjoint from $\mathrm{spec}(\Delta_{\btheta})$ for $\btheta\in U$, there exists $\delta>0$ such that for $z \in \gamma_{0}$ and $\btheta \in U$ we have $\mathrm{dist}(z,\mathrm{spec}(\Delta_{\btheta}))\geq \delta$. Therefore, by the spectral theorem,
\begin{equation} \label{eq:spectral-bound-resolvent}
    \|R(\btheta,z)\|_{L^{2}\to L^{2}}\leq \delta^{-1}.
\end{equation}
Now let $v\in H^{s}(M_{0})$, and set $u=R(\btheta,z)v$. Then elliptic regularity together with \eqref{eq:spectral-bound-resolvent} give
\[
\|u\|_{H^{s+2}} \leq C_{s} ( \|v\|_{H^{s}} +\|u\|_{L^{2}} )
\leq C_{s}\|v\|_{H^{s}},
\]
i.e., $\|R(\btheta,z)\|_{H^{s} \to H^{s+2}} \leq C_{s}$. Hence, $H^{s+2} \hookrightarrow H^{s}$ gives $\|R(\btheta,z)\|_{H^{s} \to H^{s}} \leq C_{s}$. On the other hand, $\partial_{\theta_{j}}\Delta_{\btheta} \colon H^{s+1} \to H^{s}$ since it is an operator of order 1. Then, by induction, we see that $\|\partial^\alpha_{\btheta} \Pi_{0}(\btheta)\|_{H^{s} \to H^{s}} \leq C$ for $|\alpha| \leq k$ and $\btheta \in U$. Thus, 
\[
\begin{split}
    |\partial_{\btheta}^{\alpha} \langle \Pi_{0}(\btheta)f_{\btheta},g_{\btheta}\rangle_{L^{2}(M_{0})} | &\leq \sum_{\alpha_{1}+\alpha_{2}+\alpha_{3}=\alpha}C_{\alpha_{1},\alpha_{2},\alpha_{3}} |\langle \partial_{\btheta}^{\alpha_{1}}\Pi_{0}(\btheta) \partial_{\btheta}^{\alpha_{2}}f_{\btheta},\partial_{\btheta}^{\alpha_{3}}g_{\btheta}\rangle_{L^{2}(M_{0})}|  \\
    & \leq C\sum_{\alpha_{1}+\alpha_{2}+\alpha_{3}=\alpha}C_{\alpha_{1},\alpha_{2},\alpha_{3}} \| \partial_{\btheta}^{\alpha_{2}}f_{\btheta} \|_{H^{s}(M_{0})} \|\partial_{\btheta}^{\alpha_{3}}g_{\btheta}\|_{H^{s}(M_{0})} \\
    & \leq C\|f_{\bullet}\|_{C^{k}(U,H^{s}(M_{0}))} \|g_{\bullet}\|_{C^{k}(U,H^{s}(M_{0}))},
\end{split}
\]
where $C_{\alpha_1, \alpha_2, \alpha_3}$ denotes a multinomial coefficient. The result follows from Lemma \ref{lemma:spaces_theta_sobolev}.    
\end{proof}

\begin{proof}[Proof of Theorem \ref{theorem:main1}]
We are going to obtain an expansion for $\langle e^{-t\Delta_{M}}f,g\rangle$, then the result will follow from duality.

Let $U_0$ be a good neighborhood of $\mathbf{0} \in \mathrm{U}(1)^d$ such that the conclusion of Lemma \ref{lemma:Delta_theta}, Item (i) holds. Consider a finite covering of $\mathrm{U}(1)^d$ by good geodesically convex balls $(U_i)_{i = 0}^k$ and a partition of unity subordinate to this cover $(\chi_i)_{i = 0}^k$ with $\chi_0 = 1$ near the origin. Over each $U_i$ we may trivialize the line bundle $(L_{\btheta})_{\btheta \in U_i}$ smoothly by Proposition \ref{prop:eta-theta} (3), i.e., there is a smooth family of sections of $L_{\btheta}$ of pointwise unit norm $(s_{\btheta, i})_{\btheta \in U_i}$. We may use the isomorphism \eqref{equation:iso} for $\btheta \in U_i$ identifying sections of $L_{\btheta}$ with functions on $M_0$; the identification goes as $F_{\btheta} = s_{\btheta, i} \pi^*f_{\btheta, i}$ for a section $F_{\btheta}$ of $L_{\btheta}$. With this identification, $(e^{-t\Delta_M} f)_{\btheta, i} = e^{-t\Delta_{\btheta}} f_{\btheta, i}$ for each $i$, thanks to \eqref{eq:conjugacy-propagators} and Proposition \ref{prop:fourier-theory}. We therefore conclude from Parseval's decomposition \eqref{equation:parseval} that
\[
\begin{split}
    &\langle e^{-t\Delta_M}f,g \rangle_{L^{2}(M)} = \frac{1}{(2\pi)^{d}} \int_{\mathrm{U}(1)^d} \langle{e^{-t\Delta_{\btheta}}F_{\btheta}, G_{\btheta}}\rangle_{L^2(M_0, L_{\btheta})} \dd \btheta\\
    =& \frac{1}{(2\pi)^{d}}\int_{U_0} \chi_0(\btheta) e^{-t \lambda_{0}(\btheta)} \langle{\Pi_{0}(\btheta)f_{\btheta, 0}, g_{\btheta, 0}}\rangle_{L^2} \dd \btheta+\frac{1}{(2\pi)^{d}}\sum_{i = 1}^k \int_{U_i} \chi_i(\btheta) \langle e^{-t\Delta_{\btheta}}f_{\btheta, i}, g_{\btheta, i} \rangle_{L^{2}} \dd \btheta \\
    &+ \frac{1}{(2\pi)^{d}}\int_{U_0}\chi_0(\btheta) (\langle e^{-t\Delta_{\btheta}}f_{\btheta, 0}, g_{\btheta, 0} \rangle_{L^{2}} - e^{-t \lambda_{0}(\btheta)} \langle{\Pi_{0}(\btheta)f_{\btheta, 0}, g_{\btheta, 0}}\rangle_{L^2}) \dd \btheta.
\end{split}
\]
Call the first term $Q(t)$, and the rest $S(t)$. We begin by estimating $S(t)$. In light of Lemma \ref{lemma:Delta_theta}, Items (i) and (ii), there is $C > 0$ such that (note that technically we take the maximum of the constants that appear for each $U_i$, where $i = 0, \dotsc, k$)
\[ 
|S(t)| \leq e^{-Ct} \sum_{i = 0}^k \int_{\mathrm{U}(1)^d} \chi_i(\btheta) \|f_{\btheta, i}\|_{L^{2}} \|g_{\btheta, i}\|_{L^{2}} \dd \btheta \leq Ce^{-Ct}\|f\|_{L^{2}(M)} \|g\|_{L^{2}(M)}.
\]

We now derive an asymptotic expansion for $Q(t)$. For simplicity we drop the sub-index $0$ in what follows. Since $\lambda_{0}(\btheta)$ has a nondegenerate critical point at $\btheta = \mathbf{0}$ according to Lemma \ref{lemma:resonance-non-degenerate}, and since $\Re(\lambda_{0}(\btheta)) \geq 0$, we may apply the stationary phase lemma (see \cite[Lemma 7.7.5]{Hoermander-I-03}) with $i\lambda_{0}$, and so we have, for each integer $N \geq 1$:
\begin{equation}
\label{equation:tard}
\begin{split}
    t^{d/2} Q(t) &= \frac{1}{\sqrt{(2\pi)^{d}\det \lambda''_{0}(\mathbf{0})}} \sum_{j=0}^{N-1} t^{-j} L_j (\langle \Pi_{0}(\btheta)f_{\btheta},g_{\btheta} \rangle_{L^{2}(M_{0})})|_{\btheta = \mathbf{0}} + R_{N}(t),
\end{split}
\end{equation}
where $L_j$ is an elliptic differential operator of order $2j$ for $j \geq 1$ and $L_0 = \mathrm{Id}$. Then
\begin{equation}
\label{equation:expression-cj}
    C_j(f, g) := \frac{(2\pi)^{-d/2}}{\sqrt{\det (2\langle \partial_{\theta_{i}}\eta_{\btheta}|_{\btheta=\mathbf{0}},\partial_{\theta_{j}}\eta_{\btheta}|_{\btheta=\mathbf{0}} \rangle_{L^{2}(M_{0})})}} L_j \langle{\Pi_{0}(\btheta) f_{\btheta}, g_{\btheta}}\rangle|_{\btheta = \mathbf{0}} = \mc{O}(\|f\|_{B^{s, 2j}} \|g\|_{B^{s, 2j}}),
\end{equation}
where we used Lemma \ref{lemma:regularity_Pi_theta}. The remainder term is estimated using Lemma \ref{lemma:regularity_Pi_theta} once again:
\[ 
    |R_{N}(t)| \leq C t^{-N} \sum_{|\alpha| \leq 2N+d+1}  \sup_{\btheta \in U_0} |\partial_{\btheta}^{\alpha} \langle \Pi_{0}(\btheta)f_{\btheta},g_{\btheta}\rangle_{L^{2}} | \leq Ct^{-N}\|f\|_{B^{s,2N+d+1}(M)}\|g\|_{B^{s,2N+d+1}(M)}.
\]
Finally, one can apply the Schwartz kernel theorem to obtain the existence of $\mathcal{C}_{j}, \mathcal{R}_{N} \in \mathcal{D}'(M \times M)$ satisfying
\[  
\langle \mathcal{C}_{j}, f \otimes g \rangle=C_{j}(f,g), \quad \langle \mathcal{R}_{N}(t), f \otimes g \rangle=S(t,f,g)+R_{N}(t,f,g),
\]
finishing the proof.
\end{proof}

\subsection{Computation of \texorpdfstring{$C_1$}{C1}.} As in \cite[\S4.5]{Cekic-Lefeuvre-Munoz-Thon-26}, we now compute the first coefficient of the expansion. Continuing the discussion from the proof of Theorem \ref{theorem:main1}, and using its notation, according to \cite[Lemma 7.7.5]{Hoermander-I-03} for any $j$ we have 
\begin{equation}\label{eq:stat-phase-lemma}
    L_ju = \sum_{\nu - \mu = j} \sum_{2\nu \geq 3\mu \geq 0} i^{-j} 2^{-\nu} (\mu! \nu!)^{-1} \langle{(i\lambda''_{0}(\mathbf{0}))^{-1}D, D}\rangle^{\nu} (h^\mu u)(\mathbf{0}),\quad u \in C^\infty(U_0),
\end{equation}
is a differential operator of order $2j$ acting on $u$ at $\btheta = \mathbf{0}$, where $h$ vanishes to third order at $\mathbf{0}$ and is defined by 
\begin{equation} \label{eq:g-lambda}
    -ih(\btheta) = \lambda_{0}(\btheta) - \frac{1}{2} \langle{\lambda_{0}''(\mathbf{0}) \btheta, \btheta}\rangle.
\end{equation}
Here $D = -i \partial_{\btheta}$ denotes the gradient operator times $-i$. We have $L_0 = \mathrm{Id}$ is the identity operator, and from \eqref{eq:stat-phase-lemma} we read that
\begin{align}\label{eq:L1}
\begin{split}
    i L_1 u =& 2^{-1} \langle{(i\lambda_{0}''(\mathbf{0}))^{-1}D, D}\rangle u (\mathbf{0}) + 2^{-2} \frac{1}{2!} \langle{(i\lambda_{0}''(\mathbf{0}))^{-1}D, D}\rangle^2 (h u) (\mathbf{0})\\ 
    &+  2^{-3} \frac{1}{2! 3!}\langle{(i\lambda_{0}''(\mathbf{0}))^{-1}D, D}\rangle^3 (h^2) u(\mathbf{0}),\quad u \in C^\infty_{\comp}(U_0),
\end{split}
\end{align}

Write $u(\btheta) := \langle{\Pi_{0}(\btheta) f_{\btheta}, g_{\btheta}}\rangle_{L^2}$, and let us compute its derivatives as follows. Note that we can write, near $z=\lambda_{0}(\btheta)$
\[
R(\btheta,z)=R_{\hol}(\btheta,z)-\frac{\Pi_{0}(\btheta)}{z-\lambda_{0}(\btheta)},
\]
where $R_{\hol}$ is the holomorphic part. Using this together with \eqref{eq:projector-near-the-origin}, and computing as in the proof of Lemma \ref{lemma:regularity_Pi_theta}
\begin{align*}
    \partial_{\btheta_j} \Pi_{0}(\btheta) =& \frac{1}{2\pi i}\oint_{\gamma_0} \left(R_{\hol}(\btheta,z) - \frac{\Pi_{0}(\btheta)}{z - \lambda_{0}(\btheta)}\right) \partial_{\theta_{j}}\Delta_{\btheta} \left(R_{\hol}(\btheta,z) - \frac{\Pi_{0}(\btheta)}{z - \lambda_{0}(\btheta)}\right) \dd z\\
    =& -\Pi_{0}(\btheta) ( -2i\langle\partial_{\theta_{j}}\eta_{\btheta},d\bullet\rangle+id^{*}\partial_{\theta_{j}}\eta_{\btheta}+2\langle \partial_{\theta_{j}}\eta_{\btheta},\eta_{\btheta}\rangle ) R_{\hol}(\btheta,\lambda_{0}(\btheta)) \\
    &- R_{\hol}(\btheta,\lambda_{0}(\btheta)) ( -2i\langle\partial_{\theta_{j}}\eta_{\btheta},d\bullet\rangle+id^{*}\partial_{\theta_{j}}\eta_{\btheta}+2\langle \partial_{\theta_{j}}\eta_{\btheta},\eta_{\btheta}\rangle ) \Pi_{0}(\btheta),
\end{align*}
We therefore have for an index $j \in \{1, \dotsc, d\}$ that
\begin{align*}
    \partial_{\btheta_j} u =& -\langle(\Pi_{0}(\btheta) ( -2i\langle\partial_{\theta_{j}}\eta_{\btheta},d\bullet\rangle+id^{*}\partial_{\theta_{j}}\eta_{\btheta}+2\langle \partial_{\theta_{j}}\eta_{\btheta},\eta_{\btheta}\rangle ) R_{\hol}(\btheta,\lambda_{0}(\btheta))f_{\btheta},g_{\btheta}\rangle
    \\
    &- \langle R_{\hol}(\btheta,\lambda_{0}(\btheta)) ( -2i\langle\partial_{\theta_{j}}\eta_{\btheta},d\bullet\rangle+id^{*}\partial_{\theta_{j}}\eta_{\btheta}+2\langle \partial_{\theta_{j}}\eta_{\btheta},\eta_{\btheta}\rangle ) \Pi_{0}(\btheta) f_{\btheta}, g_{\btheta}\rangle \\ 
    &+ \langle{\Pi_{0}(\btheta) \partial_{\btheta_j}f_{\btheta}, g_{\btheta}}\rangle + \langle{\Pi_{0}(\btheta) f_{\btheta}, \partial_{\btheta_j} g_{\btheta}}\rangle\\ 
    =& \langle{\Pi_{0}(\btheta)P_{j} f_{\btheta}, g_{\btheta}}\rangle + \langle{\Pi_{0}(\btheta) f_{\btheta}, P_{j}g_{\btheta} }\rangle=\langle{\Pi_{0}(\btheta)P_{j} f_{\btheta}, g_{\btheta}}\rangle + \langle{f_{\btheta}, \Pi_{0}(\btheta)P_{j}g_{\btheta} }\rangle,
\end{align*}
where $P_{j}=\partial_{\btheta_j} - (-2i\langle\partial_{\theta_{j}}\eta_{\btheta},d\bullet\rangle+id^{*}\partial_{\theta_{j}}\eta_{\btheta}+2\langle \partial_{\theta_{j}}\eta_{\btheta},\eta_{\btheta}\rangle )R_{\hol}(\btheta,\lambda_{0}(\btheta))$, and we used that $\Delta_{\btheta}$ is self-adjoint and therefore $\partial_{\theta_{j}}\Delta_{\btheta}$, $\Pi_{0}(\btheta)$, and $R_{\hol}(\btheta,\lambda_{0}(\btheta))$ are self-adjoint too. Iterating this, we compute for $k \in \{1, \dotsc, d\}$ that 
\begin{align*}
    \partial_{\btheta_k \btheta_j}u = \langle{\Pi_{0}(\btheta)P_{k} P_{j} f_{\btheta}, g_{\btheta}}\rangle + \langle{\Pi_{0}(\btheta) P_{j} f_{\btheta}, P_{k} g_{\btheta}}\rangle + \langle{ P_{k} f_{\btheta},\Pi_{0}(\btheta) P_{j} g_{\btheta} }\rangle + \langle{f_{\btheta}, \Pi_{0}(\btheta) P_{k} P_{j} g_{\btheta} }\rangle.
\end{align*}
For simplicity, from now on we assume that 
\[
    \int_{M} f ~\dd \vol_M = \int_{M_0} f_\mathbf{0} ~\dd \vol_{M_0} = \int_{M} g ~\dd \vol_{M} = \int_{M_0} g_\mathbf{0} ~\dd \vol_{M_0} = 0.
\]
In particular $u(\mathbf{0}) = 0$, and from the expression for $\partial_{\btheta_j} u$ we have $\partial_{\btheta_j} u(\mathbf{0}) = 0$. This also gives that $\Pi_{0}(\mathbf{0})f_{\mathbf{0}}=\Pi_{0}(\mathbf{0})g_{\mathbf{0}}=0$. Also, using that $h$ vanishes to third order at $\mathbf{0}$ (which follows from its definition in \eqref{eq:g-lambda}), from \eqref{eq:L1} we therefore see that the last two terms vanish, and so 
\begin{align*}
    2L_1 u &= \sum_{j, k} \lambda_{0}''(\mathbf{0})^{-1}_{jk} \partial_{\btheta_k} \partial_{\btheta_j} u (\mathbf{0})\\
    &= \sum_{j, k} \lambda_{0}''(\mathbf{0})^{-1}_{jk} \left(\langle{\Pi_{0}(\btheta) P_{j} f_{\btheta}, P_{k} g_{\btheta}}\rangle|_{\btheta = \mathbf{0}} + \langle{\Pi_{0}(\btheta) P_{k} f_{\btheta}, P_{j} g_{\btheta} }\rangle|_{\btheta = \mathbf{0}}\right)\\
    &= 2\sum_{j, k} \lambda_{0}''(\mathbf{0})^{-1}_{jk} \Pi_{0}(\mathbf{0}) P_{j} f_{\btheta}|_{\btheta = \mathbf{0}} \overline{\Pi_{0}(\mathbf{0}) P_{k}g_{\btheta}|_{\btheta = \mathbf{0}}},
\end{align*}
where in the third equality we used that $\lambda''_{\mathbf{0}}$ is symmetric, and that $\vol_{M_0}$ is assumed to be a probability measure. We are left to compute $\Pi_{\mathbf{0}} P_j f_{\btheta}|_{\btheta = \mathbf{0}}$ and the symmetric term acting on $g_{\btheta}$.

Write $\partial_{\theta_{j}}\eta_{\btheta}|_{\btheta=\mathbf{0}}=\alpha_{j}$. Then, using that $\eta_{\btheta}$ is harmonic and $\eta_{\mathbf{0}}=0$ 
\begin{align*}
    \Pi_{0}(\mathbf{0}) P_j f_{\btheta}|_{\btheta = \mathbf{0}} &= \int_{M_0}\left(\partial_{\theta_j} f_{\btheta}|_{\btheta = \mathbf{0}} - 2i\langle\alpha_{j},dR_{\hol}(\mathbf{0},0) f_{\mathbf{0}}\rangle\right) \dd \vol_{M_0}\\
    &= -i \int_{M_0} (fH_j)_{\mathbf{0}} ~\dd \vol_{M_0} - 2i\int_{M_0}  (d^{*}\alpha_{j})R_{\hol}(\mathbf{0},0)f_{\mathbf{0}} ~\dd \vol_{M_0}\\
    &= -i\int_M fH_j  ~\dd \vol_{M} \\
    &=-i\int_{M}fd^{-1}(\pi^{*}\alpha_{j}) ~\dd \vol_{M},
\end{align*}
where in the second line we used Proposition \ref{prop:differentiation-btheta} as well as $\pi^* f_{\mathbf{0}}(x) = \sum_{\n} f(\tau_{\n}\widehat{x})$ for an arbitrary $\widehat{x} \in \pi^{-1}(x)$ and integration by parts, in the third we used the fact that $d^{*}\alpha_{j}=0$, and in the fourth one we used the definition of $H_{j}$ ($dH_{j}=\pi^{*}\alpha_{j}$, $H_{j}(\hat{x}_{0})=0$, see \eqref{eq:auxiliary-Hj}).

We now prove directly that this term is nonzero. Indeed, it is zero for all $f_{\mathrm{comp}}^{\infty}(M)$ that integrates to zero, then $H_{j}$ is constant. Then, $0=\pi^{*}\alpha_{j}$. Since $\pi \colon M \to M_{0}$ is a local diffeomorphism, we get that $\alpha_{j}=0$. However, this implies that $\rho \equiv 0$, contradicting the surjectivity of the representation.

Hence,
\begin{equation}\label{equation:c1}
\begin{split}
    C_1(f, g) = \frac{(2\pi)^{-\frac{d}{2}}}{\sqrt{\det \lambda_{0}''(\mathbf{0})}}\left(\sum_{j, k}\lambda''_{0}(\mathbf{0})^{-1}_{jk} \right. &\int_M f d^{-1} (\pi^* \partial_{\theta_{j}}\eta_{\btheta}|_{\btheta=\mathbf{0}}) \dd \vol_M\\ 
    & \left. \cdot\int_M \overline{g} d^{-1} (\pi^* \partial_{\theta_k} \eta_{\btheta}|_{\btheta=\mathbf{0}}) \dd \vol_M \right),
\end{split}
\end{equation}
whenever $\int_M f \dd \vol_M = \int_M g \dd \vol_M = 0$, where $\lambda_{0}''(\mathbf{0})=2\langle \partial_{\theta_{j}}\eta_{\btheta}|_{\btheta=\mathbf{0}},\partial_{\theta_{k}}\eta_{\btheta}|_{\btheta=\mathbf{0}} \rangle_{L^{2}(M_{0})}$. Finally, note that \eqref{equation:c1} is nontrivial. This follows from the fact that $\{\mathbf{1}_{M},H_{1},\ldots,H_{d}\}$ are linearly independent. To show the independence, let us assume by contradiction that this is not the case. Hence, we can write $c_{0}+\sum_{j}c_{j}H_{j}=0$. After differentiation, we get $\pi^{*}(\sum_{j}c_{j}\alpha_{j})=0$. It follows from the fact that $\pi$ is a local diffeomorphism that $\sum_{j}c_{j}\alpha_{j}=0$. Then, $c \cdot \rho(\gamma)=0$ for any $\gamma$, where $c=(c_{1},\ldots,c_{d})$. Surjectivity of $\rho$ implies that $c=0$, giving $c_{0}=0$ too. This shows the independence and finishes our computations.

\begin{remark}
The pairing $C_2$ can be computed similarly, assuming $\Pi_{0}(\mathbf{0}) P_j f_{\btheta}|_{\btheta = \mathbf{0}}=0$ for all indices $1 \leq j \leq d$ and that $f$ has zero average. The computation of $C_\ell$ for $\ell \geq 3$ is less clear, but a natural inductive assumption would be that $P_\alpha f_{\btheta}|_{\btheta = \mathbf{0}} = 0$ for all $|\alpha| \leq \ell - 1$.    
\end{remark}

\subsection{Nonvanishing of the bilinear forms} \label{subsection:non-vanishing}

Finally, we prove that the bilinear forms $C_j$ defined in \eqref{equation:expression-cj} do not vanish. The proof follows the same structure as \cite[\S4.6]{Cekic-Lefeuvre-Munoz-Thon-26}:

\begin{lemma}
    \label{lemma:cj-non-zero}
    For all $j \geq 1$, the bilinear form $C_j \colon C^\infty_{\comp}(M) \times  C^\infty_{\comp}(M) \to \C$ is nonzero.
\end{lemma}

\begin{proof}
    The operator $L_j$ in the expression \eqref{equation:expression-cj} of $C_j(f,g)$ is an elliptic differential operator of order $2j$ applied to $\langle\Pi_{0}(\btheta)f_{\btheta},g_{\btheta}\rangle$ (and evaluated at $\btheta=\mathbf{0}$). From \S\ref{ssec:leading-resonance}, we have $\langle \Pi_{0}(\btheta)f_{\btheta},g_{\btheta}\rangle=\langle f_{\btheta},\varphi_{0}^{\btheta} \rangle \langle \varphi_{0}^{\btheta},g_{\btheta}\rangle$.

    We consider $f,g \in C^\infty_{\comp}(M)$ with $\int_{M_0} f_{\boldsymbol{0}} \dd\vol_{M_0} =1 = \int_{M_0} g_{\boldsymbol{0}} \dd\vol_{M_0}$ and let $c(\btheta) := \langle f_{\btheta},\varphi_{0}^{\btheta}\rangle \langle \varphi_{0}^{\btheta},g_{\btheta}\rangle$. Similarly as in \S\ref{ssec:leading-resonance}, this is a smooth function of $\btheta$ near $\btheta=\boldsymbol{0}$. As $\varphi_{0}^{\mathbf{0}}=\mathbf{1}_{M_0}$, we find that $c(\boldsymbol{0})=1$. Also note that for $\bk \in \Z^d$, $\tau_{\bk}^*f \in C^\infty_{\comp}(M)$ and $(\tau_{\bk}^*f)_{\btheta} = e^{i\btheta\cdot\bk}f_{\btheta}$. Therefore $\langle\Pi_{0}(\btheta)(\tau_{\bk}^*f)_{\btheta},g_{\btheta}\rangle = e^{i\btheta\cdot\bk} c(\btheta)$ near $\btheta = \mathbf{0}$.
    
    We then compute asymptotically $L_j \langle\Pi_{0}(\btheta)(\tau_{\bk}^*f)_{\btheta},g_{\btheta}\rangle|_{\btheta=\boldsymbol{0}}$ in the specific case where $\bk=(k,0, \dotsc,0) \in \Z^d$ and $k \to +\infty$. As $L_j$ is a differential operator of order $2j$, we find:
    \[
    \begin{split}
    L_j \langle\Pi_{0}(\btheta)(\tau_{\bk}^*f)_{\btheta},g_{\btheta}\rangle|_{\btheta=0}&  = L_j(e^{i\btheta\cdot\bk}c(\btheta))|_{\btheta=\boldsymbol{0}}  =k^{2j} \sigma_{L_j}(\boldsymbol{0}, \mathbf{e}_1^*) c(\boldsymbol{0}) + \mc{O}(k^{2j-1}) \\
    &= k^{2j} \sigma_{L_j}(\boldsymbol{0}, \mathbf{e}_1^*) + \mc{O}(k^{2j-1}),
    \end{split}
    \]
    where $\sigma_{L_j} \in C^\infty(T^*\mathrm{U}(1)^d)$ denotes the (homogeneous) principal symbol of $L_j$, evaluated at the point $\btheta=\boldsymbol{0} \in \mathrm{U}(1)^d$ and covector $\mathbf{e}_1^* \in T^*_{\boldsymbol{0}}\mathrm{U}(1)^d$ (such that $\mathbf{e}_1^*(\partial_{\theta_1})=1$ and $\mathbf{e}_1^*(\partial_{\theta_j})=0$ for $j \geq 2$). We can argue in two ways that $\sigma_{L_j}(\boldsymbol{0}, \mathbf{e}_1^*) \neq 0$. First, by ellipticity of $L_j$ the above expression is nonzero for $k \gg 1$ large enough. On the other hand, since $h$ vanishes to third order at $\btheta=\mathbf{0}$, from \eqref{eq:stat-phase-lemma} we see that for $\xi \in T^*_{\mathbf{0}} \mathrm{U}(1)^{d} \simeq \mathbb{R}^d$ there is $c_{j} \neq0 $ with
    \begin{equation}
        \label{equation:expression-symbol-cj}
        \sigma_{L_j}(\mathbf{0}, \xi) = c_{j} \langle{\lambda_{0}''(\mathbf{0})^{-1}\xi, \xi}\rangle^j .
    \end{equation}
\end{proof}

\subsection{Proof of Theorem \ref{theorem:main1-local}}

We will need the following version of Lemma \ref{lemma:Delta_theta}.

\begin{lemma} \label{lemma:Delta_theta_local}
Let $m>\dim M_{0}/2+\max\{\ell_{x},\ell_{y}\}$. Let $U$ be a sufficiently small neighborhood of $\mathbf{0} \in \mathrm{U}(1)^d$ such that $\lambda_0(\btheta)$ and $\Pi_{0}(\btheta)$ are defined smoothly for $\btheta \in U$. Then there exist constants $c, C>0$ such that for $t \geq 1$:
\begin{enumerate}[label=(\roman*), itemsep=5pt]
    \item If $\btheta \in U$ we have
    \[
    \|H_{\btheta}(t,x,y) -e^{-t \lambda_{0}(\btheta)} \Pi_{0}(\btheta,x,y)\|_{C^{\ell_{x},\ell_{y}}(M_{0} \times M_{0})} \leq C e^{-c t}
    \]
    \item If $\btheta \in \mathrm{U}(1)^{d}\setminus U$, we have
    \[
    \|H_{\btheta}(t,x,y)\|_{C^{\ell_{x},\ell_{y}}(M_{0} \times M_{0})} \leq Ce^{-ct}.
    \]
\end{enumerate}  
Here $\Pi_{0}(\btheta,p,q)=\varphi_{0}^{\btheta}(p) \overline{\varphi_{0}^{\btheta}(q)}$ is the Schwartz kernel of $\Pi_{0}(\btheta)$.
\end{lemma}

\begin{proof}
Let us first consider the case when $\btheta \in U$. Just as in the proof of Lemma \ref{lemma:Delta_theta}(i), we have that there exists $c>0$ with
\begin{equation} \label{eq:L2-spectral-bound}
    \|e^{-t\Delta_{\btheta}}(\mathrm{Id}-\Pi_{0}(\btheta))\|_{L^{2} \to L^{2}}\leq e^{-ct}. 
\end{equation}
Also, since $\Delta_{\btheta}$ is elliptic, $\|\bullet\|_{H^{m}}$ and $\|(\mathrm{Id}+\Delta_{\btheta})^{m/2}\bullet\|_{L^{2}}$ are equivalent. Note that the constants are uniform in $\btheta$. Then, the spectral theorem gives
\[
\|e^{-\Delta_{\btheta}/2 }v\|_{H^{m}} \leq C_{m} \|(\mathrm{Id}+\Delta_{\btheta})^{m/2}e^{-\Delta_{\btheta}/2 }v\|_{L^{2}} \leq C_{m} \sup_{\lambda \geq 0} (1+\lambda)^{m/2}e^{-\lambda/2}\|v\|_{L^{2}},
\]
showing $\|e^{-\Delta_{\btheta}/2}\|_{L^{2} \to H^{m}} \leq C_{m}$. Since $\Delta_{\btheta}$ is self-adjoint, duality gives $\|e^{-\Delta_{\btheta}/2}\|_{H^{-m} \to L^{2}} \leq C_{m}$. These bounds together with the Sobolev embedding give that the term in Item (i) of the statement can be bounded by
\[
\begin{split}
    & \leq \|e^{-t\Delta_{\btheta}}(\mathrm{Id}-\Pi_{0}(\btheta))\|_{H^{-m} \to H^{m}} \\
    & \leq \|e^{-\Delta_{\btheta}/2}\|_{L^{2}\to H^{m}} \|e^{-(t-1)\Delta_{\btheta}}(\mathrm{Id}-\Pi_{0}(\btheta))\|_{L^{2}\to L^{2}}\|e^{-\Delta_{\btheta}/2}\|_{H^{-m}\to L^{2}} \\
    & \leq C_{m}e^{-c_{0}(t-1)} \leq C_{m}'e^{-c_{0}t}.
\end{split}
\]

Regarding $\btheta \in \mathrm{U}(1)^{d}\setminus U$, Lemma \ref{lemma:lambda-theta} shows that there exists $c>0$ with $\lambda_{0}(\btheta)>0$. Hence, the spectral theorem gives $\|e^{-t\Delta_{\btheta}}\|_{L^{2}} \leq e^{-ct}$. Thus,
\[
\begin{split}
    \|H_{\btheta}(t,p,q)\|_{C^{\ell_{x},\ell_{y}}(M_{0} \times M_{0})} & \leq \|e^{-t\Delta_{\btheta}}\|_{H^{-m}\to H^{m}} \\
    & \leq \|e^{-\Delta_{\btheta}/2}\|_{L^{2} \to H^{m}}\|e^{-(t-1)\Delta_{\btheta}}\|_{L^{2}\to L^{2}} \|e^{-\Delta_{\btheta}/2}\|_{H^{-m}\to L^{2}} \\
    & \leq C_{m}e^{-c_{1}(t-1)} \leq C_{m}'e^{-c_{1}t}.
\end{split}
\]
\end{proof}

Now we give a local version of Lemma \ref{lemma:regularity_Pi_theta}.

\begin{lemma} \label{ref:Pi-local}
Let $K_{x},K_{y} \subset M$ be compact, and let $\ell_{x},\ell_{y} \geq 0$. Take a contractible neighborhood $U \subset \mathrm{U}(1)^d$ of $\mathbf{0}$, choose a smooth family $\eta_{\btheta}$ on $U$ as in Proposition \ref{prop:eta-theta}, with $\eta_\mathbf{0}=0$, and let $s_{\btheta}$ be the corresponding smooth family of trivializations.

Define the amplitude
\begin{equation} \label{eq:amplitude}
    a(\btheta,x,y)=\chi_{0}(\btheta) s_{\btheta}(x) \overline{s_{\btheta}(y)} \varphi_{0}^{\btheta}(\pi(x)) \overline{\varphi_{0}^{\btheta}(\pi(y))}.
\end{equation}
where $\chi_0 \in C_c^{\infty}(U)$. Then, for every integer $k \geq 0$, there exist $C>0$ such that 
\[
\max_{|\alpha| \leq k}\sup _{\btheta \in U}\|\partial_{\btheta}^{\alpha} a(\btheta,\bullet,\bullet)\|_{C^{\ell_{x},\ell_{y}}(K_{x}\times K_{y})} \leq C.
\]
\end{lemma}

\begin{proof}
Since $\btheta$ is small, $\lambda_{0}(\btheta)$ is a simple eigenvalue in $U$.
Furthermore, since $U$ is contractible, the eigenfunction $\varphi_{0}^{\btheta}$ is smooth in $\btheta$. Therefore $\Pi_{0}$ and hence $a$ depend smoothly on $\btheta$, $x$, and $y$. The claim now follows from the smooth dependence of $s_{\btheta}$, the compact support of $\chi_{0}$, and the compactness of $K_{x}\times K_{y}$.    
\end{proof}

\begin{proof}[Proof of Theorem \ref{theorem:main1-local}]
As in the proof of Theorem \ref{theorem:main1}, we take a partition of unity, and let $H_{\btheta,i}(t,p,q)$ be the heat kernel of $e^{-t\Delta_{\btheta}}$  in the trivialization $s_{\btheta,i}$. Then write $H_{M}(t,x,y)=S(t,x,y)+Q(t,x,y)$, where
\[
\begin{split}
    Q(t,x,y):=&\frac{1}{(2\pi)^{d}}\int_{U_{0}}\chi_{0}(\btheta)e^{-t\lambda_{0}(\btheta)}s_{\btheta,0}(x)\overline{s_{\btheta,0}(y)}\Pi_{0}(\btheta,\pi(x),\pi(y)) ~\dd\btheta, \\
    S(t,x,y):=&\frac{1}{(2\pi)^{d}} \left(\sum_{i=1}^{k}\int_{U_{i}}\chi_{i}(\btheta)s_{\btheta,i}(x) \overline{s_{\btheta,i}(y)}H_{\btheta,i}(t,\pi(x),\pi(y)) ~\dd \btheta \right. \\
    &\left.+\int_{U_{0}}\chi_{0}(\btheta)s_{\btheta,0}(x)\overline{s_{\btheta,0}(y)}(H_{\btheta,0}(t,\pi(x),\pi(y))-e^{-t\lambda_{0}(\btheta)}\Pi_{0}(\btheta,\pi(x),\pi(y))) ~\dd\btheta \right).
\end{split}
\]
By the smoothness of the trivializing sections, and the compactness of $K_{x} \times K_{y}$, using Lemma \ref{lemma:Delta_theta_local} one obtains
\[
\|S(t)\|_{C^{\ell_{x},\ell_{y}}(K_{x}\times K_{y})} \leq Ce^{-ct}.
\]
To deal with $Q$, let $\beta$ and $\gamma$ be multiindices such that $|\beta|\leq \ell_{x}$, $|\gamma| \leq \ell_{y}$. Since $a$ (see \eqref{eq:amplitude}) is compactly supported in $\btheta$, we have, writing $a_{\beta,\gamma}=\partial_{x}^{\beta}\partial_{y}^{\gamma}a$:
\[
\partial_{x}^{\beta}\partial_{y}^{\gamma}Q=\int_{U_{0}}e^{-t\lambda_{0}(\btheta)}a_{\beta,\gamma}(\btheta,x,y)~\dd \btheta.
\]
By Lemmas \ref{lemma:lambda-theta} and \ref{lemma:resonance-non-degenerate}, for each $(x,y) \in K_{x}\times K_{y}$, we can apply the stationary phase lemma to obtain
\[
t^{d/2}\partial_{x}^{\beta}\partial_{y}^{\gamma}Q=\frac{1}{(2\pi)^{d/2}\sqrt{\det \lambda_{0}''(\mathbf{0})}} \sum_{j=0}^{N-1}t^{-j}L_{j}a_{\beta,\gamma}(\mathbf{0},x,y)+R_{N,\beta,\gamma}(t,x,y),
\]
where $L_{j}$ are differential operators of order $2j$ in $\btheta$, and 
\[
|R_{N,\beta,\gamma}(t,x,y)| \leq C t^{-N}\sum_{|\alpha|\leq 2N+d+1}\sup_{\btheta \in U_{0}}|\partial_{\btheta}^{\alpha}a_{\beta,\gamma}(\btheta,x,y)|.
\]
Note that the constant $C$ is uniform in $(x,y)$ since $K_{x} \times K_{y}$ is compact. Then, 
\[
\begin{split}
    \left\| t^{d/2}\int_{U_{0}}e^{-t\lambda_{0}(\btheta)}a(\btheta,\bullet,\bullet)~\dd\btheta-\kappa \sum_{j=0}^{N-1}t^{-j} L_{j}a(\mathbf{0},x,y) \right\|_{C^{\ell_{x},\ell_{y}}} \\
    \leq Ct^{-N}\max_{|\alpha|\leq 2N+d+1}\sup_{\btheta \in U_{0}}\|\partial_{\btheta}^{\alpha}a(\btheta,\bullet,\bullet)\|_{C^{\ell_{x},\ell_{y}}}.
\end{split}
\]
The bound on the remainder then follows from this, together with Lemma \ref{ref:Pi-local} and the bound on $S$. The result now follows by setting $k_{j}=\kappa L_{j}a(\mathbf{0})$, which is bounded by Lemma \ref{ref:Pi-local}. Finally, since $L_{0}=\mathrm{Id}$, following the definitions and since $\Delta_{\mathbf{0}}=\Delta_{M_{0}}$, the $j=0$ contribution equals $\kappa$.
\end{proof}

\begin{remark}
The fact that $k_{j}$ are not identically zero follows from the fact that the coefficients $C_{j}$ are nonzero (see \S\ref{subsection:non-vanishing}). Indeed, let $f, g\in C^\infty_{\mathrm{comp}}(M)$, and take compact sets
$K_{x}, K_{y} \subset M$ containing their supports. Pairing the expansion from Theorem \ref{theorem:main1-local} with $f \otimes g$ we obtain
\[
t^{d/2}\langle e^{-t\Delta_M}f,g\rangle= \kappa \left(\int_M f\right) \overline{ \left(\int_M g\right)}+\sum_{j=1}^{N-1}t^{-j} \int_{M\times M}
k_j(x,y)f(y)\overline{g(x)}+ \mathcal{O}(t^{-N}).
\]
On the other hand, the dual expansion from Theorem \ref{theorem:main1} gives
\[
t^{d/2}\langle e^{-t\Delta_M}f,g\rangle=\kappa \left(\int_M f\right) \overline{\left(\int_M g\right)}+ \sum_{j=1}^{N-1}t^{-j}C_j(f,g) + \mathcal{O}(t^{-N}).
\]
The uniqueness of asymptotic expansions in powers of $t^{-1}$ implies that $\mathcal{C}_{j}$ is the distribution induced by the kernel $k_{j}$. Hence, if the $k_j$ were identically zero, then $C_{j}(f,g)=0$ for every $f,g \in C^\infty_{\mathrm{comp}}(M)$, contradicting Lemma \ref{lemma:cj-non-zero}. Hence $k_{j}\not\equiv 0$.
\end{remark}

\subsection{Proof of Theorem \ref{theorem:main1-global}}

Recall that $\Delta_{M}$ is invariant by isometries. In particular, it is invariant by the action of $\tau_{\n}$. Hence, $H_{M}(t,\tau_{\n}x_{0},\tau_{\mathbf{m}}y_{0})=H_{M}(t,\tau_{\n-\mathbf{m}}x_{0},y_{0})$. Therefore, it is enough to prove the result for $H_{M}(t,\tau_{\n}x_{0},y_{0})$ for $|\n|\leq R \sqrt{t}$. Write $p=\pi(x_{0})$, $q=\pi(y_{0})$. As in the proof of Theorem \ref{theorem:main1-local}, we consider a partition of unity, and we let $H_{\btheta,i}(t,p,q)$ denote the heat kernel of $\Delta_{\btheta}$ associated to the trivialization given by $s_{\btheta,i}$. Then, using Proposition \ref{prop:eta-theta} (ii) we have
\[
\begin{split}
    H_{M}(t,\tau_{\n}x_{0},y_{0}) &=\frac{1}{(2\pi)^{d}}\sum_{i=0}^{k}\int_{U_{i}}\chi_{i}(\btheta)e^{i\n \cdot \btheta}s_{\btheta,i}(x_{0}) \overline{s_{\btheta,i}(y_{0})}H_{\btheta,i}(t,p,q)~\dd\btheta \\
    &=S(t,\n,x_{0},y_{0})+\underbrace{\frac{1}{(2\pi)^{d}}\int_{U_{0}}\chi_{0}(\btheta)e^{i\n \cdot \btheta}e^{-t\lambda_{0}(\btheta)}s_{\btheta,0}(x_{0})\overline{s_{\btheta,0}(y_{0})}\Pi_{0}(\btheta,p,q)~\dd\btheta}_{Q(t,\n,x_{0},y_{0})}.
\end{split}
\]
Since $|e^{i\n \cdot\btheta}|=1$, the $C^{\ell_{x},\ell_{y}}$-norm of $S$ can be bounded using Lemma \ref{lemma:Delta_theta_local} by $Ce^{-ct}$ for any $\n \in \Z^{d}$. Now we deal with $Q$. Define $a$ as in \eqref{eq:amplitude} (with $s_{\btheta,0}$ instead of $s_{\btheta}$). Choose open neighborhoods $V_{1} \Subset V_{2} \Subset U_{0}$, and $\chi \in C_{\comp}^{\infty}(V_{2})$ with $\chi \equiv 1$ on $V_{1}$. Then, 
\[
Q:=\underbrace{\frac{1}{(2\pi)^{d}}\int \chi(\btheta)e^{i\n \cdot \btheta}e^{-t\lambda_{0}(\btheta)}a(\btheta,x_{0},y_{0}) ~\dd \btheta}_{Q_{1}}+Q_{2}.
\]
By Lemma \ref{lemma:lambda-theta}, compactness gives a constant $c_{1} > 0$ such that $\lambda_{0}(\btheta) \geq c_{1}$ on the support of $(1-\chi)a$. Hence, for every $\ell_x,\ell_y$,
\begin{equation}
\label{eq:global-far}
\sup_{\n \in \Z^{d}} |Q_{2}(t,\n,\bullet,\bullet)|_{C^{\ell_{x},\ell_{y}}(\overline{D} \times \overline{D})} \leq Ce^{-c_1t}.
\end{equation}
To deal with $Q_{1}$, let $\mathbf{z}=\n/\sqrt{t}$, and change variables $\btheta=\mathbf{u}/\sqrt{t}$. Then,
\begin{equation}
\label{eq:global-rescaled-Q}
t^{d/2}Q_{1}(t,\n,x_{0},y_{0})=\frac{1}{(2\pi)^d} \int_{\R^{d}} e^{i\mathbf{z} \cdot \mathbf{u}} e^{-t\lambda_{0}(\mathbf{u}/\sqrt{t})} A_{t}(\mathbf{u},x_{0},y_{0}) \dd \mathbf{u},
\end{equation}
where $A_{t}(\mathbf{u},x_{0},y_{0}) := \chi(\mathbf{u}/\sqrt{t}) a(\mathbf{u}/\sqrt{t},x_{0},y_{0})$ (where the integrand is extended by zero outside the rescaled support). Set $H:=D^{2}\lambda_{0}(\mathbf{0})$. By Lemma \ref{lemma:resonance-non-degenerate}, $H$ is positive definite. 

Now we will use a Taylor expansion. Let us define $\Lambda_{q}(\mathbf{u}):=\frac{1}{q!} D^{q}\lambda_{0}(\mathbf{0})[\mathbf u,\ldots,\mathbf u]$. Taylor-expanding and using Lemma \ref{lemma:resonance-non-degenerate}, we have 
\[
    t\lambda_{0}(\mathbf{u}/\sqrt{t})=\frac{1}{2}\langle H\mathbf{u},\mathbf{u}\rangle +\sum_{q=3}^{N+2}t^{1-q/2}\Lambda_{q}(\mathbf{u})+t^{-(N+1)/2}E_{N+3}(t,\mathbf{u}),
\]
where on the rescaled support we have $|E_{N+3}(t,\mathbf{u})| \leq C|\mathbf{u}|^{N+3}$. This allows us to expand the exponential up to order $N-1$ in $t^{-1/2}$: 
\begin{equation}
\label{eq:direct-exponential-expansion}
e^{-t\lambda_{0}(\mathbf{u}/\sqrt{t})}=e^{-\frac{1}{2}\langle H\mathbf{u},\mathbf{u}\rangle} \left( \sum_{j=0}^{N-1} t^{-j/2}c_{j}(\mathbf{u}) + t^{-N/2}C_{N}(t,\mathbf{u}) \right),
\end{equation}
where the $c_{j}$ are polynomials obtained from finite products of $\Lambda_{3}, \ldots, \Lambda_{j+2}$. 

On the other hand, after shrinking $V_{2}$ if necessary, there exists $c_{2}>0$ such that for $\btheta \in V_{2}$, we have $\lambda_{0}(\btheta) \geq c_{2}|\btheta|^{2}$. Therefore, on the support of $A_{t}$ we obtain $e^{-t\lambda_{0}(\mathbf{u}/\sqrt{t})} \leq e^{-c_{2}|\mathbf{u}|^{2}}$. Moreover, Taylor’s formula for the exponential shows that there exist constants $C>0$, $M_{N} \geq 0$, and $\eps>0$ such that
\begin{equation}
\label{eq:direct-CN-bound}
|C_N(t,\mathbf u)| \leq C\langle\mathbf u\rangle^{M_N} e^{\varepsilon|\mathbf u|^2}
\end{equation}
on the rescaled support. Indeed, every coefficient arising before order $t^{-N/2}$ is polynomial in $\mathbf{u}$, while the integral remainder in Taylor’s formula for the exponential is bounded by a polynomial in $\mathbf{u}$ multiplied by the exponential of the absolute value of the higher-order part of the phase. After shrinking the original neighborhood, this exponential is bounded by $e^{\eps|\mathbf{u}|^2}$, with $\eps>0$ small.

Now we use Taylor in the amplitude to obtain
\[
A_{t}(\mathbf{u},x_{0},y_{0})=\sum_{q=0}^{N-1} t^{-q/2} a_{q}(\mathbf{u},x_{0},y_{0})+t^{-N/2} F_{N}(t,\mathbf{u},x_{0},y_{0}),
\]
where $a_{q}(\mathbf{u},x_{0},y_{0})=\frac{1}{q!} D_{\btheta}^{q} a(\mathbf{0},x_{0},y_{0})[\mathbf{u},\ldots,\mathbf{u}]$, and $\|F_N(t,\mathbf{u},\bullet,\bullet)\|_{C^{\ell_{x},\ell_{y}}(\overline {D} \times \overline {D})} \leq C\langle \mathbf{u}\rangle^N$. Multiplying this expansion by \eqref{eq:direct-exponential-expansion}, we obtain polynomials
$b_j(\mathbf{u},x_{0},y_{0})$, $j=0,\ldots,N-1$, in $\mathbf{u}$, with coefficients smooth in $(x_{0},y_{0})$, and a remainder $B_{N}$ such that
\begin{equation}
\label{eq:direct-full-integrand-expansion}
\begin{split}
e^{-t\lambda_{0}(\mathbf{u}/\sqrt{t})}A_{t}(\mathbf{u},x_{0},y_{0})=e^{ \frac{1}{2}\langle H\mathbf{u},\mathbf{u}\rangle} \left(\sum_{j=0}^{N-1} t^{-j/2} b_{j}(\mathbf{u},x_{0},y_{0}) + t^{-N/2} B_{N}(t,\mathbf{u},x_{0},y_{0}) \right).
\end{split}
\end{equation}
Here, $B_{N}$ is a finite sum of terms involving $C_{N}(t,\mathbf{u})a_{q}(\mathbf{u},x_{0},y_{0})$, $c_{j}(\mathbf{u})F_{N}(t,\mathbf{u},x_{0},y_{0})$,
and products whose total power of $t^{-1/2}$ is at least $N$. Since each such term has at most polynomial growth in $\mathbf{u}$ (apart from the factor $e^{\eps|\mathbf{u}|^{2}}$ coming from \eqref{eq:direct-CN-bound}), we conclude that for every $\ell_{x},\ell_{y} \geq 0$, there exist constants $C,M_{N}>0$ and $\eps>0$ small enough such that
\[
\|B_{N}(t,\mathbf{u},\bullet,\bullet)\|_{C^{\ell_{x},\ell_{y}}(\overline{D} \times \overline{D})} \leq C\langle\mathbf{u}\rangle^{M_{N}} e^{\eps|\mathbf{u}|^{2}}.
\]
Furthermore, since $H$ is positive definite, we may choose $\eps>0$ satisfying
$\langle H\mathbf{u},\mathbf{u}\rangle/2 - \eps|\mathbf{u}|^{2} \geq
c_{1}|\mathbf{u}|^2$, for some $c_{1}>0$. Hence,
\[
e^{-\frac{1}{2}\langle H\mathbf{u},\mathbf{u}\rangle} \|B_{N}(t,\mathbf{u},\bullet,\bullet)\|_{C^{\ell_{x},\ell_{y}}} \leq C\langle \mathbf{u} \rangle^{M_{N}} e^{-c_{1}|\mathbf{u}|^{2}}.
\]
Observe that right-hand side is now integrable. Therefore, we can write
\[
t^{d/2}Q_{1}=\frac{1}{(2\pi)^{d}}\sum_{j=0}^{N-1}t^{-j/2}\int_{\R^{d}}e^{-\langle H \mathbf{u},\mathbf{u}\rangle/2}e^{i\mathbf{z}\cdot \mathbf{u}}b_{j}(\mathbf{u},x_{0},y_{0})~\dd \mathbf{u}+\mathcal{E}_{N}(t,\mathbf{z},x_{0},y_{0}),
\]
where $\|\mathcal{E}_{N}(t,\mathbf{z},\bullet,\bullet)\|_{C^{\ell_{x},\ell_{y}}}\leq C_{N}t^{-N/2}$. Also, for each $j$, the Fourier transform of the Gaussian times the polynomial $b_{j}$ has the form $e^{-\langle H^{-1}\mathbf{z},\mathbf{z} \rangle/2}
P_{j}(\mathbf{z},x_{0},y_{0})$, where $P_{j}$ is polynomial in $\mathbf{z}$ and smooth in $(x_{0},y_{0})$. The result now follows from combining the estimate of $S$ with the one of $\mathcal{E}_{N}$.

\section{Abelian covers of isometric extensions}

\label{section:decay-isometric}

Throughout this section, $p_{0} \colon P_0 \to M_0$ is a principal $G$-bundle and $\pi \colon M \to M_0$ denotes the $\Z^d$-extension, and $P =\pi^*P_0 \to M$ is the pullback $G$-bundle. When the context is clear, we will use the convention that the projection $G$-bundle map $P \to P_0$ is also denoted by $\pi$, and that the projection $P \to M$ is denoted by $p$. As in \cite{Cekic-Lefeuvre-Munoz-Thon-26}, it can be shown that $\pi \circ p \colon P\to M_0$ is a $(G \times \Z^d)$-bundle.

\subsection{Borel--Weil--Floquet theory} \label{ssection:borel-weil-floquet} In this subsection, we generalize the Floquet theory explained in \S\ref{ssection:floquet-theory} in order to include a $G$-bundle part. Let $f \in C^\infty_{\comp}(P)$. Applying Proposition \ref{prop:fourier-theory}, we may write
\[
f = \dfrac{1}{(2\pi)^d} \int_{\mathrm{U}(1)^d} F_{\btheta} ~\dd\btheta,
\]
where $F_{\btheta} \in C^\infty(P)$ is a $\Z^d$-equivariant function satisfying $\tau_{\mathbf{n}}^*F_{\btheta} = e^{i\n\cdot\btheta}F_{\btheta}$ for all $\n \in \Z^d$; the space of such functions is denoted by $C^\infty_{\btheta}(P)$. After fixing an open contractible set $U \subset \mathrm{U}(1)^d$, for $\btheta \in U$, this function can be written as $F_{\btheta} = (\pi^* f_{\btheta}) s_{\btheta}$, where $f_{\btheta} \in C^\infty(P_0)$ and $s_{\btheta} \in C^\infty_{\btheta}(P)$ has pointwise unit norm. Note that $s_{\btheta}$ can simply be taken to be the pullback by $p$ of a corresponding element of $C_{\btheta}^\infty(M)$; by slightly abusing the notation, we will denote this element by $s_{\btheta}$ as well. In turn, $f_{\btheta}$ can be decomposed using the fiberwise Fourier transform \eqref{equation:fourier-group} on the group $G$, that is $\mc{F}(f_{\btheta}) = (f_{\btheta,\bk, i})_{\bk \in \widehat{G}, i=1,\dotsc, d_{\bk}}$, where $f_{\btheta,\bk,i} \in C^\infty_{\hol}(F_0,\Lk)$ and $F_0 := P_0/T$ is the flag bundle over $M_0$. We note that technically, in the Fourier transform of the group $G$ there is an additional index $i = 1,\dotsc, d_{\bk}$ (see \S\ref{ssection:bw-calculus} for the definition of $d_{\bk}$), but we suppress it, since the arguments below are the same for every $i$.

We may also define $F_{\btheta,\bk} := (\pi^* f_{\btheta,\bk}) s_{\btheta} \in C^\infty_{\hol,\btheta}(F, \pi^*\Lk)$, where $F := P/T = \pi^*F_0$ is the flag bundle over $M$, and we see $s_{\btheta}$ as a function on $F$, and, similarly to \eqref{equation:equivariant-space}, the subscript $\btheta$ indicates the equivariance of this section with respect to the natural induced action of $\Z^d$ on $F$ which lifts to a $\Z^d$-action on $C^\infty(F, \pi^*\Lk)$. In what follows to simplify the notation we will write $\Lk \to F$ instead of $\pi^*\Lk \to F$. Then we may write $\mc{F}(f) = (f_{\kk, i})_{\bk \in \widehat{G}, i = 1, \dotsc, d_{\bk}}$, where $f_{\kk, i} \in C^\infty_{\hol}(F, \Lk)$ and
\[
    f_{\kk, i} = \frac{1}{(2\pi)^d} \int_{\mathrm{U}(1)^d} F_{\btheta, \kk, i} ~\dd\btheta.
\]
Note that here $F_{\btheta, \kk, i}$ do not depend on $U$ and the choices made when selecting $(s_{\btheta})_{\btheta \in U}$. In summary, we may either take a partial Fourier transform in $G$, and then in $\Z^d$, or vice versa, and the two procedures commute.

Alternatively, one may choose not to trivialize the line bundle $L_{\btheta} \to M_0$ (using the section $s_{\btheta}$); in this case, $f_{\btheta,\bk}$ should be seen as a fiberwise holomorphic section of $p_{F_{0}}^* L_{\btheta} \otimes \Lk \to F_0$, where $p_{F_{0}} \colon F_0 \to M_0$ is the footpoint projection and $p_{F_{0}}^*L_{\btheta} \to F_0$ is the pullback bundle. Since the realizations described here are all unitarily equivalent, we will not further distinguish between the corresponding objects.

The following result generalizes Proposition \ref{prop:fourier-theory}; its proof is a straightforward consequence of the previous paragraph, and the fact that the Fourier transform is an $L^2$-isometry, see e.g. \cite[Lemma 2.2.4]{Cekic-Lefeuvre-24}:

\begin{proposition} \label{proposition:borel-floquet-weil}
Let $f \in C^\infty_{\comp}(P)$. Then
\[
f = \dfrac{1}{(2\pi)^d} \int_{\mathrm{U}(1)^d} \mc{F}^{-1}\left(\bigoplus_{\bk \in \widehat{G}, i = 1, \dotsc, d_{\bk}} F_{\btheta,\bk, i}\right) \dd\btheta.
\]
In addition, we have the following Parseval identity: for all $f,g \in C_{\mathrm{comp}}^\infty(P)$,
\[
\langle f,g \rangle_{L^2(P)} = \dfrac{1}{(2\pi)^d} \sum_{\bk \in \widehat{G}} d_{\bk}\sum_{i = 1}^{d_{\bk}}\int_{\mathrm{U}(1)^d} \langle F_{\btheta,\bk, i}, G_{\btheta,\bk, i}\rangle_{L^2(F_0, p_{F_{0}}^{*}L_{\btheta} \otimes \Lk)} \dd \btheta.
\]
\end{proposition}

\subsection{The horizontal Laplacian} \label{subsection:horizontal-laplacian}

Let $P$ be as above. The fixed $G$-equivariant connection $\nabla^{P_0}$ pulls back through $\pi\colon P=\pi^{*}P_{0} \to P_{0}$ to a $G$-equivariant connection $\nabla:=\pi^{*}\nabla^{P_{0}}$ on $P$. We denote by $\mathbb {H}_{P} \subset TP$ the corresponding horizontal distribution. Thus $TP=\mathbb{H}_{P} \oplus \mathbb{V}_{P}$, where $\mathbb{V}_{P}:=\ker dp$. We also define their duals by $\HH_{P}^{*}(\V_{P})=0=\V_{P}^{*}(\HH_{P})$. It can be checked that both the $G$-action and the $\Z^d$-action preserve $\mathbb{H}_{P}$. The pullback connection also induces horizontal distributions $\mathbb{H}_{F} \subset TF$, and connections on the pulled-back Borel--Weil line bundles $\mathbf{L}^{\otimes\bk} \to F$. These are the pullbacks of the corresponding horizontal distribution and connections on $F_{0}$.

The restriction of the exterior derivative to the horizontal distribution defines the horizontal exterior derivative
\[
d_{\mathbb{H}_{P}}\colon C^{\infty}(P) \to C^{\infty}(P,\mathbb{H}_{P}^{*}), \qquad d_{\mathbb{H}_{P}}f:=df|_{\mathbb{H}_{P}}.
\]
We equip $\mathbb{H}_{P}$ with the metric transported from $TM$ through the isomorphism $dp|_{\mathbb{H}_{P}}\colon \mathbb{H}_{P} \to TM$, and $P$ with the measure given locally by $\dd\vol_{M} \wedge \dd g$. The \emph{horizontal Laplacian} on $P$ is
\[
\Delta_P^{\HH} :=(d_{\mathbb H_{P}})^*d_{\mathbb H_{P}},
\]
where $(d_{\mathbb{H}_{P}})^*$ denotes the formal $L^2$-adjoint with respect to these structures. One advantage of this operator is that it is compatible with both the $G$-action and the deck action of $\Z^d$, allowing it to be simultaneously decomposed by the Abelian Fourier transform and the Borel--Weil calculus, that is, since the connection on $P$ is invariant under the actions of $G$ and $\Z^{d}$, $\Delta_{P}^{\mathbb{H}}$ commutes with both actions:
\[
\Delta_{P}^{\mathbb{H}} R_{g}^{*}=R_{g}^{*}\Delta_{P}^{\mathbb{H}}, \qquad \Delta_{P}^{\mathbb{H}}\tau_{\n}^{*}=\tau_{\n}^{*}\Delta_{P}^{\mathbb{H}}.
\]
It therefore preserves $C^{\infty}_{\btheta}(P)$ and, within each such space, decomposes further according to the Borel--Weil decomposition.

We work with the horizontal Laplacian because it is canonically determined by the connection and is compatible with both the $G$- and the $\Z^{d}$-actions. For a connection metric on $P$, the Laplace--Beltrami operator decomposes as the sum of horizontal and vertical Laplacians, and the vertical component contributes only exponentially decaying modes. We return to this point in Section \ref{subsection:laplacian}.

We begin with the Abelian Fourier transform along the $\Z^{d}$-cover, which decomposes $\Delta_{P}^{\mathbb{H}}$ into a family of twisted horizontal Laplacians $(\Delta_{\btheta}^{\mathbb{H}})_{\btheta \in \mathrm{U}(1)^{d}}$ acting on $P_{0}$. Second, the Borel--Weil decomposition on the compact principal bundle $P_{0} \to M_{0}$ decomposes each operator $\Delta_{\btheta}^{\mathbb{H}}$ into operators $(\Delta_{\boldsymbol\theta,\mathbf{k}})_{\btheta \in \mathrm{U}(1)^{d},\bk \in \widehat{G}}$.

To construct these operators, let $F_{0} \to M_{0}$ be the flag bundle associated with $P_{0} \to M_{0}$. Applying the Borel--Weil calculus of \S\ref{ssection:bw-calculus}, for every $\bk \in \widehat{G}$ one obtains a connection $\nabla_{\bk}^{\HH_{F_{0}}} \colon C^{\infty}_{\mathrm{hol}} (F_{0},\mathbf {L}^{\otimes\bk}) \to  C^{\infty}_{\mathrm{hol}} (F_{0},\mathbf{L}^{\otimes\bk} \otimes \mathbb{H}_{F_{0}}^{*})$. Set $\Delta_{\bk}:=(\nabla_{\bk}^{\HH_{F_{0}}})^{*}\nabla_{\bk}^{\HH_{F_{0}}}$. In light of \cite[Lemma~2.2.10]{Cekic-Lefeuvre-24}, the Borel--Weil decomposition intertwines the horizontal differential with these connections, that is, for every $f\in C^\infty(P_{0})$,
\begin{equation}
\label{eq:horizontal-modes}
(d_{\mathbb{H}_{P_{0}}}f)_{\bk} = \nabla_{\bk}^{\HH_{F_{0}}}f_{\bk}, \qquad (\Delta_{P_{0}}^{\mathbb {H}}f)_{\bk} = \Delta_{\bk}f_{\bk}. 
\end{equation}
Here we recall that as before we have $TP_{0}=\HH_{P_{0}}\oplus\V_{P_{0}}$, $\V_{P_{0}}:=\ker dp_{0}$, and $d_{\HH_{P_{0}}} \colon C^{\infty}(P_{0}) \to C^{\infty}(P_{0},\HH_{P_{0}}^{*})$.

We now construct the twisted operators $\Delta_{\btheta,\bk}$. For $\btheta\in\mathrm{U}(1)^{d}$, write $F_{\btheta} = (\pi^{*}f_{\btheta}) (p^{*}s_{\btheta})$, where $p \colon P \to M$, $\pi \colon P \to P_{0}$, and $s_{\btheta}$ is the section introduced in Proposition~\ref{prop:eta-theta}.

Using the Leibniz rule,
\begin{equation}
\label{eq:d-H1}
d_{\mathbb {H}_{P}}F_{\btheta} = (p^{*}s_{\btheta}) d_{\mathbb{H}_{P}}(\pi^{*}f_{\btheta}) + (\pi^{*}f_{\btheta}) d_{\mathbb{H}_{P}}(p^{*}s_{\btheta}).
\end{equation}
Since $f_{\btheta}\in C^{\infty}(P_{0})$, $d_{\mathbb{H}_{P}}(\pi^{*}f_{\btheta}) =\pi^{*}(d_{\mathbb {H}_{P_{0}}}f_{\btheta})$. By Proposition \ref{prop:eta-theta} we also have $ds_{\btheta} = is_{\btheta}\pi^{*}\eta_{\btheta}$. Pulling back by $p$ and restricting to $\mathbb{H}_{P}$ gives
\[
d_{\mathbb{H}_{P}}(p^{*}s_{\btheta}) = i(p^{*}s_{\btheta}) p^{*}\pi ^{*}\eta_{\btheta} = i(p^{*}s_{\btheta}) \pi^{*}p_{0}^{*}\eta_{\btheta},
\]
where we used the fact that $\pi \circ p =p_{0}\circ\pi$. Using this together with \eqref{eq:d-H1}, we obtain
\[
d_{\mathbb{H}_{P}}F_{\btheta} = (p^{*}s_{\btheta}) \pi^{*} (d_{\mathbb{H}_{P_{0}}} + ip_{0}^{*}\eta_{\btheta}) f_{\btheta}.
\]
The same argument applies to the formal adjoint. Therefore,
\[
\Delta_{P}^{\mathbb{H}} F_{\btheta} = (p^{*}s_{\btheta}) \pi^{*}(\Delta_{\btheta}^{\mathbb{H}} f_{\btheta}), \qquad \Delta_{\btheta}^{\mathbb{H}}= (d_{\mathbb{H}_{P_{0}}} + ip_{0}^{*}\eta_{\btheta})^{*}
(d_{\mathbb{H}_{P_{0}}} + ip_{0}^{*}\eta_{\btheta}).
\]
Thus, the horizontal Laplacian $\Delta_P^{\mathbb{H}}$ is represented by the family $(\Delta_{\btheta}^{\mathbb{H}})_{\btheta \in \mathrm{U}(1)^{d}}$. 

Note that by our hypothesis on the curvature, $\Delta_{P_{0}}^{\HH}$ is (locally) hypoelliptic (\cite[\S1.2.2]{Cekic-Lefeuvre-24}), and it follows from the previous computation that the same is true for $\Delta_{P}^{\mathbb{H}}$. This ensures the smoothness of $H_{P}^{\HH}$. 

Applying now the Borel--Weil decomposition, $\mathcal F(f_{\btheta})=(f_{\btheta,\bk,i})_{\bk,i}$, and suppressing the multiplicity index $i$, equation \eqref{eq:horizontal-modes} gives
\begin{equation} \label{eq:operator-floquet-borel-weil}
    (\Delta_{\btheta}^{\mathbb{H}} f_{\btheta})_{\bk} = (\nabla_{\bk}^{\HH_{F_{0}}} + ip_{F_{0}}^{*}\eta_{\btheta})^{*} (\nabla_{\bk}^{\HH_{F_{0}}} + ip_{F_{0}}^{*}\eta_{\btheta}) f_{\btheta,\bk}=:\nabla_{\btheta,\bk}^{*}\nabla_{\btheta,\bk}f_{\btheta,\bk} =: \Delta_{\btheta,\bk} f_{\btheta,\bk}.
\end{equation}
In summary, after applying first the Abelian Fourier transform and then the Borel--Weil decomposition, the horizontal Laplacian on the noncompact bundle $P$ is reduced to the family of operators $(\Delta_{\btheta,\bk})_{\btheta \in \mathrm{U}(1)^{d},\bk \in \widehat{G}}$ acting on the compact flag bundle $F_0$. Note that a more explicit form for this family is
\[
\Delta_{\btheta,\bk} f_{\btheta,\bk} =\Delta_{\bk} f_{\btheta,\bk}- 2i \langle p_{F_{0}}^{*}\eta_{\btheta}, \nabla_{\bk}^{\HH_{F_{0}}} f_{\btheta,\bk} \rangle + i d_{\mathbb{H}_{F_{0}}}^{*} (p_{F_{0}}^{*}\eta_{\btheta}) f_{\btheta,\bk} + |p_{F_{0}}^{*}\eta_{\btheta}|^{2} f_{\btheta,\bk}.
\]

\subsection{The Laplace operator} \label{subsection:laplacian}

Throughout this subsection, fix an $\mathrm{Ad}$-invariant inner product $\langle\cdot,\cdot\rangle_{\mathfrak g}$ on the Lie algebra $\mathfrak{g}$. Since $G$ is compact, it induces a bi-invariant Riemannian metric $g_{G}$ on $G$. Together with the Riemannian metric on $M$ and the connection $TP = \mathbb{H}_{P} \oplus \mathbb{V}_{P}$, this determines the \emph{connection metric} on $P$, uniquely characterized by
\begin{enumerate}
    \item $\mathbb{H}_{P}$ and $\mathbb{V}_{P}$ are orthogonal;
    \item the restriction $dp|_{\mathbb H_{P}} \colon \mathbb{H}_{P} \to TM$ is an isometry;
    \item the induced metric on each fiber is the fixed bi-invariant metric on $G$.
\end{enumerate}
We denote by $\Delta_{P}$ the corresponding Laplace--Beltrami operator. The bi-invariant metric on $G$ also defines a Laplace--Beltrami (or Casimir) operator $\Delta_G$ on the Lie group. Since every fiber of $p \colon P\to M$ can be identified with $G$, this induces a second-order differential operator on $P$, called the \emph{vertical Laplacian}. More precisely, given $f\in C^\infty(P)$ and $u\in P$, define $\widetilde{f}_{u}(g):=f(u\cdot g)$ (where $\cdot$ denotes the right action of $G$ on $P$). The vertical Laplacian is
\[
\Delta_{P}^{\mathbb V}f(u) :=\Delta_G\widetilde{f}_{u}|_{g=\mathbf{1}_{G}}.
\]
Following the definitions, we see the following relation between the different types of Laplacians that we have so far.

\begin{proposition}
With respect to the connection metric,
\[
\Delta_{P} = \Delta_P^{\mathbb{H}} +\Delta_{P}^{\mathbb{V}}.
\]
Moreover, $[\Delta_P^{\mathbb{H}},\Delta_{P}^{\mathbb{V}}]=0$.
\end{proposition}

The importance of this decomposition stems from the spectral properties of $\Delta_{P}^{\mathbb V}$. Indeed, since $G$ is compact, the spectrum of $\Delta_G$ is discrete, $0=\mu_{0}<\mu_{1} \leq \cdots$, and therefore the same is true fiberwise for $\Delta_{P}^{\V}$. Furthermore,
\[
\ker(\Delta_{P}^{\mathbb{V}}) =\{ f\in C^\infty(P): f \text{ is constant on every fiber}\}.
\]
Hence every nonconstant fiber mode is exponentially damped by a factor $e^{-t\mu_j}$ with $\mu_j>0$. As a consequence, the long-time asymptotic behavior of the heat semigroup generated by $\Delta_P$ is entirely determined by the horizontal heat semigroup. Corollary \ref{corollary:laplace-expansion} hence follows.

\subsection{Spectral gap} \label{subsection:spectral-gap}

As in the proof of Theorem \ref{theorem:main1}, we will need a bound for the first eigenvalue of $\Delta_{\btheta,\bk}$. It would be useful to the reader to recall the definitions in \S\ref{subsection:curvatures}.

\begin{theorem} \label{thm:spectral-gap}
Assume that the curvature associated to $\nabla^{P_{0}}$ is globally nondegenerate and that $\mathrm{Hol}(P,\nabla)$ is dense in $G \times \Z^{d}$. Then, there exists $C>0$ such that, for all $\btheta \in \mathrm{U}(1)^{d}$ and every $\bk \in \widehat{G} \setminus \{\mathbf{0}\}$,
\[ 
\lambda_{0}(\btheta,\bk) \geq C.
\]
\end{theorem}

Here $\lambda_{0}(\btheta,\bk)$ is the first eigenvalue corresponding to the operator $\Delta_{\btheta,\bk}$ defined in \eqref{eq:operator-floquet-borel-weil}. To prove the result, we will adapt the result obtained in \cite{Cekic-Lefeuvre-24} for the Borel--Weil operator $\Delta_{\bk}$ constructed from the horizontal Laplacian on $G$-principal extensions, to our case where we also have an $\Z^{d}$-component. We will prove:

\begin{theorem} \label{theorem:twisted-CL}
Assume that the curvature associated to $\nabla^{P_{0}}$ is globally nondegenerate. Then, for all $\eps>0$, there exists $R>0$ (independent of $\btheta$) such that:
\[
\lambda_{0}(\btheta,\bk) \geq \left(\frac{F_{\min}}{2}-\eps \right)|\bk|, \qquad \forall \btheta \in \mathrm{U}(1)^{d},\, \forall \bk \in \widehat{G} \text{ with } |\bk|>R.
\]
\end{theorem}

The difference with \cite[Theorem~5.1.1(i)]{Cekic-Lefeuvre-24} is that now we need a uniform estimate in $\btheta$, which does not follow directly from the cited result. We claim that Theorem \ref{theorem:twisted-CL} follows from the following result.

\begin{proposition} \label{prop:open-bound}
    Let $\mathbf{l} \in \partial_{\infty}\mathfrak{a}_{+}$ be such that $F_{\min}(\mathbf{l})>0$. Then, for all $\eps>0$, there exists $R>0$ such that
    \[
    \lambda_{0}(\btheta,\bk)>\left( \frac{F_{\min}(\mathbf{l})}{2}-\eps \right) |\bk|, \quad \forall \btheta \in \mathrm{U}(1)^{d},\, \forall \bk \in \widehat{G}, \, |\bk|>R \text{ with } \left| \frac{\bk}{|\bk|}-\mathbf{l} \right|<\frac{1}{R}.
    \]
\end{proposition}

\begin{proof}[Proof of Theorem \ref{theorem:twisted-CL}]
The nondegeneracy assumption gives $F_{\min}(\mathbf{l})>0$ for any $\mathbf{l} \in \partial_{\infty}\mathfrak{a}_{+}$. By compactness, this gives $F_{\min}>0$. By compactness and Proposition \ref{prop:open-bound}, there exist $\mathbf{l}_{1},\ldots,\mathbf{l}_{N} \in \partial_{\infty}\mathfrak{a}_{+}$, and $R_{\mathbf{l}_{1}},\ldots,R_{\mathbf{l}_{N}}$ (using the same fixed $\eps$), which we can assume bigger than 1, such that 
\[ \partial_{\infty}\mathfrak{a}_{+}=\bigcup_{j=1}^{N}B \left( \mathbf{l}_{j},\frac{1}{R_{\mathbf{l}_{j}}} \right). 
\]
Let $R$ be the maximum of such $R_{\mathbf{l}_{j}}$. We claim that this $R$ satisfies the conditions of the theorem. Indeed, take $\bk \in \widehat{G}$ with $|\bk|>R$. Since $\bk/|\bk| \in \mathbb{S}^{d-1} \cap (\R^{a} \times \R_{+}^{b}) \simeq \partial_{\infty}\mathfrak{a}_{+}$, there exists $\mathbf{l}_{j}$ such that $\bk/|\bk| \in B(\mathbf{l}_{j},1/R_{\mathbf{l}_{j}})$. Hence, for any $\btheta \in \mathrm{U}(1)^{d}$ we have
\[
\lambda_{0}(\btheta,\bk)>\left( \frac{F_{\min}(\mathbf{l}_{j})}{2}-\eps \right)|\bk| \geq \left( \frac{F_{\min}}{2}-\eps \right)|\bk|.
\]    
\end{proof}

Now we deal with Proposition \ref{prop:open-bound}. First, we will need the following identity.

\begin{lemma} \label{lemma:identity}
Fix $C>0$, $\chi \in C^{\infty}(F_{0},[0,1])$. Let $X^{\mathbb{H}_{F_{0}}}, Y^{\mathbb{H}_{F_{0}}} \in C^{\infty}(F_{0}, \mathbb{H}_{F_{0}})$ be horizontal basic vector fields such that $|X|,|Y| \leq 1$ on $M_{0}$, and set $\mathbf{X}_{\btheta,\bk}:= \iota_{X^{\mathbb{H}_{F_{0}}}} \nabla_{\btheta,\bk}, \mathbf{Y}_{\btheta,\bk}=\iota_{Y^{\mathbb{H}_{F_{0}}}} \nabla_{\btheta,\bk}$. Consider a family of sections $u_{\btheta,\bk} \in C_{\hol}^{\infty}(F_{0},p_{F_{0}}^{*}L_{\btheta}\otimes \Lk)$ such that
\begin{equation} \label{eq:normalization}
    \langle\Delta_{\btheta,\bk} u_{\btheta,\bk}, u_{\btheta,\bk}\rangle_{L^2} \leq C|\bk|, \quad \|u_{\btheta,\bk}\|_{L^2}=1 .
\end{equation}
Then the following equality holds:
\begin{equation} \label{eq:control-identity}
    2 \Im \langle\mathbf{X}_{\btheta,\bk} u_{\btheta,\bk}, \chi \mathbf{Y}_{\btheta,\bk} u_{\btheta,\bk}\rangle_{L^2}=-i\langle \bk \cdot \mathbf{F}_{\overline{\nabla}} (X^{\mathbb{H}_{F_{0}}}, Y^{\mathbb{H}_{F_{0}}}) u_{\btheta,\bk}, \chi u_{\btheta,\bk}\rangle_{L^2}+\mathcal{O}_{C, X, Y, \chi}(|\bk|^{1 / 2}),
\end{equation}
where the remainder is uniform in $\btheta$.
\end{lemma}

\begin{proof}
We will begin by deriving certain identities. Since $\eta_{\btheta}$ is closed, we obtain that $F_{\overline{\nabla}_{\btheta,\bk}}=F_{\overline{\nabla}_{\bk}}+id(p_{F_{0}}^{*}\eta_{\btheta})=\bk \cdot \mathbf{F}_{\overline{\nabla}}$, where $\overline{\nabla}_{\btheta,\bk}=D_{\bk}^{\mathrm{Chern}}+\nabla_{\btheta,\bk}$ (where we used that $p_{F_{0}}^{*}\eta_{\btheta}$ vanishes on vertical vectors, so that $D_{\bk}^{\mathrm{Chern}}$ is the same as in the untwisted case), $\nabla_{\btheta,\bk}$ was introduced in \eqref{eq:operator-floquet-borel-weil}, while $\overline{\nabla}_{\bk}$ and $\bk \cdot \mathbf{F}_{\overline{\nabla}}$ were recalled in \eqref{eq:globa-connection} and \eqref{eq:bar-curvature}, respectively. Then, 
\begin{equation} \label{eq:commutation}
    \mathbf{X}_{\btheta,\bk}\mathbf{Y}_{\btheta,\bk}-\mathbf{Y}_{\btheta,\bk}\mathbf{X}_{\btheta,\bk}=\bk \cdot \mathbf{F}_{\overline{\nabla}}(X^{\mathbb{H}_{F_{0}}},Y^{\mathbb{H}_{F_{0}}})+(\overline{\nabla}_{\btheta,\bk})_{[X^{\mathbb{H}_{F_{0}}},Y^{\mathbb{H}_{F_{0}}}]}.
\end{equation}

On the other hand, by decomposing into horizontal and vertical parts $[X^{\mathbb{H}_{F_{0}}},Y^{\mathbb{H}_{F_{0}}}]=[X,Y]^{\mathbb{H}_{F_{0}}}+[X^{\mathbb{H}_{F_{0}}},Y^{\mathbb{H}_{F_{0}}}]^{\mathbb{V}_{F_{0}}}$, we have
\begin{equation} \label{eq:nabla-horizotal-decomposition}
    (\overline{\nabla}_{\btheta,\bk})_{[X^{\mathbb{H}_{F_{0}}},Y^{\mathbb{H}_{F_{0}}}]}=(\nabla_{\btheta,\bk})_{[X,Y]^{\mathbb{H}_{F_{0}}}}+(D_{\bk}^{\mathrm{Chern}})_{[X^{\mathbb{H}_{F_{0}}},Y^{\mathbb{H}_{F_{0}}}]^{\mathbb{V}_{F_{0}}}},
\end{equation}
where we used that $p_{F_{0}}^{*}\eta_{\btheta}$ vanishes on vertical vectors, so that $D_{\bk}^{\mathrm{Chern}}$ is the same as in the untwisted case, see \eqref{eq:chern} for its definition. Hence, \eqref{eq:commutation} and \eqref{eq:nabla-horizotal-decomposition} give
\[
\mathbf{X}_{\btheta,\bk}\mathbf{Y}_{\btheta,\bk}-\mathbf{Y}_{\btheta,\bk}\mathbf{X}_{\btheta,\bk}=\bk \cdot \mathbf{F}_{\overline{\nabla}}(X^{\mathbb{H}_{F_{0}}},Y^{\mathbb{H}_{F_{0}}})+(\nabla_{\btheta,\bk})_{[X,Y]^{\mathbb{H}_{F_{0}}}}+(D_{\bk}^{\mathrm{Chern}})_{[X^{\mathbb{H}_{F_{0}}},Y^{\mathbb{H}_{F_{0}}}]^{\mathbb{V}_{F_{0}}}}.
\]
Applying this to $u_{\btheta,\bk}$, and pairing with $\chi u_{\btheta,\bk}$, we obtain
\begin{equation}
    \label{eq:commutation2}
    \begin{aligned}
    \langle \mathbf{X}_{\btheta,\bk}&\mathbf{Y}_{\btheta,\bk}u_{\btheta,\bk},\chi u_{\btheta,\bk}\rangle_{L^{2}}-\langle \mathbf{Y}_{\btheta,\bk}\mathbf{X}_{\btheta,\bk} u_{\btheta,\bk},\chi u_{\btheta,\bk} \rangle_{L^{2}}\\
    =&\langle \bk \cdot \mathbf{F}_{\overline{\nabla}} (X^{\mathbb{H}_{F_{0}}},Y^{\mathbb{H}_{F_{0}}})u_{\btheta,\bk},\chi u_{\btheta,\bk}\rangle_{L^{2}}+\langle(\nabla_{\btheta,\bk})_{[X,Y]^{\mathbb{H}_{F_{0}}}}u_{\btheta,\bk},\chi u_{\btheta,\bk} \rangle_{L^{2}} \\
    &+\langle (D_{\bk}^{\mathrm{Chern}})_{[X^{\mathbb{H}_{F_{0}}},Y^{\mathbb{H}_{F_{0}}}]^{\mathbb{V}_{F_{0}}}}u_{\btheta,\bk},\chi u_{\btheta,\bk} \rangle_{L^{2}}.
    \end{aligned}
\end{equation}
This will be our main identity. To obtain \eqref{eq:control-identity}, we will bound certain terms that appear in \eqref{eq:commutation2}.

First, since $\nabla_{\btheta,\bk}$ is a Hermitian connection, we have that the (formal) adjoint of $\mathbf{X}_{\btheta,\bk}$ satisfies $\mathbf{X}_{\btheta,\bk}^{*}=-\mathbf{X}_{\btheta,\bk}-\mathrm{div}(X^{\mathbb{H}_{F_{0}}})$. Then, the left hand side of \eqref{eq:commutation2} is
\begin{equation}
    \label{eq:commutation3}
    \begin{aligned}
    &\langle \mathbf{X}_{\btheta,\bk} u_{\btheta,\bk},\chi \mathbf{Y}_{\btheta,\bk}u_{\btheta,\bk} \rangle_{L^{2}}-\langle \mathbf{Y}_{\btheta,\bk}u_{\btheta,\bk},\chi \mathbf{X}_{\btheta,\bk}u_{\btheta,\bk} \rangle_{L^{2}}\\
    &+\langle \mathbf{X}_{\btheta,\bk}u_{\btheta,\bk},(Y^{\mathbb{H}_{F_{0}}}\chi)u_{\btheta,\bk} \rangle_{L^{2}}-\langle \mathbf{Y}_{\btheta,\bk} u_{\btheta,\bk},(X^{\mathbb{H}_{F_{0}}}\chi) u_{\btheta,\bk} \rangle_{L^{2}} \\
    &+\langle \mathbf{X}_{\btheta,\bk}u_{\btheta,\bk},\mathrm{div}(Y^{\mathbb{H}_{F_{0}}})\chi u_{\btheta,\bk} \rangle_{L^{2}}-\langle \mathbf{Y}_{\btheta,\bk} u_{\btheta,\bk},\mathrm{div}(X^{\mathbb{H}_{F_{0}}})\chi u_{\btheta,\bk} \rangle_{L^{2}}.
    \end{aligned}
\end{equation}
The first two terms give the left-hand side of \eqref{eq:control-identity}.

Now we proceed to bound the last four terms in \eqref{eq:commutation3}, together with the middle term on the right-hand side of \eqref{eq:commutation2}. Using \eqref{eq:normalization} and Cauchy--Schwarz inequality, we obtain that for any $Z \in C^{\infty}(F_{0},\mathbb{H}_{F_{0}})$ and $f \in C^{\infty}(F_{0})$, 
\[ 
|\langle \iota_{Z}\nabla_{\btheta,\bk}u_{\btheta,\bk},fu_{\btheta,\bk} \rangle_{L^{2}}|\leq C^{1/2} \|Z\|_{L^{\infty}} \|f\|_{L^{\infty}} |\bk|^{1/2}.
\]
Since all the aforementioned terms are as in the left-hand side of this inequality, we obtain that they are $\mathcal{O}(|\bk|^{1/2})$. 

Finally, we have to deal with the last term on the right-hand side of \eqref{eq:commutation2}. First, observe that since $u_{\btheta,\bk}$ is holomorphic, we obtain $\overline{\partial}_{\bk}u_{\btheta,\bk}=0$. Furthermore, since the term in $p_{F_{0}}^{*}L_{\btheta}$ is the pullback from the base, it is trivial in vertical directions. Hence, $(D_{\bk}^{\mathrm{Chern}})_{Z}u_{\btheta,\bk}=(\partial_{\bk})_{Z^{1,0}}u_{\btheta,\bk}$. Integrating this against $\chi u_{\btheta,\bk}$, and using $((\partial_{\bk})_{Z^{1,0}})^{*}=-(\overline{\partial}_{\bk})_{Z^{0,1}}-\mathrm{div}Z^{0,1}$ together with \eqref{eq:normalization} and the fact that $(\overline{\partial}_{\bk})_{Z^{0,1}}u_{\btheta,\bk}=0$ (since the section is holomorphic), we obtain
\[
\langle (D_{\bk}^{\mathrm{Chern}})_{Z}u_{\btheta,\bk},\chi u_{\btheta,\bk}\rangle_{L^{2}}=-\langle u_{\btheta,\bk},(Z^{0,1}\chi+\chi\mathrm{div}Z^{0,1})u_{\btheta,\bk}\rangle_{L^{2}}=\mathcal{O}(1),
\]
finishing the proof.    
\end{proof}

\begin{proof}[Proof of Proposition \ref{prop:open-bound}]
We will prove that for any $\eps>0$, there exists $R>0$ such that for all $\btheta \in \mathrm{U}(1)^{d}$ and all $\bk \in \widehat{G}$ with $|\bk|>R$ and $|\bk/|\bk|-\mathbf{l}|<1/R$ and all $u_{\btheta,\bk} \in C_{\hol}^{\infty}(F_{0},p_{F_{0}}^{*}L_{\btheta} \otimes \Lk)$, 
\[ 
\|\nabla_{\btheta,\bk}u_{\btheta,\bk}\|_{L^{2}}^{2}\geq \left( \frac{F_{\min}(\mathbf{l})}{2}-\eps \right)|\bk| \|u_{\btheta,\bk}\|_{L^{2}}^{2}.
\]
By homogeneity we may assume $\|u_{\btheta,\bk}\|_{L^{2}}=1$. Also, note that if $\|\nabla_{\btheta,\bk}u_{\theta,\bk}\|^{2}>C|\bk|$, there is nothing to prove, so we may assume the hypotheses of Lemma \ref{lemma:identity}.

By definition of $F_{\min}(\mathbf{l})$, we have that for all $w_{0} \in F_{0}$, there exists an open neighborhood $U$ of $w_{0}$, and $X^{\mathbb{H}_{F_{0}}},Y^{\mathbb{H}_{F_{0}}}\in C^{\infty}(U,\mathbb{H}_{F_{0}})$ basic horizontal vector fields such that $|X^{\mathbb{H}_{F_{0}}}|=|Y^{\mathbb{H}_{F_{0}}}|=1$ pointwise, and $R_{w_{0}}>0$ such that:
\begin{equation} \label{eq:local-bound}
    \frac{-i}{|\bk|}\bk \cdot \mathbf{F}_{\overline{\nabla}}(w)(X^{\mathbb{H}_{F_{0}}},Y^{\mathbb{H}_{F_{0}}})>F_{\min}(\mathbf{l})-\eps, \quad \forall \bk \in \widehat{G}, \, |\bk|>R, \, \left| \frac{\bk}{|\bk|}-\mathbf{l} \right|<\frac{1}{R_{w_{0}}}, \, \forall w \in U.
\end{equation}
By compactness of $F_{0}$, there exists a finite cover $F_{0}=\cup_{j=1}^{N}U_{w_{j}}$, where each $U_{w_{j}}$ is an open set as in \eqref{eq:local-bound}. Let $R=\max R_{j}$ where $R_{j}$ correspond to $U_{w_{j}}$, denote by $X_{j}^{\mathbb{H}_{F_{0}}}, Y_{j}^{\mathbb{H}_{F_{0}}}$ the corresponding vector fields, and let $\sum_{j=1}^{N}\chi_{j}=1$ be a partition of unity subordinated to the sets $U_{w_{j}}$, with $\chi_{j} \geq 0$. Let us write $\mathbf{X}_{j,\btheta,\bk}=(\nabla_{\btheta,\bk})_{X_{j}^{\mathbb{H}_{F_{0}}}}$ and $\mathbf{Y}_{j,\btheta,\bk}=(\nabla_{\btheta,\bk})_{Y_{j}^{\mathbb{H}_{F_{0}}}}$. Now let $u_{\btheta,\bk}$ such that \eqref{eq:normalization} holds. Then, using \eqref{eq:control-identity} and \eqref{eq:local-bound}, we find that
\[
\begin{split}
    \sum_{j=1}^{N}\frac{2}{|\bk|}&\Im \langle \mathbf{X}_{j,\btheta,\bk}u_{\btheta,\bk},\chi_{j}\mathbf{Y}_{j,\btheta,\bk}u_{\btheta,\bk}\rangle_{L^{2}} \\
    =&\sum_{j=1}^{N}\frac{-i}{|\bk|} \langle \bk \cdot \mathbf{F}_{\overline{\nabla}}(X_{j}^{\mathbb{H}_{F_{0}}},Y_{j}^{\mathbb{H}_{F_{0}}})u_{\btheta,\bk},\chi_{j}u_{\btheta,\bk}\rangle_{L^{2}} +\|u_{\btheta,\bk}\|_{L^{2}}^{2}\mathcal{O}(|\bk|^{-1/2}) \\
    \geq & (F_{\min}(\mathbf{l})-\eps)\sum_{j=1}^{N}\langle u_{\btheta,\bk},\chi_{j}u_{\btheta,\bk}\rangle_{L^{2}}+\|u_{\btheta,\bk}\|_{L^{2}}^{2}\mathcal{O}(|\bk|^{-1/2}) \\
     \geq & (F_{\min}(\mathbf{l})-2\eps),
\end{split}
\]
where in the last inequality we used the normalization of the sections, and that up to increasing $R$, we have $\mathcal{O}(|\bk|^{-1/2})\geq -\eps$. On the other hand, the left-hand side can be bounded using the fact that $|X_{j}^{\mathbb{H}_{F_{0}}}|,|Y_{j}^{\mathbb{H}_{F_{0}}}| \leq 1$:
\[
\left| \sum_{j=1}^{N}\frac{2}{|\bk|}\Im \langle \mathbf{X}_{j,\btheta,\bk}u_{\btheta,\bk},\chi_{j}\mathbf{Y}_{j,\btheta,\bk}u_{\btheta,\bk}\rangle_{L^{2}}  \right| \leq \frac{2}{|\bk|} \sum_{j=1}^{N}\int_{F_{0}}\chi_{j}|\nabla_{\btheta,\bk}u_{\btheta,\bk}|^{2}=\frac{2}{|\bk|} \|\nabla_{\btheta,\bk}u_{\btheta,\bk}\|_{L^{2}}^{2}.
\]
This finishes the proof.
\end{proof}

We will also need the Borel--Weil version of Lemma \ref{lemma:lambda-theta}.

\begin{lemma}\label{lemma:lambda-theta-k}
Assume that the holonomy group $\mathrm{Hol}(P,\nabla)$ is dense in $G \times \Z^{d}$. Then, $\lambda_{0}(\btheta,\bk)=0$ if and only if $\btheta=0$ and $\bk=\mathbf{0}$.
\end{lemma}

\begin{proof}
We recall that the holonomy was defined in \eqref{eq:holonomy}, and that we will write $(g_{\gamma},\n_{\gamma})$ to emphasize that the element comes from the parallel transport associated to the loop $\gamma$ (where $\n_{\gamma}=\rho(\gamma)$). Define the character $\mathbf{e}_{\btheta}(\n):=e^{i\n\cdot \btheta}$. 

Now, suppose that $\lambda_{0}(\btheta,\bk)=0$. Then, there exists a nonzero $f_{\btheta,\bk}\in C_{\mathrm{hol}}^{\infty}(F_{0},p_{F_{0}}^{*}L_{\btheta}\otimes \Lk)$ such that $\nabla_{\btheta,\bk}f_{\btheta,\bk}=0$. Since $f_{\btheta,\mathbf k}$ is parallel, parallel transport along every loop based at $x_{0}$ fixes its value $\xi$ in the fiber over $x_{0}$. We now write $\xi=f_{\btheta,\bk}|_{(F_{0})_{x_{0}}}$, and note that 
\[
\xi \in H^{0}((F_{0})_{x_{0}},p_{F_{0}}^{*}L_{\btheta}\otimes \Lk|_{(F_{0})_{x_{0}}}) \simeq (L_{\btheta})_{x_{0}} \otimes H^{0}((F_{0})_{x_{0}},\Lk|_{(F_{0})_{x_{0}}}) \simeq (L_{\btheta})_{x_{0}} \otimes H^{0}(G/T,\mathbf{J}^{\otimes \bk}).
\]
By the Borel--Weil Theorem, $H^{0}(G/T,\mathbf{J}^{\otimes \bk})$ is an irreducible representation of $G$, which we denote by $\rho_{\bk}$. The action of $G \times \Z^{d}$ on $L_{\btheta}\otimes H^{0}(G/T,\mathbf{J}^{\otimes \bk})$ is then given by $\varrho_{\btheta,\bk}(g,\n):=\mathbf{e}_{\btheta}(\n)\otimes\rho_{\bk}(g)$. Of course, after trivializing $(L_{\btheta})_{x_{0}}$, this is just $\varrho_{\btheta,\bk}(g,\n)=\mathbf{e}_{\btheta}(\n)\rho_{\bk}(g)$. Since $\xi \neq 0$ and $(L_{\btheta})_{x_{0}}$ is one dimensional, after choosing a nonzero $\xi_{1} \in (L_{\btheta})_{x_{0}}$, we can write $\xi=\xi_{1} \otimes \xi_{2}$, where $\xi_{2} \in H^{0}(G/T,\mathbf{J}^{\otimes \bk})$ is nonzero too. If we take $\gamma \in \ker \rho$, then $\alpha_{\btheta}(\n_{\gamma})=1$ since $\n_{\gamma}=\rho(\gamma)=0$. Thus, we obtain
\[
\xi_{1} \otimes \xi_{2}=\xi=\varrho_{\btheta,\bk}(g_{\gamma},\n_{\gamma})\xi=\mathbf{e}_{\btheta}(\n_{\gamma})\xi_{1} \otimes \rho_{\bk}(g_{\gamma})\xi_{2}=\xi_{1} \otimes \rho_{\bk}(g_{\gamma})\xi_{2}.
\]
Since $\xi_{1} \neq 0$, we conclude that $\xi_{2}$ is fixed by $H_{G}=\{ g_{\gamma} \in G:\rho(\gamma)=0 \}$. Now note that since $\Z^{d}$ is discrete, $G \times \{\mathbf{0}\} \leq G \times \Z^{d}$ is open. Hence, the density of the holonomy group implies $\overline{\mathrm{Hol}(P,\nabla) \cap (G \times \{\mathbf{0}\})}=G \times \{\mathbf{0}\}$. This implies that $\overline{H_{G}}=G$. By continuity, we obtain that $\xi_{2}$ is fixed for all $g \in G$, showing that the irreducible representation has a nonzero fixed element, a contradiction for $\bk \neq \mathbf{0}$. Hence, $\bk=\mathbf{0}$. Arguing as before, this implies $\alpha_{\btheta}(\rho(\gamma))\xi_{1}=\xi_{1}$, but since $\xi_{1} \neq 0$, and $\rho$ is surjective, we obtain $1=\mathbf{e}_{\btheta}(\n)=e^{i \n \cdot \btheta}$ for all $\n \in \Z^{d}$, showing that $\btheta=\mathbf{0}$. 

The reverse implication is obvious.
\end{proof}

\begin{remark} \label{remark:holonomy}
Observe that the density of the full holonomy group is stronger than what is required in Lemma \ref{lemma:lambda-theta-k}. Indeed, the proof only uses that
$\overline{\mathrm{pr}_G(\operatorname{Hol}(P, \nabla) \cap(G \times\{\mathbf{0}\}))}=G$ (where $\mathrm{pr}_{G} \colon G \times \Z^{d} \to G$ is the projection onto the first entry), together with the surjectivity of $\rho$. Equivalently, it is enough to assume that the holonomy elements with trivial $\Z^d$-component have dense $G$-projection.    
\end{remark}

Finally, we prove the spectral gap for $\Delta_{\btheta,\bk}$.

\begin{proof}[Proof of Theorem \ref{thm:spectral-gap}]
By Theorem \ref{theorem:twisted-CL}, we have that there exists $R>0$ such that for all $\btheta \in \mathrm{U}(1)^{d}$ and for all $\bk \in \widehat{G}$ with $|\bk|>R$ we have $\lambda_{0}(\btheta,\bk) \geq F_{\min}R/4$, where we took $\eps=F_{\min}/4$. By continuity and Lemma \ref{lemma:lambda-theta-k}, we have that for $0<|\bk| \leq R$, there exists a constant $C_{R}$ so that $\lambda_{0}(\btheta,\bk)>C_{R}$ for all $\btheta \in \mathrm{U}(1)^{d}$. Then, simply take $C=\min\{F_{\min}R/4,C_{R}\}$.    
\end{proof}

\subsection{Proof of Theorem \ref{theorem:main2}}

We use a similar setting as in Theorem \ref{theorem:main1-local}. This, together with Proposition \ref{proposition:borel-floquet-weil} gives
\[
\begin{split}
\langle e^{-t\Delta_{P}^{\HH}}f , g \rangle_{L^2(P)} & = \underbrace{\dfrac{1}{(2\pi)^d}\int_{\btheta \in \mathrm{U}(1)^d} \langle e^{-t\Delta_{\btheta, \bk = \mathbf{0}}}f_{\btheta, \bk=\mathbf{0}}, g_{\btheta, \bk=\mathbf{0}}\rangle_{L^2(M_0, L_{\btheta})} \dd\btheta}_{=A} \\
& + \underbrace{\dfrac{1}{(2\pi)^d} \sum_{\bk \in \widehat{G} \setminus \{\mathbf{0}\}} d_{\bk} \sum_{i = 1}^{d_{\bk}} \int_{\btheta \in \mathrm{U}(1)^d} \langle e^{-t\Delta_{\btheta,\bk}}f_{\btheta,\bk, i}, g_{\btheta,\bk, i}\rangle_{L^2(F_0, p_{F_{0}}^*L_{\btheta}\otimes \Lk)} \dd\btheta}_{=B}.
\end{split}
\]
For $f \in C_{\mathrm{comp}}^{\infty}(P)$, define $\mathcal{F}_{0}(f)=f_{\bk=\mathbf{0}}$, i.e.,
\[ 
\mathcal{F}_{0}(f)=\int_{P_x} f(w \cdot g) \dd g, \qquad x \in M
\]
where $w \in P_x$ is arbitrary. In the term $A$, $f_{\btheta, \bk=\mathbf{0}},g_{\btheta, \bk=\mathbf{0}} \in C^\infty_{\hol}(F_0, p_{F_{0}}^*L_{\btheta})$ are identified with elements of $C^\infty(M_0,L_{\btheta})$. Hence, we can apply Theorem \ref{theorem:main1} to obtain
\[
t^{d/2}A=\kappa (\mathbf{1}_{M}\otimes \mathbf{1}_{M})(\mathcal{F}_{0}(f)\otimes \mathcal{F}_{0}(g))+\sum_{j=1}^{N-1}t^{-j}\langle \mathcal{C}_{j},\mathcal{F}_{0}(f)\otimes\mathcal{F}_{0}(g)\rangle+R_{N}^{M}(t,\mathcal{F}_{0}(f),\mathcal{F}_{0}(g)).
\]
Note that $(\mathbf{1}_{M}\otimes \mathbf{1}_{M})(\mathcal{F}_{0}(f)\otimes \mathcal{F}_{0}(g))=(\mathbf{1}_{P}\otimes \mathbf{1}_{P})(f,g)$.

To deal with $B$, we first see that the spectral gap for $\Delta_{\btheta,\bk}$ (Theorem \ref{thm:spectral-gap}) together with the Cauchy--Schwarz inequality give the following bound: there exists a positive constant $C$ such that for any $\btheta \in \mathrm{U}(1)^{d}$ and $\bk \neq \mathbf{0}$, we have 
\[
|\langle e^{-t\Delta_{\btheta,\bk}}f,g \rangle_{L^2(F_0, p_{F_{0}}^*L_{\btheta}\otimes \Lk)}| \leq e^{-Ct}\|f\|_{L^2(F_0, p_{F_{0}}^*L_{\btheta}\otimes \Lk)} \|g\|_{L^2(F_0, p_{F_{0}}^*L_{\btheta}\otimes \Lk)}. 
\]
We claim that this gives the result. Indeed, we have
\[
\begin{split}
    |B| & \leq ce^{-Ct} \sum_{\bk \in \widehat{G} \setminus \{\mathbf{0}\}} d_{\bk} \sum_{i = 1}^{d_{\bk}} \int_{\btheta \in \mathrm{U}(1)^d} \|f_{\btheta,\bk, i} \|_{L^2(F_0, p_{F_{0}}^*L_{\btheta}\otimes \Lk)}  \|g_{\btheta,\bk, i} \|_{L^2(F_0, p_{F_{0}}^*L_{\btheta}\otimes \Lk)} \dd\btheta \\
    & \leq ce^{-Ct} \int_{\mathrm{U}(1)^{d}} \|f_{\btheta}\|_{L^{2}(P_{0},p_{0}^{*}L_{\btheta})} \|g_{\btheta}\|_{L^{2}(P_{0},p_{0}^{*}L_{\btheta})}~ \dd \btheta \leq ce^{-Ct}\|f\|_{L^{2}(P)}\|g\|_{L^{2}(P)}.
\end{split}
\]
The result now follows from the Schwartz kernel theorem.

\subsection{Proof of Theorem \ref{theorem:main2-local}}

We first establish the analog of Lemma \ref{lemma:Delta_theta_local}. 

\begin{lemma}
\label{lemma:nontrivial-G-local}
Let $p_{0}\colon P_{0}\to M_{0}$ be the bundle projection. Let
$\mathsf{P}_{G}F(u) := \int_{G}F(ug)~\dd g$ be the orthogonal projection onto the right $G$-invariant functions, where $\dd g$ is Haar probability measure. Note that $\mathsf{P}_{G}=p_{0}^{*}\mathcal{F}_{0}$.

For every $\ell_{u},\ell_{v}\geq 0$, there exist constants $c,C>0$ such that the Schwartz kernel $E_{\btheta}(t,u,v)$ of $e^{-t\Delta_{\btheta}^{\mathbb H}} (\mathrm{Id}-\mathsf{P}_{G})$ satisfies, for $t \geq 1$:
\[
\sup_{\btheta \in \mathrm{U}(1)^{d}} \|E_{\btheta}(t)\|_{C^{\ell_{u},\ell_{v}}(P_{0}\times P_{0})} \leq Ce^{-ct}.
\]
The estimate is understood in any local trivialization of the family of flat line bundles $L_{\btheta}\to M_{0}$.    
\end{lemma}

\begin{proof}
By Theorem \ref{thm:spectral-gap} we obtain
\begin{equation}
\label{eq:nontrivial-G-L2-gap}
\|
e^{-t\Delta_{\btheta}^{\mathbb H}}
(\mathrm{Id}-\mathsf P_{G})
\|_{L^{2}\to L^{2}}
\leq e^{-Ct}.
\end{equation}
uniformly in $\btheta \in \mathrm{U}(1)^{d}$. By the discussions in \cite{Cekic-Lefeuvre-24}, the nondegeneracy. of the curvature implies that H\"ormander's bracket condition holds, implying subellipticity. Furthermore, since $\Delta_{\btheta}^{\mathbb{H}}$ and $\Delta_{\mathbf{0}}^{\mathbb{H}}$ coincide up to terms of order 1, and $P_{0}$ and $\mathrm{U}(1)^{d}$ are compact, the subelliptic estimates are uniform in $\btheta$, i.e., for any integer $m \geq 0$, there exist $r_{m},C_{m} \geq 0$ with 
\begin{equation}
\label{eq:subelliptic}
\|u\|_{H^{m}(P_{0})} \leq C_{m} \|
(\mathrm{Id}+\Delta_{\btheta}^{\mathbb{H}})^{r_{m}}u \|_{L^{2}(P_{0})}.
\end{equation}
The estimate is understood locally after trivializing the flat line bundle
$L_{\btheta}$. Its constants can be chosen uniformly in $\btheta$. Indeed, the dependence in $\btheta$ in $\Delta_{\btheta}^{\mathbb{H}}$ is on the lower order terms. Furthermore, those coefficients depend smoothly on $\btheta$. Uniformity then follows from the compactness of
$P_{0}$ and $\mathrm{U}(1)^{d}$. Hence, the proof follows that of Lemma \ref{lemma:Delta_theta_local}, mutatis mutandis.    
\end{proof}

Now we prove Theorem \ref{theorem:main2-local}. As in the previous proof, we take a finite open cover $\mathrm U(1)^{d} =
U_{0}\cup\cdots\cup U_{r}$ by contractible trivializing neighborhoods for the family $L_{\btheta}\to M_{0}$, and choose a subordinate partition of unity $\sum\chi_{i}=1$. For $\btheta \in U_{i}$, let $s_{\btheta,i}$ be the corresponding trivializing section on $M$. Abusing of the notation, we also denote its pullback to $P$ by $s_{\btheta,i}$. Let $H_{\btheta,i}^{P,\mathbb H}(t,u_{0},v_{0})$ be the kernel of the twisted horizontal heat operator on $P_{0}$ in this trivialization. Then:
\begin{equation} \label{eq:kernel-P}
    H_{P}^{\mathbb H}(t,u,v) = \frac{1}{(2\pi)^{d}} \sum_{i=0}^{r} \int_{U_{i}} \chi_{i}(\btheta) s_{\btheta,i}(p(u)) \overline{ s_{\btheta,i}(p(v))} 
H_{\btheta,i}^{P,\mathbb{H}} (t,\pi(u),\pi(v)) ~\dd\btheta.
\end{equation}
The projection $\mathsf P_{G}$ onto the trivial $G$-representation commutes with $\Delta_{\btheta}^{\mathbb H}$. Since Haar measure is normalized, pullback by $p_{0}$ identifies $L^{2}(M_{0},L_{\btheta})$ isometrically with
$\mathrm{Ran} \mathsf{P}_{G}$, and under this identification one has $ \Delta_{\btheta}^{\mathbb{H}}|_{\mathrm{Ran}\mathsf{P}_{G}} = \Delta_{\boldsymbol{\theta}}$. Therefore,
\begin{equation} \label{eq:kernel-P-to-M}
    H_{\btheta,i}^{P,\mathbb{H}}(t,u_{0},v_{0}) = H_{\btheta,i}^{M}(t,p_{0}(u_{0}),p_{0}(v_{0})) + E_{\btheta,i}(t,u_{0},v_{0}),
\end{equation}
where $E_{\btheta,i}$ is the kernel of $e^{-t\Delta_{\btheta}^{\mathbb{H}}} (\mathrm{Id}-\mathsf{P}_{G})$, and $H_{\btheta,i}^{M}(t,x,y)$ denote the heat kernel of $\Delta_{\btheta}$ on $M_{0}$ in the same trivialization. Setting
\[
E_{P}(t,u,v)=\frac{1}{(2\pi)^{d}} \sum_{i=0}^{r} \int_{U_{i}} \chi_{i}(\btheta) s_{\btheta,i}(p(u)) \overline{s_{\btheta,i}(p(v))} E_{\btheta,i}( t,\pi(u),\pi(v)) ~\dd\btheta,
\]
we see that Lemma \ref{lemma:nontrivial-G-local} together with the smooth dependence of the sections $s_{\btheta,i}$, and the compactness of $K_{u}\times K_{v}$, imply $\|E_{P}(t)\|_{C^{\ell_{u},\ell_{v}}(K_{u}\times K_{v})} \leq Ce^{-ct}$. 

On the other hand, note that the contribution of the first term on the right-hand side of \eqref{eq:kernel-P-to-M}, after integration in \eqref{eq:kernel-P} is $H_{M}(t,p(u),p(v))$. Now write $K_{x}=p(K_{u}), K_{y}:=p(K_{v}) \subset M$. Applying Theorem
\ref{theorem:main1-local} to $K_{x}$ and $K_{y}$, we obtain smooth functions
$k_{j}^{M}$, $j=1,\ldots,N-1$, such that 
\[
t^{d/2}H_{P}^{\mathbb{H}}(t,u,v)=\kappa + \sum_{j=1}^{N-1}t^{-j} k_{j}^{M}(p(u),p(v))+ R_{N}^{M}(t,p(u),p(v)) + t^{d/2}E_{P}(t,u,v).
\]
Now define $k_{j}^{P}(u,v) := k_{j}^{M}(p(u),p(v))$, and $R_{N}^{P}(t,u,v) := R_{N}^{M}(t,p(u),p(v))+ t^{d/2}E_{P}(t,u,v)$. The proof is finished after we show the promised bounds for these operators. However, this follows from the bounds in Theorem \ref{theorem:main1-local}, the one of $E_{P}$, and the fact that $p\colon P\to M$ is smooth, and $K_{u},K_{v}$ are compact.

\subsection{Proof of Theorem \ref{theorem:main2-global}}

\begin{lemma}
\label{lemma:nontrivial-G-global}
Let $D_{P}:=p^{-1}(D)\subset P$. Under the standing assumptions, for every pair of integers $\ell_{u},\ell_{v} \geq 0$, there exist constants $c,C>0$ such that for $t \geq 1$
\[
\sup_{\n,\mathbf{m} \in \Z^{d}} \| E_{P} (t,\tau_{\n}(\bullet), \tau_{\mathbf{m}}(\bullet))\|_{C^{\ell_{u},\ell_{v}}(\overline{D_{P}}\times\overline{D_{P}})} \leq Ce^{-ct},
\]
where $E_{P}(t,u,v)$ denotes the contribution to $H_{P}^{\mathbb H}(t,u,v)$ of the orthogonal complement of the right $G$-invariant functions. 
\end{lemma}

\begin{proof}
After choosing a finite cover of $\mathrm{U}(1)^{d}$ associated to trivializations of $L_{\btheta} \to M_{0}$, together with a subordinate
partition of unity $\sum_i\chi_i=1$, we can write
\[
E_{P}( t,\tau_{\n}u_{0},\tau_{\mathbf{m}}v_{0}) = \frac{1}{(2\pi)^d} \sum_{i=0}^{r} \int_{U_i} \chi_{i}(\btheta) e^{i(\n-\mathbf{m})\cdot\btheta} s_{\btheta,i}(p(u_{0})) \overline{s_{\btheta,i}(p(v_{0}))} E_{\btheta,i} (t,\pi(u_{0}),\pi(v_{0})) ~\dd\boldsymbol{\theta}.
\]
Since $|e^{i(\n-\mathbf{m})\cdot\btheta}|=1$, the bound is independent of $\n$ and $\mathbf{m}$. Derivatives in $u_{0}$ and $v_{0}$ fall on the trivializing sections, on the fixed smooth maps $p$ and $\pi$, or on $E_{\btheta,i}$. The trivializing sections and all their relevant derivatives are uniformly bounded on the finite family of trivializations and on $\overline{D_{P}}$ (note that since $G$ is compact and $D\Subset M$, $D_P=p^{-1}(D)$ is indeed relatively compact). Therefore, the result follows from Lemma \ref{lemma:nontrivial-G-local}.    
\end{proof}

Now we prove Theorem \ref{theorem:main2-global}. Arguing as in the proof of Theorem \ref{theorem:main2-local}, and since $p \circ \tau_{\n}=\tau_{\n} \circ p$, we see that
\begin{equation} \label{eq:decomposition-global}
    H_{P}^{\mathbb {H}}(t,\tau_{\n}u_{0},\tau_{\mathbf{m}}v_{0})= H_{M}(t,\tau_{\n}p(u_{0}), \tau_{\mathbf{m}}p(v_{0})) + E_{P} (t,\tau_{\n} u_{0} ,\tau_{\mathbf{m}}v_{0}),
\end{equation}
where $E_{P}$ is as in Lemma \ref{lemma:nontrivial-G-global}. If we write  $x_{0}=p(u_{0})$, and $y_{0}:=p(v_{0})$, we have $x_{0},y_{0} \in\overline D$. Define $P_{j}^{P}(\mathbf{z},u_{0},v_{0}):=P_{j}^{M}(\mathbf{z},p(u_{0}),p(v_{0}))$ and $R_{N}^{P}(t,\mathbf{r},u_{0},v_{0})=R_{N}^{M}(t,\mathbf{r},p(u_{0}),p(v_{0}))+ t^{d/2}E_{P}(t,\tau_{\mathbf{r}}u_{0},v_{0})$. Then, the expansion follows from \eqref{eq:decomposition-global} together with Theorem \ref{theorem:main1-global}. Finally, the bound on $R_{N}^{P}$ follows from the following two bounds. First, Theorem \ref{theorem:main1-global}, the chain rule, and compactness of $\overline{D_{P}}$ imply
\[
\sup_{|\mathbf{r}|\leq R\sqrt{t}}\| R_{N}^{M}(t,\mathbf{r},p(\bullet),p(\bullet)) \|_{C^{\ell_{u},\ell_{v}} (\overline{D_{P}}\times\overline{D_{P}})} \leq C_{R,N,\ell_{u},\ell_{v}}(e^{-ct}+t^{-N/2}),
\]
while the bound on $E_{P}$ is given by Lemma \ref{lemma:nontrivial-G-global}. This finishes the proof.

\subsection{The magnetic case} \label{subsection:G=U(1)}

In this short section we show how to use our results to obtain expansions in the case when $G=\mathrm{U}(1)$ (which is not semisimple). In this case, let $L \to M_{0}$ be the Hermitian line bundle associated with the principal $\mathrm{U}(1)$-bundle $P_{0} \to M_{0}$, and let $\nabla^{L}$ be its Hermitian connection. We write $F_{\nabla^{L}}=iB$, where $B \in \Omega^{2}(M_{0})$. The main step in our proof is to obtain a spectral gap. As in \S\ref{subsection:spectral-gap}, to do so, we only need a lower bound for $\lambda_{0}(\btheta,\bk)$ for $k=\bk$ big enough.

We first recall the result in \cite{Helffer-Kordyukov-09}. There, it is shown that if $S=\{x\in M_{0} : B(x)=0\}$ is smooth compact hypersurface, and there exist a neighborhood $V$ of $S$, and a constant $ C>0$, such that for every $x \in V$ one has
\[
C^{-1} d(x, S)^r \leq |B(x)| \leq C d(x, S)^r, 
\]
then there exist constants $c_{0}>0$ and $h_{0}>0$ such that
\begin{equation}
\label{eq:HK-bound}
\|h\nabla_{k}u\|_{L^{2}}^{2} \geq c_{0} h^{\frac{2r+2}{r+2}} \|u\|_{L^{2}}^{2}, \quad 0<h<h_{0}.
\end{equation}

Note that in this case, the curvature is not globally degenerate since $B$ vanishes on $S$ (although the holonomy condition is satisfied as long as $B$ is not identically zero).

We claim that if \eqref{eq:HK-bound} holds not just for sections of $L$ (or $L^{\otimes k}$), but rather for sections of the twisted bundle $L^{\otimes k} \otimes L_{\btheta}$, then there exist constants $c>0$ and $k_{0} \geq 1$ such that:
\[
\lambda_{0}(\btheta,k) \geq c|k|^{\frac{2}{r+2}},
\qquad
\boldsymbol{\theta}\in\mathrm U(1)^{d},
\quad |k|\geq k_{0},
\]
implying a spectral gap as in \S\ref{subsection:spectral-gap}. That is, one needs a twisted version of the results in \cite{Helffer-Kordyukov-09}, similarly as we showed in \S\ref{subsection:spectral-gap} a twisted version of the results in \cite[Chapter~5]{Cekic-Lefeuvre-24}.

To prove the claim, choose a finite open cover $M_{0}=V_{1}\cup\cdots\cup V_{J}$ by contractible sets over which $L$ is unitarily trivial, so that over $V_{j}$, $\nabla^{L}=d+iA_{j}$, where $A_{j} \in \Omega^{1}(V_{j})$. Now, suppose without loss of generality that $h:=k^{-1}>0$. We see that in the above local trivialization, we have $h\nabla_{\btheta,k} = hd+iA_{j}+ih\eta_{\btheta}$. Hence, Young's inequality gives 
\[
\|h\nabla_{\btheta,k}u\|_{L^{2}}^{2} \geq \frac{1}{2}\|(hd+iA)u\|_{L^{2}}^{2} - 2h^{2} \|\eta_{\btheta}u\|_{L^{2}}^{2},
\]
where $(hd+iA)$ denotes intrinsically the semiclassical covariant
derivative associated with $\nabla^{L^{\otimes k}}$.
By Proposition \ref{prop:eta-theta}, we can bound $\eta_{\btheta}$ uniformly in the $L^{\infty}$ by a constant $C>0$ for all $\btheta \in \mathrm{U}(1)^{d}$. Then, we obtain 
\begin{equation} \label{eq:semiclassical-bound}
    \|h\nabla_{\btheta,k}u\|_{L^{2}}^{2} \geq \frac{1}{2}\|h\nabla_{k}u\|_{L^{2}}^{2} - 2h^{2}C^{2} \|u\|_{L^{2}}^{2},
\end{equation}
where $\nabla_{k}$ denotes the connection induced by $\nabla^{L}$ on $L^{\otimes k}$. Using the twisted version of \eqref{eq:HK-bound} together with \eqref{eq:semiclassical-bound}, we obtain 
\[
\|h\nabla_{\btheta,k}u\|_{L^{2}}^{2} \geq \frac{c_{0}}{2} h^{\frac{2r+2}{r+2}}\|u\|_{L^{2}}^{2} - 2C^{2}h^{2}\|u\|_{L^{2}}^{2}.
\]
Since $2-\frac{2r+2}{r+2}= \frac{2}{r+2}>0$, we may assume that $2C^{2}h^{2} \leq \frac{c_{0}}{4} h^{\frac{2r+2}{r+2}}$. Therefore, we conclude that 
\[
\langle \Delta_{\btheta,k}u,u \rangle_{L^{2}} \geq \frac{c_{0}}{4} k^{\frac{2}{r+2}} \|u\|_{L^{2}}^{2},
\]
which gives the claim. 

The proof of the asymptotic follows the same procedure as in the previous sections.

\bibliographystyle{alpha}
\bibliography{biblio}

@book {Hoermander-I-03,
    AUTHOR = {H{\"o}rmander, Lars},
     TITLE = {The analysis of linear partial differential operators. {I}},
    SERIES = {Classics in Mathematics},
      NOTE = {Distribution theory and Fourier analysis,
              Reprint of the second (1990) edition [Springer, Berlin;
              MR1065993 (91m:35001a)]},
 PUBLISHER = {Springer-Verlag, Berlin},
      YEAR = {2003},
     PAGES = {x+440},
      ISBN = {3-540-00662-1},
   MRCLASS = {35-02},
  MRNUMBER = {1996773},
       DOI = {10.1007/978-3-642-61497-2},
       URL = {https://doi.org/10.1007/978-3-642-61497-2},
}

@misc{Cekic-Lefeuvre-Munoz-Thon-26,
 author = {Ceki\'c, Mihajlo and Lefeuvre, Thibault and {Mu\~noz-Thon}, Sebasti\'an},
 title = {Decay of correlations on Abelian covers of isometric extensions of volume-preserving Anosov flows},
 year = {2026},
 howpublished = {Preprint, {arXiv}:2603.06379 [math.{DS}] (2026)},
 url = {https://arxiv.org/abs/2603.06379},
 arXiv = {arXiv:2603.06379}
}

@phdthesis{Charles-00,
  TITLE = {{Aspects semi-classiques de la quantification g{\'e}om{\'e}trique}},
  AUTHOR = {Charles, Laurent},
  URL = {https://theses.hal.science/tel-00001289},
  NOTE = {Rapporteurs: Louis Boutet de Monvel, Yves Colin de Verdi{\`e}re. Suffragants: Yvar Ekeland, Andr{\'e} Voros, San Vu Ngoc.},
  SCHOOL = {{Universit{\'e} Paris Dauphine - Paris IX}},
  YEAR = {2000},
  MONTH = Dec,
  TYPE = {Theses},
  HAL_ID = {tel-00001289},
  HAL_VERSION = {v1},
}

@misc{Cekic-Lefeuvre-24,
 author = {Ceki{\'c}, Mihajlo and Lefeuvre, Thibault},
 title = {Semiclassical analysis on principal bundles},
 year = {2024},
 howpublished = {Preprint, {arXiv}:2405.14846 [math.{AP}] (2024)},
 url = {https://arxiv.org/abs/2405.14846},
 arXiv = {arXiv:2405.14846}
}

@article {Geng-Iyer-21,
    AUTHOR = {Geng, Xi and Iyer, Gautam},
     TITLE = {Long time asymptotics of heat kernels and {B}rownian winding
              numbers on manifolds with boundary},
   JOURNAL = {ALEA Lat. Am. J. Probab. Math. Stat.},
  FJOURNAL = {ALEA. Latin American Journal of Probability and Mathematical
              Statistics},
    VOLUME = {18},
      YEAR = {2021},
    NUMBER = {2},
     PAGES = {1297--1323},
      ISSN = {1980-0436},
   MRCLASS = {58J35 (58J65 60F05)},
  MRNUMBER = {4282190},
       DOI = {10.30757/alea.v18-48},
       URL = {https://doi.org/10.30757/alea.v18-48},
}

@article {Helffer-Kordyukov-09,
    AUTHOR = {Helffer, B. and Kordyukov, Y. A.},
     TITLE = {Spectral gaps for periodic {S}chr\"odinger operators with
              hypersurface magnetic wells: analysis near the bottom},
   JOURNAL = {J. Funct. Anal.},
  FJOURNAL = {Journal of Functional Analysis},
    VOLUME = {257},
      YEAR = {2009},
    NUMBER = {10},
     PAGES = {3043--3081},
      ISSN = {0022-1236,1096-0783},
   MRCLASS = {58J50 (35P05 35R01 81Q20)},
  MRNUMBER = {2568685},
MRREVIEWER = {Enrique\ G.\ Reyes},
       DOI = {10.1016/j.jfa.2009.04.007},
       URL = {https://doi.org/10.1016/j.jfa.2009.04.007},
}

@article {Kha-18,
    AUTHOR = {Kha, Minh},
     TITLE = {Green's function asymptotics of periodic elliptic operators on
              abelian coverings of compact manifolds},
   JOURNAL = {J. Funct. Anal.},
  FJOURNAL = {Journal of Functional Analysis},
    VOLUME = {274},
      YEAR = {2018},
    NUMBER = {2},
     PAGES = {341--387},
      ISSN = {0022-1236,1096-0783},
   MRCLASS = {58J37 (35J08 35J10 35J15 35P05 58J05 58J50)},
  MRNUMBER = {3724142},
MRREVIEWER = {L\'eonard\ Todjihounde},
       DOI = {10.1016/j.jfa.2017.10.016},
       URL = {https://doi.org/10.1016/j.jfa.2017.10.016},
}

@article {Kotani-Sunada-00,
    AUTHOR = {Kotani, Motoko and Sunada, Toshikazu},
     TITLE = {Albanese maps and off diagonal long time asymptotics for the
              heat kernel},
   JOURNAL = {Comm. Math. Phys.},
  FJOURNAL = {Communications in Mathematical Physics},
    VOLUME = {209},
      YEAR = {2000},
    NUMBER = {3},
     PAGES = {633--670},
      ISSN = {0010-3616,1432-0916},
   MRCLASS = {58J37 (58J35 58J65)},
  MRNUMBER = {1743611},
MRREVIEWER = {Ivan\ G.\ Avramidi},
       DOI = {10.1007/s002200050033},
       URL = {https://doi.org/10.1007/s002200050033},
}

@article {Lott-99,
    AUTHOR = {Lott, John},
     TITLE = {Remark about heat diffusion on periodic spaces},
   JOURNAL = {Proc. Amer. Math. Soc.},
  FJOURNAL = {Proceedings of the American Mathematical Society},
    VOLUME = {127},
      YEAR = {1999},
    NUMBER = {4},
     PAGES = {1243--1249},
      ISSN = {0002-9939,1088-6826},
   MRCLASS = {58G11 (58G18)},
  MRNUMBER = {1476376},
MRREVIEWER = {Friedbert\ Pr\"ufer},
       DOI = {10.1090/S0002-9939-99-04685-7},
       URL = {https://doi.org/10.1090/S0002-9939-99-04685-7},
}

\end{document}